\documentclass{article}
\usepackage[letterpaper,top=0.25in,bottom=0.15in,left=1in,right=1in,includeheadfoot]{geometry}
\usepackage{amssymb,amsmath,amsthm}
\usepackage{mathtools}
\usepackage{thmtools}
\usepackage{thm-restate}
\usepackage{hyperref} 
\usepackage{enumerate}
\usepackage{graphicx}
\graphicspath{ {./images/} }
\usepackage{xcolor}
\theoremstyle{definition}

\newcommand{\R}{\mathbb{R}}
\newcommand{\Z}{\mathbb{Z}}
\newcommand{\Q}{\mathbb{Q}}

\newcommand{\0}{\emptyset}

\newtheorem{theorem}{Theorem}[section]

\newtheorem{claim}[theorem]{Claim}

\newtheorem{corollary}[theorem]{Corollary}
\newtheorem{definition}[theorem]{Definition}
\newtheorem{example}[theorem]{Example}
\newtheorem{lemma}[theorem]{Lemma}

\newtheorem{proposition}[theorem]{Proposition}

\newtheorem{remark}[theorem]{Remark}

\newcommand{\cork}{\text{cork}}
\newcommand{\diam}{\text{diam}}
\newcommand{\Homeo}{\text{Homeo}}
\newcommand{\Maps}{\text{Maps}}

\newcommand{\MS}{\text{MaxShell}}
\newcommand{\Out}{\text{Out}}
\newcommand{\PHE}{\text{PHE}}
\newcommand{\rk}{\text{rk}}

\usepackage{epigraph}
\usepackage{etoolbox}
\makeatletter
\patchcmd{\epigraph}{\@epitext{#1}}{\itshape\@epitext{#1}}{}{}
\makeatother

\title{Dense Conjugacy Classes and Stability of Locally Finite Graphs}
\author{Rachmiel Klein}
\date{December 2025} 

\begin{document}

\maketitle

\begin{abstract}
    Mapping class groups of locally finite graphs are the analogue of those of infinite-type surfaces, and serve as a ``big'' version of $\Out(F_n)$. In this paper, we investigate which of these mapping class groups have a dense conjugacy class. We obtain a complete classification for self-similar locally finite graphs, and show that a large class of mapping class groups do not have a dense conjugacy class. One of the main tools we develop is flux homomorphisms, which we define for a broad class of locally finite graphs. Along the way, we develop a combinatorial notion for locally finite graphs, and we use it to provide a simple criterion for determining whether a locally finite graph is stable.
\end{abstract}

\section{Introduction}

The mapping class group of a locally finite graph $X$, denoted $\Maps(X)$, is the group of proper homotopy equivalences of $X$, up to proper homotopy. It was first introduced by Algom-Kfir and Bestvina in \cite{AK-B} with two main motivations: one coming from the study of outer automorphisms of free groups, and the other coming from the study of infinite-type surfaces. On the outer automorphism side, if $X$ is any finite graph, then $\Maps(X)\cong \Out(F_{\rk(\pi_1(X))})$, and hence, in general, $\Maps(X)$ serves as an infinite-type analogue of Out$(F_n)$. On the surface side, there is a natural correspondence between infinite-type surfaces and graphs up to proper homotopy which may be seen by taking the boundary of a regular neighborhood of $X$. For example, a torus corresponds to a circle and a bi-infinite cylinder corresponds to a line. In general, $\Maps(X)$ has a natural Polish topology with many nice properties. See the work of Algom-Kfir and Bestvina \cite{AK-B}, Domat, Hoganson, and Kwak \cite{DHK}, and Hill, Kopreski, Rechkin, Shaji, and Udall \cite{HKRSU} for some examples.

A topological group with a dense conjugacy class has the \emph{Rokhlin property}. Examples include the group of homeomorphisms of Cantor space and that of the Hilbert Cube \cite{GW}, as well as the symmetric group of the natural numbers \cite{M}. For more examples, history, and context, see Kechris and Rosendal \cite{KR}. A complete classification of which mapping class groups of connected orientable surfaces have a dense conjugacy class was obtained by Lanier and Vlamis \cite{LV} and Hern\'andez Hern\'andez, Hru\v{s}\'ak, Morales, Randecker, Sedano, and Valdez \cite{HHMRSV}: the mapping class group of a connected orientable $2$-manifold has the Rokhlin property if and only if the manifold is either the $2$-sphere or a non-compact manifold whose genus is either zero or infinite and whose end space is self-similar with a unique maximal end.

Given the duality between graphs and surfaces, it is natural to ask which mapping class groups of locally finite graphs have a dense conjugacy class. The classification obtained for surfaces does not hold as stated for the setting of locally finite graphs; there are many graphs and corresponding surfaces whose mapping class groups either both contain dense conjugacy classes or both do not, but there are also many graphs for which this does not hold. In this paper, we take the first step towards classifying graphs $X$ such that $\Maps(X)$ has a dense conjugacy class. We obtain a complete classification among graphs with self-similar end spaces; for these graphs, the result is similar to the surface setting.

\begin{restatable}{theorem}{mRr}\label{thm: main Rokhlin result}
    Let $X$ be a self-similar locally finite graph. The mapping class group $\Maps(X)$ has a dense conjugacy class if and only if $X$ is proper homotopy equivalent to the Cantor tree, or the genus of $X$ is either zero or infinite and $X$ has a unique maximal end. See Figure~\ref{fig: self-similar flowchart}.
\end{restatable}

\begin{figure}
    \centering
    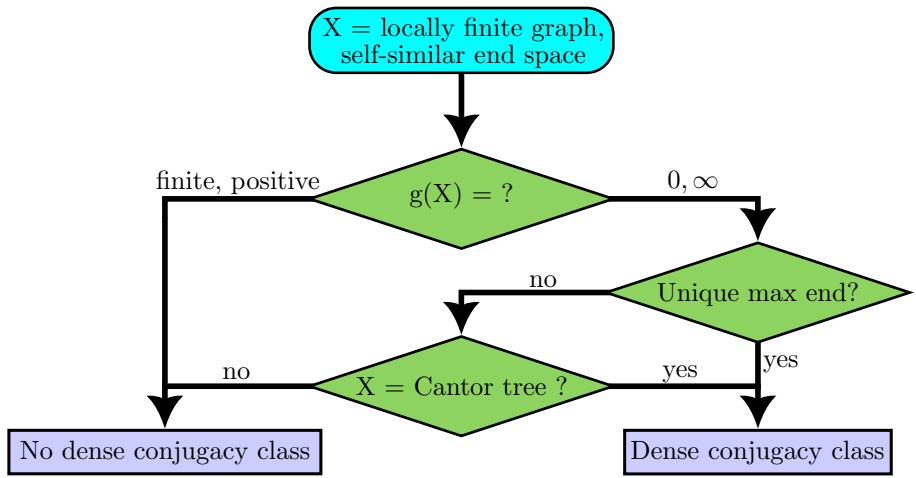
    \caption{Flowchart for self-similar locally finite graphs.}
    \label{fig: self-similar flowchart}
\end{figure}

When $X$ is not self-similar, the situation for locally finite graphs becomes more complicated than for surfaces, and a full classification is currently out of reach. For example, one obstruction to having a dense conjugacy class in the setting of surfaces is the existence of a non-displaceable subsurface. There is no meaningful analogue of a non-displaceable subsurface in the setting of locally finite graphs. Other obstructions, such as containing a subgroup that is proper, open, and normal, do apply to this setting. We show that such subgroups exist in a broad variety of settings by constructing non-trivial group homomorphisms to discrete groups including $\Z$, $\Out(F_n)$, and the symmetric group on $n$ letters. The homomorphisms from $\Maps(X)$ to $\Z$, which we call \emph{flux maps}, generalize the flux maps defined by \cite{DHK}, which in turn are analogous to the flux maps of \cite[Lemma~6.7]{MRlong}.


The flux maps of \cite{DHK} count the number of loops passing from one end to another and are defined on the subgroup of $\Maps(X)$ consisting of mapping classes which induce the identity on the end space of $X$. We extend the flux map construction to apply to all of $\Maps(X)$ and to count ends in addition to loops. Intuitively, the ends we count must satisfy a maximality condition, which we call being a \emph{greatest common divisor}, or \emph{gcd}, to ensure that the flux maps will be finite-valued. See Definitions~\ref{def: gcd} and~\ref{def: flux map for ends} for the precise definitions of gcds and flux maps respectively, and see Section~\ref{sec: order on ends} for a discussion of end types and the order on ends.

\begin{theorem}[Obstructions to dense conjugacy classes]\label{thm: obstructions to dense conjugacy classes}
    Let $X$ be a locally finite graph such that at least one of the following holds.
    \begin{enumerate}[(1)]
        \item All maximal ends are stable and there exist two maximal ends with a gcd which is not of Cantor type.
        \item There exists an end type $E(\mu)$ with $1<|E(\mu)|<\infty$.
        \item The genus of $X$ is finite.
    \end{enumerate}
    Then $\Maps(X)$ does not have a dense conjugacy class.
\end{theorem}

The methods we employ to prove that Theorem~\ref{thm: obstructions to dense conjugacy classes} (1) implies the conclusion generalize to a much less restrictive class of locally finite graphs $X$, as well as many subgroups of $\Maps(X)$; see Theorem~\ref{thm: flux map general} for the general statement. In particular, the following corollary follows immediately from Theorem~\ref{thm: flux map general}.

\begin{restatable}{corollary}{PMaps}\label{cor: PMaps}
    Let $X$ be a locally finite graph with more than one end accumulated by genus. Then P$\Maps(X)$ does not contain a dense conjugacy class.
\end{restatable}

Another group naturally associated to a locally finite graph $X$ is $\Homeo(\partial X,\partial X_g)$, the group of homeomorphisms of the end space of $X$. In Section~\ref{sec: homeo end space}, we give conditions under which $\Homeo(\partial X,\partial X_g)$ does and does not have a dense conjugacy class.

Many mapping class groups of locally finite graphs have a dense conjugacy class if and only if the groups of their corresponding surfaces do, but a full classification for graphs must have different hypotheses. We provide an uncountable class of mapping class groups that have a dense conjugacy class but whose surface counterparts do not.

\begin{restatable}{theorem}{Tdp}\label{thm: trees and direct products}
    Let $X$ be a locally finite tree with a countable end space and a unique maximal end. If $C$ is a locally finite tree whose end space is homeomorphic to Cantor space, then $\Maps(X\vee C)$ has a dense conjugacy class.
\end{restatable}

Any graph $X\vee C$ as in Theorem~\ref{thm: trees and direct products} is an example of a graph such that the mapping class group of the corresponding surface does not have a dense conjugacy class. This is an uncountable class of graphs, because it follows from \cite{MS} and the constructions in Section~\ref{sec: signatures} that such graphs are naturally indexed by countable ordinal numbers.

To facilitate our investigation of dense conjugacy classes, we develop a combinatorial notation for locally finite graphs based on their end space, called a \emph{signature}. The definition is technical, so we instead refer to Figure~\ref{fig: examples of graphs} for examples and refer readers to Definition~\ref{def: signature}. We then develop the stronger notion of an \emph{ordered signature}, which explicitly encodes the order on ends and is equivalent to the graph being stable. Stability (Definition~\ref{def: stability}) is a natural condition on the end space of a locally finite graph that ensures that it is ``well-behaved.'' First introduced by Mann and Rafi \cite{MRlong}, many papers have restricted to the class of stable locally finite graphs or surfaces, including works of Fanoni, Ghaswala, and McLeay \cite{FGM} and works of Bar-Natan and Verberne \cite{BV}, along with many others.

\begin{figure}
    \centering
    \includegraphics[width=1\linewidth]{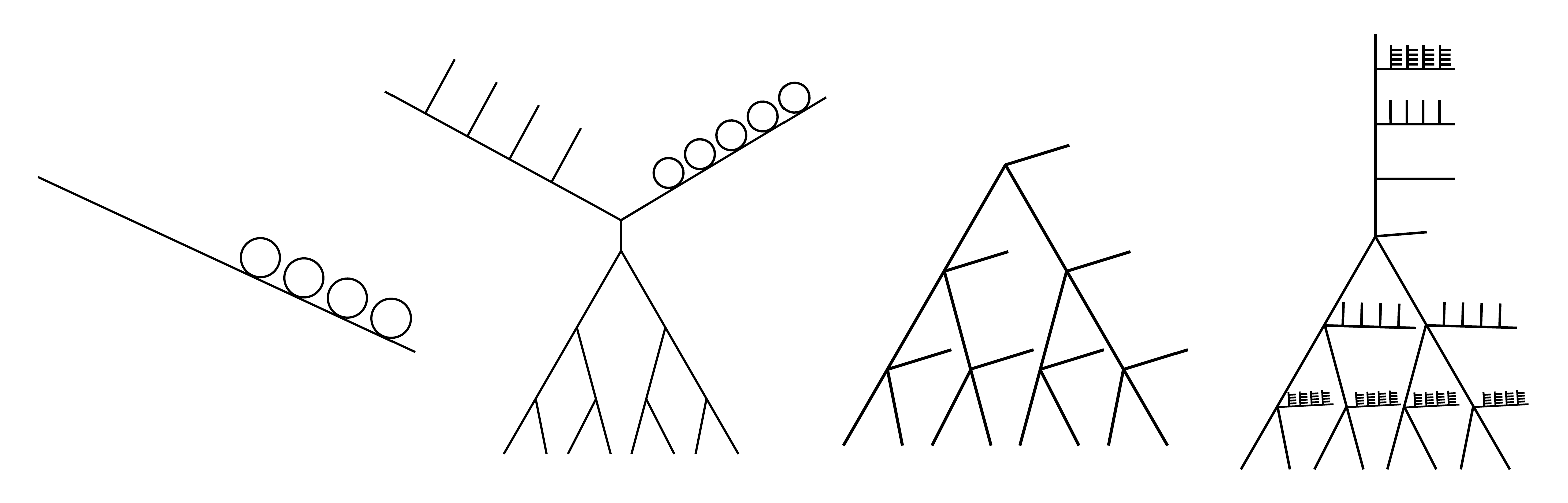}
    \caption{The signatures of the graphs depicted, from left to right, are $1\vee o(1)$, $(\omega+1)\vee C\vee o(1)$, $1\to C$, and $\{\omega^n+1\}_n\to(1\vee C)$. In signatures, the symbols $1$ and $C$ represent a ray and a Cantor tree respectively. Arrows roughly represent convergence towards a greater end, and $o(-)$ represents an end accumulated by genus. Throughout the paper, all ends of line segments in figures are assumed to be rays; we often drop arrows for neater drawings.}
    \label{fig: examples of graphs}
\end{figure}

\begin{theorem}\label{thm: ordered signatures intro}
    A locally finite graph is stable if and only if it has an ordered signature.
\end{theorem}

While ordered signatures are the most structured signatures, there is another condition on signatures that in practice is easier to verify and is equivalent to stability as well; see Theorem~\ref{thm: ordered signatures}. This allows us to verify that complex graphs, such as one with signature
$$\{\{o(\omega^n+1)\}_n\to \bigvee_{i=1}^m((1\to C)\to(\omega^{\omega^i}+1))\}_m\to (1\vee C),$$
are stable. The arrow notation in a signature was created to mimic convergence towards a greater end, and Theorem~\ref{thm: ordered signatures intro} intuitively states that a locally finite graph is stable if and only if it has a signature where every arrow captures this motivation.

Returning to dense conjugacy classes, a natural question is whether having a dense conjugacy class is inherited by subgroups or super-groups. Even among stable locally finite graphs, this is not the case:

\begin{restatable}{theorem}{Rn}\label{thm: Rokhlin nesting}
    Let $X$ be a locally finite graph. There exist two self-similar locally finite graphs $X_Y$ and $X_N$ such that there exists an embedding of $\Maps(X)$ as a closed subgroup into both $\Maps(X_Y)$ and $\Maps(X_N)$, with $\Maps(X_Y)$ having a dense conjugacy class and $\Maps(X_N)$ not having a dense conjugacy class. Moreover, this embedding is induced by an embedding of end spaces.
\end{restatable}

The locally finite graphs for which it is still unknown whether the mapping class has a dense conjugacy class or not roughly fall into three categories. The first category is comprised of graphs which are essentially too simple to have any predictive structure, such as $1\vee o(1)$ or $(\omega+1)\vee C\vee o(1)$. In these graphs, there do not exist ends (or loops) which are dominated by two maximal ends. The second category includes graphs such as $\{\omega^n+1\}_n\to(1\vee C)$. These are graphs where for any pair of maximal end types, there is an increasing sequence of end types dominated by those maximal end types. The third category are unstable graphs where the methods involving stable maximal ends do not apply. This includes the graphs constructed in \cite{MRshort} and the graph in Example~\ref{ex: keystone lfg}.

\subsection*{Outline of the paper}

In Section~\ref{sec: background}, we discuss the necessary background material for the paper. We start with properties of locally finite graphs and $\Maps(X)$ in Section~\ref{sec: lfgs and phes}, continue to ordinal numbers in Section~\ref{sec: ordinals}, and conclude with the order on ends in Section~\ref{sec: order on ends}. Section~\ref{sec: notation/lemmas} presents the notation of signatures. Section~\ref{sec: structure analysis} is about local structures (see Definition~\ref{def: local structure}), and their basic properties. Section~\ref{sec: classification of local structures} is a classification of local structures, Section~\ref{sec: poset of lfgs} describes the poset of local structures, and Section~\ref{sec: wedge decomposition} studies wedge decompositions, an important tool which will be used for the rest of the paper. In Section~\ref{sec: ordered signatures}, we define ordered signatures and prove Theorem~\ref{thm: ordered signatures intro}. In Section~\ref{sec: Rokhlin}, we turn our attention to the study of dense conjugacy classes in $\Maps(X)$. In Section~\ref{sec: finite genus and maximal end type}, we prove Theorem~\ref{thm: obstructions to dense conjugacy classes} (2) and (3), in Section~\ref{sec: Cantor type} we prove Theorems~\ref{thm: main Rokhlin result} and~\ref{thm: Rokhlin nesting}, in Section~\ref{sec: flux maps} we prove Theorem~\ref{thm: obstructions to dense conjugacy classes} (1) and Corollary~\ref{cor: PMaps}, and in Section~\ref{sec: trees and direct products} we prove Theorem~\ref{thm: trees and direct products}. In Section~\ref{sec: homeo end space}, we investigate $\Homeo(\partial X,\partial X_g)$.

\subsection*{Acknowledgments}

Thank you to the author's advisor Carolyn R. Abbott for her help and support throughout the project. Thank you to Ferr\'{a}n Valdez for suggesting the problem and for his advice, and thank you to Nicholas G. Vlamis for many helpful conversations and pointing the author in the direction of flux maps.  Thank you to George Domat for many helpful brainstorming sessions, Robbie Lyman for help visualizing intricate proper homotopy equivalences of graphs, and Kasra Rafi and Assaf Bar-Natan for conversations about complex end spaces. Thank you to Brian Udall for pointing out multiple errors in initial drafts of Proposition~\ref{prop: flux map for ends} and for his suggestions, and thank you to Arya Vadnere for help in proving $3\Rightarrow 4$ in Proposition~\ref{prop: classification of local structures}. The author was partially supported by NSF grants DMS-2106906 and DMS-2340341.

\section{Background}\label{sec: background}

\subsection{Locally finite graphs}\label{sec: lfgs and phes}

We give an overview of the definitions and  properties of locally finite graphs necessary for this paper. For further details see \cite{AK-B} and \cite{MRlong}. All graphs $G$ are assumed to be locally finite and connected, and $\mathcal{V}(G)$ denotes the set of vertices of $G$.

\begin{definition}
    The \emph{end space} of a locally finite graph $X$ is defined to be $\partial X:=\varprojlim_{K\subset X}\pi_0(X\setminus K)$ with the inverse limit topology, where the limit runs over all compact subsets $K\subset X$.
\end{definition}

The space $\partial X$ is totally disconnected, compact, and metrizable, and thus is a closed subset of Cantor space. There is a natural compact topology on $X\cup\partial X$, sometimes referred to as the end/Freudenthal compactification.

We next extend the notion of genus from the surface setting into the setting of locally finite graphs. With surfaces, genus intuitively counts how many holes there are. With graphs, genus intuitively counts how many loops there are. Recall that the fundamental group of a graph is free.

\begin{definition}
    The \emph{genus} $g(X)\in\Z_{\geq0}\cup\{\infty\}$ of $X$ is the rank of $\pi_1(X,x_0)$ for some $x_0\in X$. The \emph{core graph} $X_g\subseteq X$ is the smallest sub-graph containing all immersed loops.
\end{definition}

The end space $\partial X_g\subseteq\partial X$ is the subspace of ends ``accumulated by genus,'' or such that every neighborhood in $X\cup\partial X$ of the end has positive genus. The set $\partial X_g$ is always closed in $\partial X$ and is the empty set precisely when $g(X)<\infty$ \cite[Section~2]{AK-B}. A locally finite graph is \emph{finite-type} if $|\partial X|$ and $g(X)$ are both finite, and is \emph{infinite-type} otherwise.

There are several natural ways to define the mapping class group of a locally finite graph. One possibility is to use the same definition as in surfaces, but in this case, the resulting group would be quite restrictive. For example, this group is finite for finite graphs. On the other hand, the group of homotopy equivalences of a locally finite graph up to homotopy recovers the motivation coming from the theory of Out$(F_n)$. In fact, for any locally finite graph $X$, the group of homotopy equivalences of $X$ modulo homotopy is isomorphic to Out$(F_n)$, where $n=g(X)$. Howeaver, these groups do not give information about the topology of the space of ends when the graph is infinite. This naturally leads us to the following definition. Recall that a continuous map is \emph{proper} if the inverse image of all compact sets are compact.

\begin{definition}
    Two proper maps $f$ and $g$ from $X$ to $Y$ are \emph{properly homotopic} if there exists a map $H\colon X\times[0,1]\rightarrow Y$ such that $H(-,0)=f$, $H(-,1)=g$, and $H(-,t)\colon X\rightarrow Y$ is proper for all $t\in[0,1]$.
\end{definition}

\begin{definition}
    A \emph{proper homotopy equivalence} between two locally finite graphs $X$ and $Y$ is a proper map $f \colon X\rightarrow Y$ such that there exists another proper map $g \colon Y\rightarrow X$ such that $f\circ g$ and $g\circ f$ are properly homotopic to the identity. Let $\PHE(X)$ be the set of proper homotopy equivalences from $X$ to itself.
\end{definition}

Proper homotopy equivalences of graphs are analogous to homeomorphisms of surfaces, and the following is the analogue of mapping class groups of surfaces.

\begin{definition}[\cite{AK-B}]
    The \emph{mapping class group of $X$}, denoted $\Maps(X)$, is the quotient $\PHE(X)/\sim$, where $f\sim g$ if $f$ and $g$ are properly homotopic. We refer to elements $[f]$ of $\Maps(X)$ as \emph{mapping classes}.
\end{definition}

We further endow $\Maps(X)$ with the topology constructed in \cite{AK-B}, which we briefly describe. Let $K$ be a finite subgraph of $X$. Let $U_K$ be the set of mapping classes $[f]$ with a representative $f$ such that the following four conditions hold.
\begin{enumerate}
    \item $f|_K=id$,
    \item $f$ preserves each component of $X\setminus K$,
    \item there is a representative $g$ of $[f]^{-1}$ which satisfies the first two conditions, and
    \item there are proper homotopies $g\circ f\simeq id$ and $f\circ g\simeq id$ that are stationary on $K$ and preserve each component of $X\setminus K$. 
\end{enumerate}
The topology on $\Maps(X)$ is the coarsest topology such that all subsets of the form $U_K$ are open and such that multiplication and inversion are continuous. The following two propositions summarize results in \cite{AK-B} which are relevant to our discussion.

\begin{proposition}[\cite{AK-B}]\label{AK-B 4.7}
    Each $U_K$ is a clopen subgroup, and for every neighborhood $U$ of the identity, there exists $K$ such that $U_K\subseteq U$.
\end{proposition}

\begin{proposition}[\cite{AK-B}]\label{AK-B Polish}
    The topological group $\Maps(X)$ is Polish. If $\Maps(X)$ is of finite type, then it is isomorphic as a topological group to Out$(F_{g(X)})$ equipped with the discrete topology. Otherwise, it is homeomorphic to the set of irrational numbers.
\end{proposition}

There exists a classification of locally finite graphs, analogous to the classification for infinite-type surfaces.

\begin{definition}
    The \emph{characteristic pair} of a locally finite graph $X$ is $(\partial X,g(X))$ if $g(X)<\infty$, and $(\partial X,\partial X_g)$ otherwise. Two characteristic pairs $(\partial X,\partial X_g)$ and $(\partial X',\partial X'_g)$ are homeomorphic if there exists a homeomorphism from $\partial X$ to $\partial X'$ inducing a homeomorphism from $\partial X_g$ to $\partial X'_g$. Similarly, characteristic pairs $(\partial X,n)$ and $(\partial X',m)$ are homeomorphic if $\partial X\cong\partial X'$ and $n=m$.
\end{definition}

We will frequently abuse notation in the following way. If $U,V\subseteq\partial X$, then $U_g:=U\cap\partial X_g$. By $U\cong V$ we mean $(U,U_g)\cong(V,V_g)$, and by $U\subseteq V$ we mean that $U\subseteq V$ and $U_g\subseteq V_g$. In addition, unless otherwise specified, the notation $U\subseteq V$ will assume that $U$ is clopen in $V$.

\begin{theorem}[\cite{ADMQ}]\label{AK-B 2.2}
    Two connected locally finite graphs are proper homotopy equivalent if and only if their characteristic pairs are homeomorphic.
\end{theorem}

Characteristic pairs, which determine a locally finite graph, can in turn be characterized using closed subsets of Cantor space.

\begin{proposition}[\cite{AK-B}]\label{AK-B 2}
    Given a closed subset $B$ of Cantor space, a closed subset $A$ of $B$, and given $n\in\Z_{\geq0}$, there exist connected locally finite graphs $X$ and $Y$ such that $(B,A)$ is the characteristic pair of $X$ and $(B,n)$ is the characteristic pair of $Y$.
\end{proposition}

Not only are characteristic pairs invariant up to proper homotopy equivalence of a locally finite graph, but so is the mapping class group:

\begin{proposition}[{\cite[Corollary~4.5]{AK-B}}]\label{AK-B 4.5}
    Let $X$ and $Y$ be two connected locally finite graphs which are proper homotopy equivalent. Then $\Maps(X)$ and $\Maps(Y)$ are isomorphic as topological groups.
\end{proposition}

It is often convenient to work with a specific graph which is proper homotopy equivalent to the original graph, leading to the following definition:

\begin{definition}[{\cite[Section~2]{AK-B}}]
    Let $X$ be a connected locally finite graph. We say that $X$ is in \emph{standard form} if either 
    \begin{enumerate}
        \item $X$ has no vertices of valence one and $X$ is formed by attaching edges to a tree (the \emph{underlying tree of $X$}) such that both ends of each attached edge are attached to the same vertex, or
        \item $X$ is homeomorphic to a ray.
    \end{enumerate}
\end{definition}

\begin{proposition}[{\cite[Corollary~2.6]{AK-B}}]\label{AK-B 2.6}
    Every locally finite connected infinite graph is proper homotopy equivalent to a graph in standard form.
\end{proposition}

A mapping class $[f]\in\Maps(X)$ induces a map on $\partial X$. We abuse notation and refer to this induced map by $[f]$. In addition, the induced map on ends preserves $\partial X_g$, so mapping classes induce homeomorphisms on characteristic pairs.

\begin{proposition}[{\cite{AK-B}}]\label{AK-B 2.3} 
    There exists a surjective, continuous, and open group homomorphism 
    $$\sigma\colon\Maps(X)\to\Homeo(\partial X,\partial X_g)$$
    sending a mapping class to the induced map on the characteristic pair. If $X$ is a tree, then $\sigma$ is an isomorphism of topological groups.
\end{proposition}

The map $\sigma$ is analogous to the group homomorphism from $\Maps(X)$ to Out$(\pi_1(X))$ which sends a mapping class to its induced map on the level of fundamental groups, but unlike $\sigma$, this map is not surjective when $g(X)$ is infinite \cite[Example~4.1]{AK-B}.

\subsection{Ordinal numbers}\label{sec: ordinals}

Ordinal numbers throughout this paper are endowed with the order topology: inductively, $0$ is the empty topological space, if $\alpha$ is an ordinal then $\alpha+1$ is the one-point compactification of $\alpha$, and limit ordinals are endowed with the inductive limit topology. 

\begin{example}
    Some examples of ordinals as topological spaces include:
    \begin{itemize}
        \item The ordinal $n\in\Z_{\geq0}$ represents a topological space with $n$ discrete points, because the one point compactification of a compact space $A$ is $A$ with an added isolated point.
        \item The first infinite ordinal $\omega$ is given the topology of the natural numbers, and $\omega+1$ is given the topology of the set $\{\frac{1}{n}\,:\,n\in\Z_{>0}\}\cup\{0\}$.
        \item The ordinal $\omega^2+1$ is homeomorphic to $\left\{(\frac{1}{ni},\frac{1}{i})\,:\,n,i\in\Z_{>0}\right\}\cup\left\{(0,\frac{1}{i})\,:\,i\in\Z_{>0}\right\}\cup\{(0,0)\}$ in $\R^2$.
    \end{itemize}
\end{example}

The following definition concerns any topological space $\mathcal{X}$. In this paper, $\mathcal{X}$ will typically be (a subset of) the end space of a locally finite graph.

\begin{definition}
    For any topological space $\mathcal{X}$, the derived set of $\mathcal{X}$, denoted $\mathcal{X}'$, consists of the set of accumulation points of $\mathcal{X}$. Given an ordinal $\alpha$, the \emph{$\alpha^{\text{th}}$ Cantor-Bendixon derivative} of $\mathcal{X}$, denoted $\mathcal{X}^\alpha$, is defined using transfinite induction on ordinal numbers as follows.
    \begin{enumerate}
        \item $\mathcal{X}^0=\mathcal{X}$.
        \item For successor ordinals, $\mathcal{X}^{\alpha+1}=(\mathcal{X}^\alpha)'$.
        \item If $\lambda$ is a limit ordinal, then $\mathcal{X}^\lambda=\bigcap_{\alpha<\lambda}\mathcal{X}^\alpha$.
    \end{enumerate}
    The \emph{Cantor-Bendixon rank} of a point $x\in\mathcal{X}$ is the minimal ordinal $\alpha$ such that $x\in \mathcal{X}^\alpha$ but $x\not\in \mathcal{X}^{\alpha+1}$, if such an ordinal exists.
\end{definition}

We will often cite the following theorem due to Mazurkiewicz and Sierpi\'{n}ski.

\begin{theorem}[\cite{MS}]\label{MS 1}
    Let $P$ be a closed, bounded, and countably infinite subset of Euclidean space of any dimension. Then $P$ is homeomorphic to a well-ordered set. Moreover, if $\alpha$ is the minimal ordinal such that the cardinality of $P^\alpha$ is finite, then $P$ is homeomorphic to $\omega^\alpha\cdot n+1$, where $|P^\alpha|=n$.
\end{theorem}

If $X$ is any locally finite graph, then the end space $\partial X$ is a closed and bounded subset of Cantor space. Theorem~\ref{MS 1} implies that if $\partial X$ is countable, it is homeomorphic to $\omega^\alpha\cdot n+1$ for some $n\in\Z_{\geq 0}$ and countable ordinal $\alpha$.

\begin{remark}\label{rmk: one}
    If $P$ is a set of finite cardinality $n$, then $P$ is homeomorphic to the ordinal $n$. Howeaver, if we tried to apply Theorem~\ref{MS 1} to $P$, we would get that $P$ is homeomorphic to $\omega^0\cdot n+1=n+1$. Given this, we abuse notation throughout the paper by setting $\omega^0+1$ equal to $1$.
\end{remark}

\subsection{Order on ends}\label{sec: order on ends}

Ends in a locally finite graph can be endowed with a preorder \cite{MRlong}. We state the definition and then give the relevant background.

\begin{definition}[Preorder on ends]\label{def: order on ends}
    Given $\nu,\nu'\in\partial X$, we say that $\nu\preceq\nu'$ if every neighborhood of $\nu'$ contains a homeomorphic copy of a neighborhood of $\nu$. We say $\nu$ and $\nu'$ are \emph{equivalent}, or that they are of the same \emph{type}, and write $\nu\sim\nu'$, if $\nu\preceq\nu'$ and $\nu'\preceq\nu$. We also write $\nu\prec\nu'$ if $\nu\preceq\nu'$ and $\nu\not\sim\nu'$. We define $E(\nu)=\{\nu'\in\partial X\,:\,\nu'\sim\nu\}$ to be the \emph{end type of $\nu$}.
\end{definition}

The relation $\sim$ is an equivalence relation, and thus, the preorder $\prec$ descends to a partial order on the space of end types. We often conflate the preorder with the partial order. See Figure~\ref{fig: order on ends}.

\begin{definition}
    We say an end $\nu\in\partial X$ is of \emph{Cantor type} if $|E(\nu)\cap U|>1$ for every neighborhood $U$ of $\nu$.
\end{definition}

The following results concerning the order on ends were proven by Mann and Rafi in the setting of infinite-type surfaces, but the proofs are the same in the setting of locally finite graphs.

\begin{proposition}[\cite{MRlong}]\label{MR 4.7}
    The partial order $\prec$ has maximal elements. Furthermore, for every maximal end $\mu$, the equivalence class $E(\mu)$ is either finite or a Cantor set.
\end{proposition}

\begin{theorem}[\cite{MRshort}]\label{MR 1.2}
    Let $\nu\sim\nu'$ be two ends of the same type in a locally finite graph $X$. There exists an element of $\Maps(X)$ taking $\nu$ to $\nu'$.
\end{theorem}

Theorem~\ref{MR 1.2} shows that for any end $\nu\in\partial X$, the orbit $\Maps(X)\cdot\nu$ of $\nu$ contains $E(\nu)$. In Lemma~\ref{lem: caste system}, we show the other inclusion, so that end types are exactly $\Maps(X)$-orbits. Note that Theorem~\ref{MR 1.2} also shows that if $\nu_1\sim\nu_2$, then $\nu_1$ is of Cantor type if and only if $\nu_2$ is.

\begin{figure}[h]
    \centering
\begingroup%
  \makeatletter%
  \providecommand\color[2][]{%
    \errmessage{(Inkscape) Color is used for the text in Inkscape, but the package 'color.sty' is not loaded}%
    \renewcommand\color[2][]{}%
  }%
  \providecommand\transparent[1]{%
    \errmessage{(Inkscape) Transparency is used (non-zero) for the text in Inkscape, but the package 'transparent.sty' is not loaded}%
    \renewcommand\transparent[1]{}%
  }%
  \providecommand\rotatebox[2]{#2}%
  \newcommand*\fsize{\dimexpr\f@size pt\relax}%
  \newcommand*\lineheight[1]{\fontsize{\fsize}{#1\fsize}\selectfont}%
  \ifx\svgwidth\undefined%
    \setlength{\unitlength}{292.26967639bp}%
    \ifx\svgscale\undefined%
      \relax%
    \else%
      \setlength{\unitlength}{\unitlength * \real{\svgscale}}%
    \fi%
  \else%
    \setlength{\unitlength}{\svgwidth}%
  \fi%
  \global\let\svgwidth\undefined%
  \global\let\svgscale\undefined%
  \makeatother%
  \begin{picture}(1,0.29782756)%
    \lineheight{1}%
    \setlength\tabcolsep{0pt}%
    \put(0,0){\includegraphics[width=\unitlength,page=1]{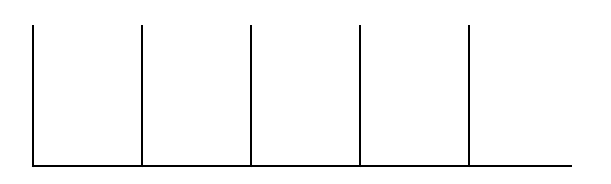}}%
    \put(0.06428633,0.26403225){\color[rgb]{0,0,0}\makebox(0,0)[t]{\smash{\begin{tabular}[t]{c}$\nu_0$\end{tabular}}}}%
    \put(0.24090257,0.26377694){\color[rgb]{0,0,0}\makebox(0,0)[t]{\smash{\begin{tabular}[t]{c}$\nu_1$\end{tabular}}}}%
    \put(0.42082828,0.26377458){\color[rgb]{0,0,0}\makebox(0,0)[t]{\smash{\begin{tabular}[t]{c}$\nu_2$\end{tabular}}}}%
    \put(0.59943279,0.26364741){\color[rgb]{0,0,0}\makebox(0,0)[t]{\smash{\begin{tabular}[t]{c}$\nu_3$    \end{tabular}}}}%
    \put(0.7782762,0.26373329){\color[rgb]{0,0,0}\makebox(0,0)[t]{\smash{\begin{tabular}[t]{c}$\nu_4$\end{tabular}}}}%
    \put(0.95957221,0.02261054){\color[rgb]{0,0,0}\makebox(0,0)[t]{\smash{\begin{tabular}[t]{c}$\mu$\end{tabular}}}}%
  \end{picture}%
\endgroup%

    \caption{For all $i$ and $j$, $\nu_i\sim\nu_j$, and $\nu_i\prec\mu$. $\mu$ is the unique maximal end.}
    \label{fig: order on ends}
\end{figure}
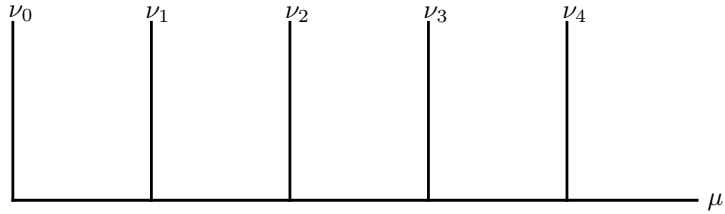

We next define stability, which is a natural condition that rules out many pathological locally finite graphs, defined in \cite[Section~4]{MRlong}.

\begin{definition}\label{def: stability}
    Let $\nu\in\partial X$. A neighborhood $U\subseteq\partial X$ of $\nu$ is \emph{stable} if for any smaller neighborhood $U'\subseteq U$, there exists a homeomorphic copy of $U$ in $U'$. An end is \emph{stable} if it has a stable neighborhood, and a locally finite graph is \emph{stable} if every end is stable. 
\end{definition}

\begin{example}\label{ex: stability}
    The following are examples of stability.
    \begin{enumerate}
        \item If $X$ is a tree with countable end space, then $X$ is stable \cite[Remark~5.5]{FGM}.
       \item The locally finite graph formed from a stable tree by attaching loops to each vertex of $X$ is stable.
       \item The Cantor tree is stable.
       \item Non-stable locally finite graphs are more complicated and it will be easier to give examples after Definition~\ref{def: signature}. Intuitively, if there is a sequence of ends $\{\nu_n\}_n$ converging to an end $\mu$, none of which are dominated by $\mu$ and which are all of different types, then $\mu$ must not be stable. For concrete examples, see Example~\ref{ex: keystone lfg} and the first bullet point under part three in Definition~\ref{def: ordered signature}. There are other ways to violate stability as well; the graphs in \cite[Theorem~1.1]{MRshort} are not stable and contain a unique maximal end type.
    \end{enumerate}
\end{example}

We next define the related notion of self-similarity. In Section~\ref{sec: classification of local structures}, we detail the relationship between self-similarity and stability.

\begin{definition}
    A locally finite graph $X$ or its end-space $\partial X$ is \emph{self-similar} if given any clopen decomposition of the end space $\partial X=E_1\sqcup\cdots\sqcup E_n$, there exists $i\in\{1,...,n\}$ such that $E_i$ contains a homeomorphic copy of $\partial X$.
\end{definition}

Mann and Rafi \cite[Section~4]{MRlong} show that, without loss of generality, one can take $n=2$. 

\begin{example} 
    The following are examples of self-similarity.
    \begin{enumerate}
        \item The Cantor tree is self-similar, because a small clopen neighborhood of a point in Cantor space is homeomorphic to Cantor space.
        \item The graph in Figure~\ref{fig: order on ends} is self-similar, as any clopen neighborhood of the maximal end is homeomorphic to the full end space.
        \item A line is not self-similar, for a decomposition separating the two ends does not satisfy the condition. A Cantor tree with one ray attached is not self-similar for similar reasons.
    \end{enumerate}
\end{example}



\section{Notation and preliminary lemmas}\label{sec: notation/lemmas}

\epigraph{``The foundational letters were engraved, permuted, transformed, and with them, it was depicted all that was formed and all that would be formed. Drawn in water, lit up in fire, rattled in the wind, they were created from profound chaos. From here on, go out and calculate that which the mouth cannot speak and the ear cannot hear.''\par\raggedleft$\sim$ Sefer Yetzirah\textup{, adapted} $\sim$}

In Section~\ref{sec: signatures}, we give the definition of signatures, which is presented as a list of symbols for certain fundamental graphs and operations from which all other graphs are pieced together. In Section~\ref{sec: building blocks}, we prove lemmas which will be used throughout the paper, culminating in showing that all locally finite graphs have signatures in Proposition~\ref{prop: building blocks}.

\subsection{Signatures}\label{sec: signatures}

The goal of this section is to develop a notation to be able to describe any locally finite graph $X$, which we call a \emph{signature}. The name ``Loch Ness Monster'' is commonly used to refer to the graph with one end accumulated by loops, and while useful as a nickname, this name does not suggest how this graph is constructed or what similarities it has to other graphs. From the rules of Definition~\ref{def: signature} below, the Loch Ness Monster graph has ``$o(1)$'' as a signature, because it is a ray, hence the ``$1$'', which is accumulated by loops, hence the ``$o(-)$''.

\begin{definition}\label{def: signature} We define \emph{signatures} inductively. Points $1$ through $3$ are the base cases, and points $4$ through $7$ are ways of building new signatures; each can be applied finitely many times. Recall that all graphs are assumed to be connected.
\begin{enumerate}
    \item (Rose) The signature of a graph proper homotopy equivalent to $\bigvee_{i=1}^n S^1$ is $R_n$. Observe that the signature of a graph consisting of one vertex is $R_0$.
    
    \item (Cantor Tree) The signature of a graph proper homotopy equivalent to a rooted infinite binary splitting tree is $C$.
    
    \item (Countable Ordinals) Any tree with countably many ends and a unique maximal end has end space homeomorphic to some ordinal of the form $\omega^\alpha+1$, where $\alpha$ is some countable ordinal, by Theorem~\ref{MS 1}. Given this, the signature of a tree with countably many ends and a unique maximal end is $\omega^\alpha+1$ where $\alpha$ is the corresponding ordinal. Given Remark~\ref{rmk: one}, a single ray is represented by $1$.
    
    \item (Genus) If $G'$ is a locally finite graph proper homotopy equivalent to the graph formed by attaching a copy of $S^1$ to each vertex of a locally finite graph $G$, if $G$ has signature $X$, then $G'$ has signature $o(X)$.
    
    \item (Wedge) Given two locally finite graphs $G$ and $G'$ with signatures $X$ and $Y$ respectively, the signature of a locally finite graph proper homotopy equivalent to the wedge product of $G$ and $G'$ is $X\vee Y$.
    
    \item (Convergence) If $\{G_n\}_{n\in\Z_{\geq0}}$ is a sequence of locally finite graphs with signatures $\{Y_n\}_{n\in\Z_{\geq0}}$ respectively, and if $G$ is another locally finite graph whose signature is given, then $\{Y_n\}_n\rightarrow (G,x_0)$ is a the signature of a locally finite graph proper homotopy equivalent to one constructed in the following way. Choose a base vertex $x_0\in\mathcal{V}(G)$, and at each vertex $v\in\mathcal{V}(G)$, attach a copy of $G_{d(x_0,v)}$ to $G$. If $G$ has signature $X$, then $\{Y_n\}_n\rightarrow X$ is a signature when the choice of $x_0\in\mathcal{V}(G)$ does not matter, and $Y\to X$ is a signature when all the $Y_n$ are equal.
    
    \item (Spread) Suppose $\{G_n\}_{n\in\Z_{\geq0}}$ is a sequence of locally finite graphs with signatures $\{Y_n\}_{n\in\Z_{\geq0}}$,  respectively, such that $Y_0\cong R_0$, and suppose $G$ is another locally finite graph whose signature is given. If $a \colon \mathcal{V}(G)\rightarrow\Z_{\geq0}$, then $(\{Y_n\}_n,a)\rightarrow G$ is the signature of a locally finite graph which is proper homotopy equivalent to one formed from $G$ by attaching a copy of $G_{a(v)}$ to each vertex $v\in\mathcal{V}(G)$. Note that since $Y_0\cong R_0$, $Y_0\vee G\cong G$, i.e., vertices of $G$ need not change valence under the attaching map $a$.
\end{enumerate}
\end{definition}

Convergence is a special case of spread in which the function $a$ is the distance function from a particular vertex. We chose to separate convergence from spread because, as we will see in Theorem~\ref{thm: ordered signatures}, spread is not needed in stable graphs. We discuss other redundancies in the notation in Section~\ref{sec: building blocks}.

There are two important subtleties in spread and hence convergence. First of all, note that these signatures are built relative to a particular locally finite graph. To avoid ambiguity arising from how many of each $Y_n$ graphs are attached to the underlying graph $G$, in general, it is necessary to be given the vertex set of $G$ before attaching the other graphs. In Theorem~\ref{thm: ordered signatures}, we show that when a graph is stable, there exists a signature involving only other signatures and not any particular locally finite graphs.

The second important subtlety is the requirement that the signature of the graph $G$ appearing to the right of an arrow is given. This inductive requirement is included in order to avoid degenerate signatures which give no information about the graphs they represent. For example, had this requirement not been included, every locally finite graph $G$ would have a signature $R_0\to(G,v)$ where $v\in\mathcal{V}(G)$; this gives no information on the structure of $G$. In Proposition~\ref{prop: building blocks}, we show that every locally finite graph has a signature, and thus, can be constructed from simpler graphs whose signatures are simpler as well.

By construction, signatures are collections of symbols associated to particular locally finite graphs which are invariant under proper homotopy. Throughout the paper, we sometimes mention signatures and actually mean a particular graph with that signature; for example, we may refer to the graph $(\omega^3+1)\vee C$. Also note that a proper homotopy equivalence of graphs induces a homeomorphism of their end spaces \cite[Section~2]{AK-B}. Thus, for example, by an abuse of notation, we may refer to the end space of any graph with signature $(\omega^3+1)\vee C$ by $\partial((\omega^3+1)\vee C)$. 

In the context of $(\{Y_n\}_n,a)\to\mathcal{C}$, where each of the $Y_n$ are arbitrary locally finite trees, $\mathcal{C}$ is an infinite binary splitting tree, and $a$ is a specific bijection, a notation that is similar in flavor capturing behavior of the end space was developed by \cite[Definition~0.4]{Ket}.

\begin{example}\label{ex: signatures}
    Some examples of signatures were given in Figure~\ref{fig: examples of graphs}; and we present additional examples here. See Figure~\ref{fig:examples of signatures} for pictures of each, from left to right. Note that all ends of line segments in figures throughout the paper are assumed to continue to infinity; we often drop arrows for neater drawings.
    \begin{enumerate}
        \item $R_3\vee 1\vee1\vee1$.
        \item $(\{Y_n\}_n,a)\rightarrow \mathcal{C}$, where $Y_0\cong R_0$ and $X_n\cong (\omega^{n-1}+1)\rightarrow o(1)$ for $n>0$, $a$ is injective and $\mathcal{C}$ is an infinite binary splitting tree.
        \item $R_1\to (K_5,x_0)$ is proper homotopy equivalent to $R_{11}$, where $K_5$ is the complete graph with five vertices which has signature $R_6$ and where $x_0\in\mathcal{V}(K_5)$. Note that there are graphs with one vertex whose signature is $R_6$ just like $K_5$, and attaching $R_1$ to the vertex of one of these graphs results in $R_7$, not $R_{11}$. Thus, the proper homotopy equivalence type of a graph may only be well defined if graphs appear to the right of arrows as opposed to signatures appearing to the right of arrows.
    \end{enumerate}
\end{example}

\begin{figure}[t]
    \centering
\begingroup%
  \makeatletter%
  \providecommand\color[2][]{%
    \errmessage{(Inkscape) Color is used for the text in Inkscape, but the package 'color.sty' is not loaded}%
    \renewcommand\color[2][]{}%
  }%
  \providecommand\transparent[1]{%
    \errmessage{(Inkscape) Transparency is used (non-zero) for the text in Inkscape, but the package 'transparent.sty' is not loaded}%
    \renewcommand\transparent[1]{}%
  }%
  \providecommand\rotatebox[2]{#2}%
  \newcommand*\fsize{\dimexpr\f@size pt\relax}%
  \newcommand*\lineheight[1]{\fontsize{\fsize}{#1\fsize}\selectfont}%
  \ifx\svgwidth\undefined%
    \setlength{\unitlength}{377.19498762bp}%
    \ifx\svgscale\undefined%
      \relax%
    \else%
      \setlength{\unitlength}{\unitlength * \real{\svgscale}}%
    \fi%
  \else%
    \setlength{\unitlength}{\svgwidth}%
  \fi%
  \global\let\svgwidth\undefined%
  \global\let\svgscale\undefined%
  \makeatother%
  \begin{picture}(1,0.37019673)%
    \lineheight{1}%
    \setlength\tabcolsep{0pt}%
    \put(0,0){\includegraphics[width=\unitlength,page=1]{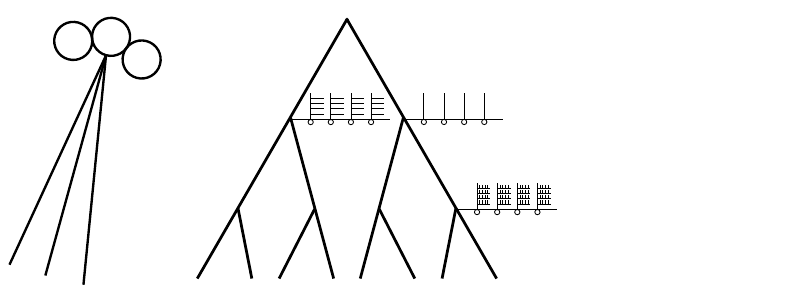}}%
    \put(0.47947462,0.35053197){\color[rgb]{0,0,0}\makebox(0,0)[t]{\smash{\begin{tabular}[t]{c}$n=0$\end{tabular}}}}%
    \put(0.52522392,0.23371749){\color[rgb]{0,0,0}\makebox(0,0)[t]{\smash{\begin{tabular}[t]{c}$1$\end{tabular}}}}%
    \put(0.35484051,0.23370671){\color[rgb]{0,0,0}\makebox(0,0)[t]{\smash{\begin{tabular}[t]{c}$2$\end{tabular}}}}%
    \put(0.56028132,0.10571779){\color[rgb]{0,0,0}\makebox(0,0)[t]{\smash{\begin{tabular}[t]{c}$3$\end{tabular}}}}%
    \put(0.46598681,0.10571779){\color[rgb]{0,0,0}\makebox(0,0)[t]{\smash{\begin{tabular}[t]{c}$4$\end{tabular}}}}%
    \put(0.38077882,0.10571779){\color[rgb]{0,0,0}\makebox(0,0)[t]{\smash{\begin{tabular}[t]{c}$5$\end{tabular}}}}%
    \put(0.28054826,0.10540061){\color[rgb]{0,0,0}\makebox(0,0)[t]{\smash{\begin{tabular}[t]{c}$6$\end{tabular}}}}%
    \put(0,0){\includegraphics[width=\unitlength,page=2]{examples_of_signatures.pdf}}%
  \end{picture}%
\endgroup%

    \centering
    \caption{The graphs in Example~\ref{ex: signatures}. In the middle graph, there is a copy of $(\omega^{n-1}+1)\rightarrow o(1)$ attached to each vertex $n$. To avoid over-complicating the figure, we have included only the first three.}
    \label{fig:examples of signatures}
\end{figure}

\subsection{Building blocks}\label{sec: building blocks}

This section discusses signatures in greater depth and introduces many lemmas which will be useful throughout the paper. It culminates with proving that every locally finite graph has a signature in Proposition~\ref{prop: building blocks}.

There are many signatures for a given locally finite graph. For example, for a countable successor ordinal $\alpha+1$, the signatures $\omega^{\alpha+1}+1$ and $(\omega^\alpha+1)\rightarrow 1$ represent proper homotopy equivalent graphs. For a countable ordinal $\alpha$ and any sequence $\{\beta_n\}_{n\in\Z_{\geq0}}$ of countable ordinals such that $\alpha=\sup\{\beta_1,\beta_2,...\}$, the graph $\omega^\alpha+1$ has signature $\{\omega^{\beta_n}+1\}_n\rightarrow 1$. Additionally, genus may be represented by an application of convergence, because $o(X)\cong R_1\rightarrow X$. Many parts of the notation above were introduced to simplify notation, and as we will see in Section~\ref{sec: ordered signatures}, give more structural information about the order on the space of ends.

It is natural to require only finitely many applications of genus, wedge, convergence, and spread. In the case of wedge, if $\{X_n\}_{n\in\Z_{\geq0}}$ is a sequence of locally finite graphs, then the degree of the wedge point of the locally finite graph $\bigvee_{i=0}^{n}X_i$ is at least $n+1$. Thus, $X_0\vee X_1\vee X_2\vee\cdots $ is not representative of a locally finite graph. Similar issues arise with genus, for example with $\cdots o(o(o(X)))\cdots$, and nested arrows to the left, as in $\cdots(X_3\rightarrow(X_2\rightarrow(X_1\rightarrow X_0)))\cdots$. Infinite nested arrows to the right does not lead to these issues, but are not straightforward and can be avoided. For example, we have that $\cdots((1\rightarrow 1)\rightarrow 1)\rightarrow\cdots\cong\omega^\omega+1$. Proposition~\ref{prop: building blocks} implies that these restrictions do not reduce the class of locally finite graphs that are describable in this way.

We now turn our attention to proving lemmas which will be useful throughout the paper. The following lemma is an implication of \cite[Lemma~4.17]{MRlong}, and we include a proof for completeness.

\begin{lemma}[Stable neighborhoods are all the same]\label{lem: 4.17}
    If $U$ is a clopen stable neighborhood of $\nu\in\partial X$, then any sub-neighborhood $\nu\in V\subset U$ is stable. Moreover, if $\nu'\sim\nu$, then $\nu'$ is stable too, and any clopen stable neighborhood $U'$ of $\nu'$ is homeomorphic to $U$. Often, we will apply this in the case that $\nu=\nu'$.
\end{lemma}
\begin{proof}
    First assume that $V$ is contained within $U$. It follows by taking $x=y$ in \cite[Lemma~4.17]{MRlong} that all sufficiently small neighborhoods of $\nu$ are homeomorphic to $U$. In particular, if $W$ is a sufficiently small neighborhood of $\nu$, then there exists a homeomorphism $f:U\rightarrow W$ fixing $\nu$ and a homeomorphic copy $f(V)$ of $V$ in $W$. The set $f(V)$, being a subset of $W$, is also a sufficiently small neighborhood of $\nu$, and is therefore homeomorphic to $U$. Thus $U$ is homeomorphic to $V$. For the moreover statement, suppose that $\nu'\sim\nu$. By \cite[Lemma~4.17]{MRlong}, all sufficiently small neighborhoods of $\nu'$ are homeomorphic to $U$ via a homeomorphism taking $\nu$ to $\nu'$, which immediately gives the result.
\end{proof}

We now give a complete classification of the orbits of maximal ends, expanding on Proposition~\ref{MR 4.7}. Recall that the orbit of an end $\nu\in\partial X$ under $\Maps(X)$ is equal to $E(\nu)$ by the discussion after Theorem~\ref{MR 1.2}.

\begin{lemma}[Orbits of maximal ends]\label{lem: orbits of max ends}
    Let $\nu\in\partial X$ be an end of the locally finite graph $X$. The following are equivalent.
    \begin{enumerate}
        \item $\nu$ is maximal.
        \item $E(\nu)$ is finite or homeomorphic to Cantor space.
        \item $E(\nu)$ is closed.
    \end{enumerate}
\end{lemma}
\begin{proof}
    The implication $1$ implies $2$ was shown in Proposition~\ref{MR 4.7}, and the implication $2$ implies $3$ is immediate. For the implication $3$ implies $1$, if $\nu$ was dominated by some $\mu\in\partial X$, then there would exist ends of the same type as $\nu$ in arbitrarily small neighborhoods of $\mu$, and hence $\mu$ would be a limit point of the set of ends equivalent to $\nu$. Hence the orbit of $\nu$ would not be closed.
\end{proof}

Our next result shows that a wedge of locally finite graphs corresponds to a clopen decomposition of end spaces up to proper homotopy. We will use this often as a tool for going back and fourth between a graph and its end space.

\begin{lemma}[Wedges and clopen sets]\label{lem: wedges and clopens}
    Let $X$ be a locally finite graph. There are locally finite graphs $Y$ and $Z$ such that $X\cong Y\vee Z$ if and only if there is a clopen decomposition of the end space $\partial X\cong A\sqcup B$ such that $A\cong\partial Y$ and $B\cong\partial Z$. 
\end{lemma}
\begin{proof}
    For the forward direction, up to proper homotopy, we may assume that $Y\vee Z$ has a vertex $x_0$ with valence two, with one adjacent vertex in $Y$ and the other in $Z$. Let $U$ be a small connected open neighborhood of $x_0$ not containing any other vertices. Then $(Y\vee Z)\setminus U$ has two components, with $Y$ and $Z$ in different components. Furthermore, each component is closed, and thus the end spaces of $Y$ and $Z$ are both closed in the end space of $Y\vee Z$. Thus $\partial Y$ and $\partial Z$ are clopen in $\partial(Y\vee Z)$, and we can take $A=\partial Y$ and $B=\partial Z$.
    
    For the backward direction, suppose $\partial X\cong A\sqcup B$. Note that $A_g$ is closed in $\partial X$, because $A$ and $\partial X_g$ are closed in $\partial X$. Define $Y$ to be a locally finite graph with characteristic pair $(A,A_g)$ such that $g(Y)$ is zero if $A_g=\0$ and infinity otherwise. Likewise define $Z$ to be a locally finite graph with characteristic pair $(B,B_g)$ with $g(Z)$ defined analogously. Thus $A=\partial Y$ and $B=\partial Z$. By construction, $Y\vee Z$ has the same characteristic pair as $X$. If $0<g(X)<\infty$, then we may replace $Y$ with $Y':=Y\vee R_{g(X)}$. After this, $Y\vee Z$ is proper homotopy equivalent to $X$ by Theorem~\ref{AK-B 2.2}.
\end{proof}

The following result shows that closed subsets of $\partial X$ have arbitrarily small clopen neighborhoods. We will often use this result in tandem with Lemma~\ref{lem: wedges and clopens} to separate ends into different wedge components.

\begin{lemma}[Small clopen neighborhoods]\label{lem: small clopen neighborhoods}
    If $A\subseteq\partial X$ is a closed subset of an end space, and if $V\subset \partial X\setminus A$ is closed, then there exists a clopen neighborhood $U$ of $A$ with $U\cap V=\0$.
\end{lemma}
\begin{proof}
    We first prove the statement when $A=\{\nu\}$ is a singleton. Without loss of generality, assume that $X$ is in the standard form of \cite[Section~2]{AK-B}. Let $x_0$ be a vertex in $X$, and consider a non-backtracking path based at $x_0$ in $X$ limiting to $\nu$ with vertices $\{x_n\}_n$. Because $X$ is in standard form, $X\cong\{Y_n\}_n\rightarrow 1$, where each $Y_n$ is attached to the path limiting to $\nu$ via $x_n$. For each $n$, the set $\{\nu\}\cup(\bigcup_{i\geq n} \partial Y_i)$ is clopen in $\partial X$. If, for contradiction, $V$ intersected infinitely many of these sets, then $\nu$ would be a limit point of $V$, contradicting that $V$ is closed. Thus take $U$ to be one of these sets for suitably large $n$.

    Now let $A$ be any closed subset of $\partial X$. For each $\nu\in A$, consider a clopen neighborhood $U_\nu$ of $\nu$ such that $U_\nu\cap V=\0$, which exists by the above argument. Thus $\{U_\nu\,:\,\nu\in A\}$ is an open cover of $A$, which is a closed subset of a compact space, hence compact. We may pass to a finite sub-cover and observe that a finite union of clopen sets is clopen. Thus, the union of sets in this sub-cover is a clopen neighborhood of $A$ disjoint from $V$.
\end{proof}

\begin{lemma}[Cantor or discrete]\label{lem: Cantor or discrete}
    Let $X$ be a locally finite graph. If $\partial X$ is nonempty and not homeomorphic to Cantor space, then $X$ has a discrete end.
\end{lemma}
\begin{proof}
    The end space $\partial X$ is compact and Hausdorff. By \cite[Definition~2.5]{AK-B}, it embeds into Cantor space, and therefore has a countable basis consisting of clopen sets. If $\partial X$ is not homeomorphic to Cantor space, then by Brouwer's characterization of Cantor space, it has an isolated point. Hence $X$ has a discrete end.
\end{proof}

Our next result is more technical and is used to show that every locally finite graph has a signature. It makes precise the notion that the operation of spread is general, in that every locally finite graph can be written as a spread applied to a simpler locally finite graph.

\begin{lemma}[The invocation of spread]\label{lem: invocation of spread}
    Let $V$ be a closed subset of an end space $\partial X$. Let $Z$ be a sub-graph of $X$ with characteristic pair $(V,V_g)$. Then there exists a sequence of sub-graphs $\{Y_n\}_n$ of $X$ such that $X\cong(\{Y_n\}_n,a)\rightarrow Z$ for some bijective function $a \colon \mathcal{V}(Z)\rightarrow\Z_{> 0}$.
\end{lemma}
\begin{proof}
    By Lemma~\ref{lem: small clopen neighborhoods}, for all $\nu\in \partial X\setminus V$, there exists a clopen set $U_\nu\ni\nu$ contained in $\partial X\setminus V$. Hence, $\{U_\nu\}_{\nu\in\partial X}$ is a clopen cover of $\partial X\setminus V$. We modify this cover to ensure that it is countable and disjoint. Because $\partial X$ embeds into Cantor space, it is both second countable and metrizable, and therefore has a countable sub-cover $\{U_m''\}_{m\in\Z_{>0}}$. For disjointness, let $U_1'=U_1''$, and for $m>1$, inductively define $U_m'=U_m''\setminus(U_1'\cup\cdots\cup U_{m-1}')$. Re-index so that all $m$ such that $U_m'=\0$ form a tail end of the sequence. Note that if there are infinitely many non-empty $U_m'$, then this re-indexing ensures that none of the $U_m'$ are empty.

    Using the construction in \cite[Definition~2.5]{AK-B}, we fix an embedding of $\partial X$ into $\R$. For any clopen subset of the end space $U\subseteq\partial X$, there exist finitely many disjoint open intervals $I_k$ in $\R$ such that $\partial X\cap(\bigcup_k I_k)=U$. Indeed, by Lemma~\ref{lem: small clopen neighborhoods} and compactness of $U$, there are finitely many open intervals in $\R$ covering $U$ with $\partial X\cap(\bigcup_k I_k)=U$, and we may assume that they are disjoint because if two overlapped, we could combine them into one open interval. Thus for each $m$ we obtain finitely many disjoint open intervals $I_{m,1},...,I_{m,k_m}$ covering $U_m'$ that are disjoint from $\partial X\setminus U_m'$. $\{I_{m,i}\,:\,m,i\in\Z_{>0}$ and $1\leq i\leq k_m\}$ is a set of pairwise disjoint open intervals in $\R$, and thus we assume that they are indexed in an order-preserving way by $\Z$. Define $U_m:= I_m\cap\partial X$ so that $\{U_m\}_{m\in\Z}$ is a clopen decomposition of $\partial X\setminus V$.

    \begin{figure}[t]
        \centering
        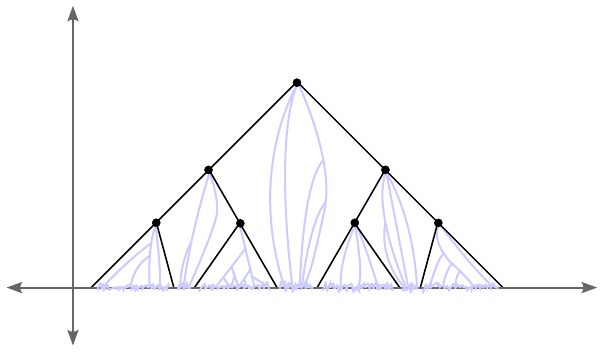
        \centering
        \caption{The graph $X$ viewed as a subspace of $\R^2$ with $\partial X$ on the x-axis. Here, $Z$ is in black and the rest of $X$ is in periwinkle.}
        \label{fig: image of fuzz}
    \end{figure}

    Next, we construct the collection $\{Y_n\}_{n\in\Z_{\ge0}}$ and the attaching map $a$. View $X$ as a topological subset of the upper half plane of $\R^2$ with $\partial X$ in $\R\times\{0\}$ using the embedding fixed above. See Figure~\ref{fig: image of fuzz}. Every connected component of $\{(x,y)\in\R^2\,:\,y\geq0\}\setminus(Z\cup\partial Z)$ intersects a countable set of vertices of $Z$ which limit to two points in $V$ (one point if $Z$ is a ray). Fix an enumeration $\{x_n\}_{n\in Z_{>0}}$ of the vertices of $Z$. For each $m$, we assign $U_m$ to a vertex $x_n$ of $Z$ incident to the connected component of $\{(x,y)\in\R^2\,:\,y\geq0\}\setminus (Z\cup\partial Z)$ containing $U_m$ in such a way that any pair of adjacent open sets $U_m,U_{m+1}$ in the same connected component correspond to adjacent vertices incident to that component. See Figure~\ref{fig: construction of attaching map}. Let $Y_0:=\{*\}$, and for $n>0$, let $Y_n$ be a sub-graph such that $\partial Y_n$ is homeomorphic to the disjoint union of all the $U_m$ corresponding to $x_n$ if such a disjoint union is nonempty, and $R_0$ otherwise. Define $a(x_n)=n$.

    \begin{figure}[h]
        \centering
        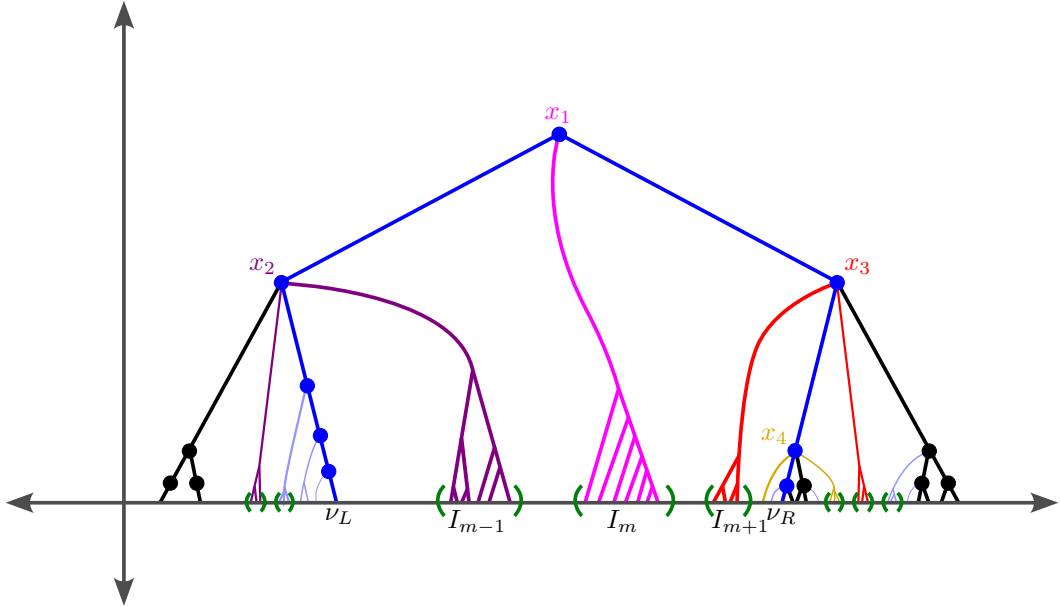
        \centering
        \caption{The path formed from the countable collection of vertices adjacent to the large central region which limits to $\nu_L$ and $\nu_R$ is colored blue. The subgraphs $Y_n$ for $1\leq n\leq 4$ are depicted in the same color as the $x_n$ of the same index.}
        \label{fig: construction of attaching map}
    \end{figure}

    Lastly, $X$ is proper homotopy equivalent to $(\{Y_n\}_n,a)\rightarrow Z$ for $a \colon \mathcal{V}(Z)\rightarrow\Z_{>0}$, because both of these graphs have sub-spaces of their end spaces homeomorphic to $Z$, and the rest of their end spaces are covered by clopen sets homeomorphic to the set of $U_n$. Moreover, we constructed $a$ so that for a sequence of ends in $\partial X\setminus V$ that converge to an end in $V$, the corresponding ends in $(\{Y_n\}_n,a)\rightarrow Z$ converge to the corresponding end.
\end{proof}

Using the previous lemma, we are now able to show that every graph has a signature. In Section~\ref{sec: ordered signatures}, we will show that when a graph is stable, it has an ordered signature, which reveals much more information about the properties of the graph.

\begin{proposition}\label{prop: building blocks}
    Let $X$ be a locally finite graph. Then $X$ has a signature.
\end{proposition}
\begin{proof}
    If $X$ is a finite graph, then $X\cong R_{g(X)}$. We now show the statement assuming $X$ is a tree. By the Cantor-Bendixon Theorem, the end space $\partial X$ splits uniquely into $U$ and $V$ where $V$ is perfect and $U$ is countable. If $V=\0$, then $X\cong\bigvee_{i=1}^k(\omega^\alpha+1)$ by Lemma~\ref{lem: wedges and clopens} and Theorem~\ref{MS 1}, where $k$ is the number of maximal ends. Otherwise, if $V\neq\0$, then $V$ must be homeomorphic to Cantor space. In this case, let $\mathcal{C}$ denote the infinite binary splitting tree. Then $\mathcal{C}$ has signature $C$, and we may apply Lemma~\ref{lem: invocation of spread} to get that $X\cong(\{Y_n\}_n,a)\to\mathcal{C}$ for some sequence $\{Y_n\}_n$ of subgraphs of $X$, each with countable end space, and some function $a\colon\mathcal{V}(\mathcal{C})\to\Z_{>0}$. Thus, we may write $X$ more explicitly as $(\{\omega^{\alpha_n}+1\}_n,a)\to\mathcal{C}$, where all the $\alpha_n$ are countable ordinals, and $a\colon\mathcal{V}(\mathcal{C})\to\Z_{>0}$. This shows that all locally finite trees have signatures.

    As signatures are invariant under proper homotopy equivalence and by Proposition~\ref{AK-B 2.6}, we may assume that $X$ is in standard form. Let $T$ be the underlying tree. Then the signature of $T$ was already constructed above, and $X\cong(\{Y_n\}_n,a)\to T$ for some $a\colon\mathcal{V}(T)\to\Z_{\geq0}$, where $Y_0=R_0$ and $Y_n=R_1$ for $n>0$.
\end{proof}

Intuitively, the proof of Proposition~\ref{prop: building blocks} shows that every locally finite graph can be constructed by taking wedge products of $C$, $R_1$, and $\omega^\alpha+1$ for countable ordinals $\alpha$.

\section{Local Structures}\label{sec: structure analysis}

In this section, we study stable ends. By Lemma~\ref{lem: 4.17}, all stable neighborhoods of a given stable end are homeomorphic, and thus when discussing a stable end, it often suffices to discuss an arbitrary stable neighborhood of that end. Stable neighborhoods are assumed to be clopen, and thus are the end space of a certain locally finite graph. For example, in the graph $\omega^\omega+1$, any stable neighborhood of an end of Cantor-Bendixon rank two has stable neighborhoods homeomorphic to the end space of the graph $\omega^2+1$ (see Figure~\ref{fig: local structure}). In fact, many properties about the end $\nu$ correspond to analogous properties about the graph $\omega^2+1$. For this reason, we say that the end $\nu$ has the local structure of $\partial (\omega^2+1)$. We give the general definition below.

\begin{figure}[h]
    \centering
\begingroup%
  \makeatletter%
  \providecommand\color[2][]{%
    \errmessage{(Inkscape) Color is used for the text in Inkscape, but the package 'color.sty' is not loaded}%
    \renewcommand\color[2][]{}%
  }%
  \providecommand\transparent[1]{%
    \errmessage{(Inkscape) Transparency is used (non-zero) for the text in Inkscape, but the package 'transparent.sty' is not loaded}%
    \renewcommand\transparent[1]{}%
  }%
  \providecommand\rotatebox[2]{#2}%
  \newcommand*\fsize{\dimexpr\f@size pt\relax}%
  \newcommand*\lineheight[1]{\fontsize{\fsize}{#1\fsize}\selectfont}%
  \ifx\svgwidth\undefined%
    \setlength{\unitlength}{368.04328979bp}%
    \ifx\svgscale\undefined%
      \relax%
    \else%
      \setlength{\unitlength}{\unitlength * \real{\svgscale}}%
    \fi%
  \else%
    \setlength{\unitlength}{\svgwidth}%
  \fi%
  \global\let\svgwidth\undefined%
  \global\let\svgscale\undefined%
  \makeatother%
  \begin{picture}(1,0.32213331)%
    \lineheight{1}%
    \setlength\tabcolsep{0pt}%
    \put(0,0){\includegraphics[width=\unitlength,page=1]{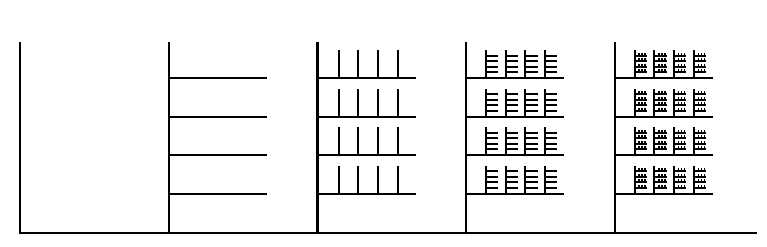}}%
    \put(0.41045904,0.28219827){\color[rgb]{0,0,0}\makebox(0,0)[t]{\smash{\begin{tabular}[t]{c}$\nu$\end{tabular}}}}%
    \put(0,0){\includegraphics[width=\unitlength,page=2]{local_structure.pdf}}%
  \end{picture}%
\endgroup%

    \caption{The end $\nu$, contained in the graph $\omega^\omega+1$, has stable neighborhoods homeomorphic to $\partial(\omega^2+1)$. One such stable neighborhood is the set of ends in the periwinkle box.}
    \label{fig: local structure}
\end{figure}

\begin{definition}\label{def: local structure}
    The end space $\partial Z$ of a locally finite graph $Z$ is called a \emph{local structure} if there is some locally finite graph $X$ and some stable end $\nu\in\partial X$ such that any stable neighborhood of $\nu$ is homeomorphic to $\partial Z$.
\end{definition}

In tandem with Definition~\ref{def: local structure}, we extend the definition of the order on ends to involve local structures.

\begin{definition}
    Let $\nu\in \partial X$ be any end, and let $\partial Z$ be a local structure.
    \begin{enumerate}
        \item We say that \emph{$\nu$ is of type $\partial Z$} and write $\nu\sim\partial Z$ if $\nu$ has a stable neighborhood homeomorphic to $\partial Z$. Given a locally finite graph $X$, we write $E(\partial Z)$ for the set $\{\nu\in\partial X\,:\,\nu\sim\partial Z\}$. 
        \item We say that \emph{$\nu$ dominates $\partial Z$}, and write $\nu\succeq\partial Z$, if every neighborhood of $\nu$ contains a homeomorphic copy of $\partial Z$.
        \item We say that \emph{$\nu$ is dominated by $\partial Z$}, and write $\nu\preceq\partial Z$, if $E(\nu)$ intersects a stable neighborhood of an end of type $\partial Z$ in $Z\vee X$.
    \end{enumerate}
\end{definition}

If $\nu\sim\partial Z$, we can view $\nu$ as a maximal end in $Z$, because $\nu$ has a stable clopen neighborhood $U$ with $U\cong\partial Z$ such that all points in $U\setminus\{\nu\}$ have orbits accumulating to $\nu$.

\begin{example}
    An end $\nu\in\partial X$ is of type $\partial 1$ or $\partial o(1)$ if and only if $\nu$ is is an isolated point in $\partial X$. An end $\nu\in\partial X$ is of type $\partial C$ or $\partial o(C)$ if and only if there exists a clopen neighborhood in $\partial X$ of $\nu$ which is homeomorphic to Cantor space.
\end{example}

In Section~\ref{sec: classification of local structures}, we fully characterize local structures. In Section~\ref{sec: poset of lfgs}, we define and study the structure of the poset of local structures. Section~\ref{sec: wedge decomposition} describes a canonical decomposition of any locally finite graph with stable maximal ends into the wedge product of finitely many local structures, and Section~\ref{sec: ordered signatures} gives a criterion for determining if a given locally finite graph is stable.

\subsection{Classification of local structures}\label{sec: classification of local structures}

In this section, we show that the end space of a locally finite graph is a local structure if and only if it is self-similar. Showing this involves exploring the relationship between stability and self-similarity, and we will see in a precise sense that stability is a local version of self-similarity.

\begin{proposition}[Classification of local structures]\label{prop: classification of local structures}
    Let $Z$ be any locally finite graph. Then the following are equivalent.
    \begin{enumerate}
        \item $\partial Z$ is a local structure.
        \item $Z$ has a self-similar end space.
        \item There is a unique maximal end type of size either one or infinity in $\partial Z$, and all maximal ends are stable.
        \item The whole end space $\partial Z$ is a stable neighborhood of a maximal end in $\partial Z$.
    \end{enumerate}
\end{proposition}
\begin{proof}
    To show $1\Rightarrow 2$, let $\nu\sim\partial Z$ where $\nu\in\partial X$, and let $U\cong\partial Z$ be a clopen stable neighborhood of $\nu$ in $\partial X$. Let $f\colon \partial Z\to U$ be a homeomorphism. Suppose that $E_1\sqcup\cdots\sqcup E_n$ is a clopen decomposition of $\partial Z$ and that $f^{-1}(\nu)\in E_i$. Then $f(E_i)$ contains $\nu$ and is a clopen subset of $U$. Thus we have that $E_i\cong f(E_i)\cong U\cong \partial Z$, where the middle homeomorphism follows from Lemma~\ref{lem: 4.17}. Therefore, $\partial Z$ is self-similar.

    We now show $2\Rightarrow 3$. First we show that $Z$ must have a unique maximal end type of size one or infinity. Assume for contradiction that $\mu_1,\mu_2\in\partial Z$ are two maximal ends such that $\mu_1\not\sim\mu_2$. By Lemma~\ref{lem: orbits of max ends}, both $E(\mu_1)$ and $E(\mu_2)$ are closed. By Lemma~\ref{lem: small clopen neighborhoods}, there exists a clopen neighborhood $E_1$ of $E(\mu_1)$ such that $E_1\cap E(\mu_2)=\0$. Let $E_2:=\partial Z\setminus E(\mu_2)$. Then $E_1\sqcup E_2$ is a clopen decomposition of $\partial Z$ such that there is no homeomorphic copy of $\partial Z$ in either $E_1$ or $E_2$, for such a copy would have to contain elements of both $E(\mu_1)$ and $E(\mu_2)$. Thus $\partial Z$ must have a unique maximal end type. Likewise, if $Z$ had more than one but finitely many maximal ends, we would be able to follow a similar argument as above, where $E_1$ and $E_2$ separate maximal ends to reach a similar contradiction.

    Next we show that every maximal end must be stable. Assume for contradiction that there exists a non-stable maximal end. Then since there exists a unique maximal end type, by Lemma~\ref{lem: 4.17}, there are no stable maximal ends. Then for each maximal end $\mu\in\partial Z$, there exists a clopen neighborhood $V_\mu'$ of $\mu$ which does not contain any homeomorphic copy of all of $\partial Z$. By Lemma~\ref{lem: orbits of max ends}, the set of all maximal ends is compact, and thus is covered by finitely many of these sets, say $V_{\mu_1}',...,V_{\mu_n}'$. We modify them to make them disjoint in the following way. Let $V_{\mu_1}:=V_{\mu_1}'$, and for $i>1$ let $V_{\mu_i}=V_{\mu_i}'\setminus\{V_{\mu_1}\cup\cdots\cup V_{\mu_{i-1}}\}$. Then it is still true that none of the $V_{\mu_i}$ contain a homeomorphic copy of $\partial Z$. If $V=\partial Z\setminus\{V_{\mu_1}\cup\cdots\cup V_{\mu_n}\}$, then $V$ is also a clopen set and cannot contain a homeomorphic copy of $\partial Z$, as it does not contain maximal ends. Thus, $V_{\mu_1}\sqcup\cdots\sqcup V_{\mu_n}\sqcup V$ is a clopen decomposition which gives a contradiction to self-similarity.

    We now show $3\Rightarrow 4$. Let $U\subsetneq\partial Z$ be a stable neighborhood of a maximal end $\mu\in\partial Z$. Let $V=\partial Z\setminus U$. If we can show that there is a homeomorphic copy $V'\subset U$ of $V$ which does not contain $\mu$, then the result would follow because $U\setminus V'\cong U$ by Lemma~\ref{lem: 4.17} and so $U=V'\sqcup(U\setminus V')\cong V\sqcup U=\partial Z$.

    \begin{claim}
        For each $\nu$ in $V$, there is an open neighborhood $V_\nu\subset V$ of $\nu$ such that $U$ contains a homeomorphic copy of $V_\nu$ which does not contain $\mu$.
    \end{claim}
    \begin{proof}
        If $\nu$ is non-maximal in $\partial Z$, then $\nu\prec\mu$, and so the result follows by the definition of the order on ends. If $\nu$ is maximal, then $\nu\sim\mu$ by the assumption that there is a unique maximal end type. In this case, there is more than one end in $E(\mu)$, and so $E(\mu)$ must be homeomorphic to Cantor space by Proposition~\ref{MR 4.7}. Then $E(\mu)\cap U$ contains more than one element, and thus the result follows from Theorem~\ref{MR 1.2}.
    \end{proof}

    Let $\{V_\nu\}_{\nu\in V}$ be the clopen cover of $V$ obtained from the above claim. Because $U$ is clopen, so is $V$, and thus there exists a finite sub-cover $\{V_{\nu_1},...,V_{\nu_n}\}$ of $V$. There is a homeomorphic copy $V_1$ of $V_{\nu_1}$ in $U$, and a homeomorphic copy $V_2$ of $V_{\nu_2}$ in $U\setminus V_1\cong U$ (the homeomorphism follows from Lemma~\ref{lem: 4.17}), and so on. Thus $U\cong\partial Z$.

    The implication $4\Rightarrow 1$ follows immediately from definitions.
\end{proof}

Proposition~\ref{prop: classification of local structures} suggests that stability is a local version of self-similarity. Indeed, the bi-implication $1\Leftrightarrow 2$ is saying that a neighborhood of an end is stable if and only if a locally finite graph with end space equal to that neighborhood has self-similar end space. There is a construction in \cite{MRshort} of graphs with non-stable end spaces which contain a unique maximal end type such that this end type is of size one or infinity. Thus, the requirement in statement $3$ that maximal ends are stable is crucial. The graphs in \cite{MRshort} are not self-similar, and do not satisfy any of the conditions in Proposition~\ref{prop: classification of local structures}.

\subsection{The poset of local structures}\label{sec: poset of lfgs}

In this section, we investigate the poset of local structures. We start by proving that the set of local structures has a natural partial ordering on it, and then give properties of this poset. Throughout, we compare and contrast this poset with the poset of countable ordinals.

\begin{definition}
    Let $\partial Z$ and $\partial Z'$ be local structures. We say that $\partial Z\preceq \partial Z'$ if there is a clopen homeomorphic copy of $\partial Z$ in $\partial Z'$.
\end{definition}

To show that the set of local structures together with the relation $\preceq$ defined above gives a poset, we need the following lemma.

\begin{lemma}[Translating local structures to ends]\label{lem: translate local structures to ends}
    Let $\partial Z$ and $\partial Z'$ be local structures. Then $\partial Z\preceq\partial Z'$ if and only if there exists a locally finite graph $X$ with stable ends $\nu\sim\partial Z$ and $\nu'\sim\partial Z'$ such that $\nu\preceq\nu'$ in $\partial X$.
\end{lemma}
\begin{proof}
    First suppose that $\partial Z\preceq\partial Z'$, and let $X\cong Z\vee Z'$. Then the maximal ends in $Z$ are of type $\partial Z$, and the maximal ends in $Z'$ are of type $\partial Z'$. Let $\nu\sim\partial Z$ and $\nu'\sim\partial Z'$ in $X$. By stability of $\nu'$, there exist arbitrarily small copies of $\partial Z'$ around $\nu'$, and each one contains a homeomorphic copy of $\partial Z$. Each of these neighborhoods contains an end of the same type as $\nu$, and thus $\nu'$ is a limit point of $E(\nu)$, which implies that $\nu\preceq\nu'$ in $\partial X$. Conversely, suppose that $\nu\sim\partial Z$, $\nu'\sim\partial Z'$, and $\nu\preceq\nu'$ in $\partial X$. Let $U$ be a stable neighborhood of $\nu'$. Then there exists $\lambda\sim\nu$ such that $\lambda\in U$ containing arbitrarily small neighborhoods homeomorphic to $\partial Z$. Thus, $\partial Z\preceq\partial Z'$.
\end{proof}

One benefit of studying local structures is that they allow us to compare stable ends on different graphs. We will later show that for any locally finite graph $X$, there exists an end type which is larger than all the end types in $X$, which suggests that the poset of stable ends in a given locally finite graph is one small part of a larger poset. The previous lemma gives a method of translating between these two posets. We are now ready to show that the relation $\preceq$ indeed gives a partial ordering on local structures, a fact which we will use without reference after its proof.

\begin{proposition}[Poset of local structures]\label{prop: poset of end types}
    The relation defined on local structures above is a poset, where equivalence is taken up to homeomorphism of end spaces.
\end{proposition}
\begin{proof}
    It is straightforward to verify that the relation is reflexive and transitive. For anti-symmetry, suppose that $\partial Z\preceq\partial Z'$ and that $\partial Z'\preceq\partial Z$. Let $\mu$ and $\mu'$ be of types $\partial Z$ and $\partial Z'$ respectively in the locally finite graph $Z\vee Z'$. By Lemma~\ref{lem: translate local structures to ends}, we have that $\mu\preceq\mu'$ and $\mu'\preceq\mu$ in $Z\vee Z'$. Thus, because the order on ends gives a poset, $\mu\sim\mu'$. By Lemma~\ref{lem: 4.17}, they must have homeomorphic stable neighborhoods, and thus $\partial Z\cong\partial Z'$.
\end{proof}

We now show that there are only four minimal local structures. To show this, we make use of the fact that local structures, similar to ordinals, are \emph{downward closed} in the following sense. If $\alpha$ and $\beta$ are ordinals, then $\alpha\leq\beta$ implies that $\alpha\subseteq\beta$. Similarly, if $\partial Z$ and $\partial Z'$ are local structures, then $\partial Z\preceq\partial Z'$ implies that $\partial Z\subseteq\partial Z'$. Ordinals and local structures both contain all predecessors, and thus, the minimal local structures are precisely those where every end is of the same type.

\begin{lemma}[Minimal local structures]\label{lem: minimal local structures}
    The set of minimal elements of the local structure poset is precisely given by $$\{\partial 1,\partial o(1),\partial C,\partial o(C)\}.$$
\end{lemma}
\begin{proof}
    By the discussion in the previous paragraph, a local structure is minimal if and only if every end in it is of the same type. This is clearly the case with $\partial 1,\partial o(1),\partial C$, and $\partial o(C)$. Now suppose for contradiction that $\partial Z$ is another minimal local structure. It suffices to assume that $Z$ is either a tree or that every end in $Z$ is accumulated by genus, or else there would be multiple end types in $\partial Z$. Suppose $Z$ is a tree. The Cantor-Bendixon theorem guarantees the existence of a decomposition $\partial Z=A\sqcup B$ where $A$ is a perfect set and where $B$ is at most countable. As every end in $\partial Z$ must be of the same type, either $A$ or $B$ must be empty. If $A$ is empty, then $\partial Z$ has an isolated point by Lemma~\ref{lem: Cantor or discrete}, and so $\partial Z$ must equal $\partial 1$ as $\partial Z$ is minimal, which is a contradiction. If $B$ is empty, then $A$ is homeomorphic to Cantor space by Brouwer's characterization of Cantor space, and so $\partial Z\cong\partial C$, which is a contradiction. The proof is analogous when every end of $\partial Z$ is accumulated by genus.
\end{proof}

The next proposition shows that the poset of stable local structures satisfies the descending chain condition. The countable ordinals also satisfy this condition by virtue of being well-ordered, but it is not as obvious that arbitrary descending chains of local structures have this property when they are not of the form $\omega^\alpha+1$ for some countable ordinal $\alpha$.

\begin{proposition}[Descending chain condition]\label{prop: descending chain condition}
    Let $Z_0$ be a stable and self-similar locally finite graph, and let $\partial Z_0\succeq \partial Z_1\succeq \partial Z_2\succeq\cdots$ be a descending sequence of local structures. Then there exists $n\in\Z_{\geq0}$ such that for all $m\geq n$, we have that $\partial Z_n\cong \partial Z_m$.
\end{proposition}
\begin{proof}
    Suppose for contradiction that the $\partial Z_i$ form a strictly decreasing sequence of local structures. We will inductively define a descending sequence $\{U_i\}_i$ of subsets of $\R$. We can homeomorphically map $\partial Z_0$ to some subset $U_0\subset\R$ of the real numbers by the construction in \cite[Definition~2.5]{AK-B}, and moreover, we may assume that the diameter of $U_0$, which we refer to as $\diam(U_0)$, is less than one in $\R$. Now suppose we have defined $U_0,...,U_{i-1}$. Because $\partial Z_{i}\prec \partial Z_{i-1}$, there exists $\nu_i\in U_{i-1}$ such that $\nu_i\sim\partial Z_i$. We can realize a homeomorphic copy of $\partial Z_i$ as a subset $U_i$ of $U_{i-1}$ which is clopen in $U_{i-1}$ and which contains $\nu_i$. Furthermore, by stability, we may assume that $\diam(U_i)\leq\frac{1}{i+1}$. Thus, the $U_i$ are descending, and they are all compact and non-empty, with diameter approaching zero. This implies that $\bigcap_i U_i$ consists of a single point $x\in\R$. Let $\nu$ be the end in $\partial Z_0$ which maps to $x$, and let $\{V_i\}_i$ be a descending sequence of clopen subsets of $\partial Z_0$ such that $V_i\cong U_i$ for all $i\in\Z_{\geq 0}$. Note that as $Z_0$ is stable, the end $\nu$ must be stable too.
    
    Suppose that there exists $i\in\Z_{\geq0}$ such that $\nu\sim \partial Z_i$. But then as $\nu$ is stable, Lemma~\ref{lem: 4.17} implies that $V_{i+1}$, being a clopen subset of $V_i$ which contains $\nu$, must be homeomorphic to $V_i$. This would be a contradiction to the assumption that $\partial Z_i\not\cong\partial Z_{i+1}$. Thus, $\nu\not\sim\partial Z_i$ for any $i$. 
    
    $\nu$ is a stable end and thus has a stable neighborhood. Because $\diam(U_i)$ approaches zero, any stable neighborhood of $\nu$ contains $V_i$ for some $i$. Thus, as $\nu\in V_i$, $V_i$ is a stable neighborhood of $\nu$ by Lemma~\ref{lem: 4.17}. But this implies that $\nu\sim \partial Z_i$, a contradiction.
\end{proof}

The poset of local structures is much larger than the poset of ends in any given locally finite graph. In fact, as we show in the next lemma, there is no locally finite graph that contains all local structures, and no locally finite graph whose end space can embed as a clopen set into all local structures. Thus, the poset of stable ends in a given locally finite graph is a strict sub-poset of the poset of local structures.

\begin{lemma}[Incomparable local structures]\label{lem: incomparable local structures}
    Let $X$ be any locally finite graph. Then there exists a local structure $\partial Z$ such that $\partial X$ and $\partial Z$ do not contain homeomorphic copies of each other.
\end{lemma}
\begin{proof}
    If $\partial X$ is countable, then $Z:=C$ proves the statement, and so we assume that $\partial X$ is uncountable. Let $\alpha$ be the minimal ordinal such that if $\nu\in\partial X$ is of type $\partial (\omega^\beta+1)$ for some countable ordinal $\beta$, then $\alpha>\beta$. Note that $\alpha$ must be a countable ordinal, because the set of ends of type $\partial(\omega^\beta+1)$ for some countable ordinal $\beta$ must be countable by the Cantor-Bendixson Theorem, and the supremum of an ascending chain of countable ordinals must be countable. We now claim that $\partial X$ and $\partial (\omega^\alpha+1)$ do not contain homeomorphic copies of the other. Indeed there is no copy of $\partial X$ in $\partial (\omega^\alpha+1)$ because $\partial X$ is uncountable by assumption whereas $\partial(\omega^\alpha+1)$ is not. Conversely there is no homeomorphic copy of $\partial (\omega^\alpha+1)$ in $\partial X$ because there is no end in $\partial X$ which is of type $\partial (\omega^\alpha+1)$ by construction.
\end{proof}

\begin{lemma}[Immediate successors]\label{lem: immediate successors}
    Let $X$ be a locally finite graph. There exists a local structure $\partial Z^1$ with a unique maximal end that contains a homeomorphic copy of $\partial X$, but such that $X$ does not contain a homeomorphic copy of $\partial Z^1$. Moreover, if $\partial X$ is a local structure, then there is no local structure $\partial Y$ such that $\partial X\prec\partial Y\prec\partial Z^1$. There also exists a local structure $\partial Z^C$ with infinitely many maximal ends satisfying the same conditions.
\end{lemma}
\begin{proof}
    By Lemma~\ref{lem: incomparable local structures}, there exists a local structure $\partial Z'$ which does not contain a clopen copy of $\partial X$ and vice-versa. Let $Z^1:=(X\vee Z')\rightarrow 1$. The end space of $Z^1$ is a local structures, and it is immediate that there is a homeomorphic copy of $\partial X$ in $\partial Z^1$. There are no ends in $X$ of type $\partial Z'$, but there are in $\partial Z^1$, and hence the first result is shown. Now suppose that $\partial X$ is a local structure. Every end in $\partial Z^1$ is either contained in a copy of $\partial Z'$, in a copy of $\partial X$, or is of type $\partial Z^1$. If there exists a local structure $\partial Y$ such that $\partial X\prec\partial Y\prec\partial Z^1$, then there must be an end in $\partial Z^1$ of type $\partial Y$. Such an end $\nu$ cannot be contained in a copy of $\partial X$, or else by the definition of the poset, we would have that $\partial X\prec\nu\preceq\partial X$, a contradiction. If such a $\nu$ were in a copy of $\partial Z'$, then since $\partial X\prec\partial Y$, we would have that $\partial X\prec\nu\preceq\partial Z'$, a contradiction to $\partial Z'$ being incomparable to $\partial X$. Lastly, if such a $\nu$ were of type $\partial Z^1$, then we would have that $\partial Y\cong\partial Z$, a contradiction. An identical argument works for $Z^C:=(X\vee Z')\rightarrow C$.
\end{proof}

Because local structures all have immediate successors, by Lemma~\ref{lem: immediate successors}, both $\partial Z^1$ and $\partial Z^C$ function similarly to successor ordinals. On the other hand, many local structures, such as $\omega^\omega+1$, are not immediate successors to any other local structures and function similarly to limit ordinals. These are minimal upper bounds of certain sequences of locally finite graphs. In the following proposition, we show that minimal upper bounds exist for any countable collection of local structures.

\begin{proposition}[Minimal upper bounds]\label{prop: minimal upper bounds}
    Let $\{\partial Z_n\}_{n\in\Z_{>0}}$ be a countable collection of local structures of cardinality greater than one. There exists a local structure $\partial Z^1$ with a unique maximal end and the property that $\partial Z^1\succ\partial Z_n$ for all $n$, and such that there is no $\partial Y$ such that $\partial Z_n\prec\partial Y\prec\partial Z^1$. There also exists a local structure $\partial Z^C$ with infinitely many maximal ends satisfying the same conditions.
\end{proposition}
\begin{proof}
    If the collection $\{\partial Z_n\}_n$ contains an element $\partial Z_N$ such that $\partial Z_n\preceq\partial Z_N$ for all $n$, then we may apply Lemma~\ref{lem: immediate successors} to obtain the result. Otherwise, define $Z^1=\{Z_1\vee\cdots\vee Z_n\}_n\rightarrow 1$. Because $\{\partial Z_n\}_{n\in\Z_{>0}}$ does not contain a unique maximal local structure, there must be at least one end in $\partial Z^1$ which is not of type $\partial Z_n$ for any $n$. Since $\partial Z^1$ is self-similar, it is a local structure by Proposition~\ref{prop: classification of local structures}. Note that $\partial Z_n\preceq\partial Z^1$ for all $n$ by the construction of $Z^1$, and hence $\partial Z^1$ is an upper bound to this collection. To show that it is a minimal upper bound, observe that any end in $\partial Z^1$ is either of type $\partial Z^1$, in which case that end is maximal, or is contained in a copy of one of the $\partial Z_n$, and is therefore of a type which is not larger than the $\partial Z_n$ in which it embeds. Thus, $\partial Z^1$ is a minimal upper bounds of the sequence $\{\partial Z_n\}_{n\in\Z_{>0}}$. An identical argument works for $Z^C:=\{Z_1\vee\cdots\vee Z_n\}_n\rightarrow C$.
\end{proof}

Throughout this section, we have shown that the poset of local structures and that of countable ordinals share many similar properties. A natural question that arises from this is whether the set of countable ordinals is in bijection with the set of local structures. As the next proposition will show, this is only true when the continuum hypothesis, which states that $\aleph_1=2^{\aleph_0}$, is true: the set of countable ordinals is of cardinality $\aleph_1$, and Proposition~\ref{prop: cardinality of local structures} states that there are $2^{\aleph_0}$ local structures. We will need the following lemma.

\begin{lemma}\label{lem: class of incomparables}
    For each $n\in\Z_{\geq 0}$, let $X_n=(\omega^n+1)\rightarrow o(1)$ and let $Y_n=(\omega^n+1)\to C$. If $n\neq m$, then the pair $\partial X_n$ and $\partial X_m$, and the pair $\partial Y_n$ and $\partial Y_m$ are both incomparable in the poset of local structures.
\end{lemma}
\begin{proof}
    For each $n$, there is only one local structure in $\partial X_n$ which is accumulated by genus. Thus, if $\partial X_m\preceq \partial X_n$ for some $m$ and $n$, then the homeomorphic copy of $\partial X_m$ in $\partial X_n$ would need to include maximal points, which would yield that $\partial X_m\cong \partial X_n$. In other words, if the local structures are not incomparable, they must be equal. If $m<n$, then $\partial X_n\not\preceq\partial X_m$ because $\partial X_n$ contains ends of Cantor-Bendixon rank $n$ which are not accumulated by genus and $\partial X_m$ does not. The proof for the $\partial Y_n$ is very similar, replacing the property of being accumulated by genus with being of Cantor type.
\end{proof}

\begin{proposition}\label{prop: cardinality of local structures}
    There are exactly $2^{\aleph_0}$ local structures with countable end space, exactly $2^{\aleph_0}$ local structures of genus zero, and exactly $\aleph_1$ local structures which fall in both categories. Moreover, there are exactly $2^{\aleph_0}$ stable locally finite graphs.
\end{proposition}
\begin{proof}
    Every locally finite graph has countably many vertices, and thus the vertices can be enumerated by $\Z_{\geq0}$. This implies that a locally finite graph can be represented as a symmetric function $g\colon\Z_{\geq0}\times\Z_{\geq0}\to\{0,1\}$ where two nodes $a$ and $b$ have an edge between them if and only if $g(a,b)=1$. Thus, the number of locally finite graphs, stable or not, is bounded above by $2^{\Z_{\geq0}\times\Z_{\geq0}}=2^{\aleph_0}$.

    We next establish a bound from below. To do this, we construct a subset of $2^{\Z_{\geq 0}}$ which is in bijection with $\R$ such that any two elements $A$ and $B$ of this subset have the property that both $A\setminus B$ and $B\setminus A$ are non-empty. Fix a bijection $f\colon\Q\rightarrowtail \hspace{-1.9ex} \twoheadrightarrow\Z_{\geq 0}$, and for each $r\in\R$, define $S_r:=f([r,r+1]\cap\Q)$. Thus the set $\{S_r\,:\,r\in\R\}$ satisfies the desired properties. Let $\pi_r\colon\Z_{\geq 0}\rightarrowtail \hspace{-1.9ex} \twoheadrightarrow S_r$ be a bijection. Let $X_{S_r}:=\{X_{\pi_r(0)}\vee\cdots\vee X_{\pi_r(n)}\}_n\to 1$ and let $Y_{S_r}:=\{Y_{\pi_r(0)}\vee\cdots\vee Y_{\pi_r(n)}\}_n\to 1$, where $X_n$ and $Y_n$ are the graphs defined in Lemma~\ref{lem: class of incomparables}. Both of these graphs are stable because all the $X_i$ and $Y_i$ graphs are stable. We claim that for $r\neq r'$, the pair $\partial X_{S_{r}}$ and $\partial X_{S_{r'}}$ and the pair $\partial Y_{S_{r}}$ and $\partial Y_{S_{r'}}$ are incomparable. Indeed if $p\in S_r\setminus S_{r'}$ then $\partial X_p\preceq\partial X_{S_r}$ but $\partial X_p\not\preceq\partial X_{S_{r'}}$ so $\partial X_{S_{r'}}\not\preceq\partial X_{S_r}$. We also get that $\partial Y_{S_{r'}}\not\preceq\partial Y_{S_r}$ by a similar argument. A symmetric argument for $q\in S_{r'}\setminus S_r$ shows that $\partial X_{S_{r}}\not\preceq\partial X_{S_{r'}}$ and $\partial Y_{S_{r}}\not\preceq\partial Y_{S_{r'}}$. Thus we have found $2^{\aleph_0}$ stable local structures with countable end space, namely the set $\{X_{S_r}\,:\,r\in\R\}$, and $2^{\aleph_0}$ local structures of genus zero, namely the set $\{Y_{S_r}\,:\,r\in\R\}$. These local structures are stable, so we have also established the moreover statement.

    For the overlap statement, any local structure which has genus zero and countable end space, by Theorem~\ref{MS 1}, is of the form $\partial(\omega^\alpha+1)$, where $\alpha$ is some countable ordinal. Thus, these local structures are indexed by countable ordinals, of which there are $\aleph_1$ many.
\end{proof}

\subsection{Wedge decomposition}\label{sec: wedge decomposition}

In this section, we prove the existence and uniqueness of a natural wedge decomposition for graphs whose maximal ends are all stable. A version of this decomposition for the end space of a surface was done by Schaffer-Cohen in \cite[Lemma~4.5]{SC} (and \cite{BV}). We start by showing that locally finite graphs with stable maximal ends have finitely many maximal ends.

\begin{lemma}[Finitely many maximal end types] \label{lem: stable -> finite max end types}
    If every maximal end of $X$ is stable, then $X$ has finitely many maximal end types.
\end{lemma}
\begin{proof}
    Suppose for contradiction that $X$ has infinitely many maximal end types. Then there exists a sequence of maximal ends $\{\mu_n\}_n$ such that $\mu_n\not\sim\mu_m$ for $n\neq m$. Since $\partial X$ is compact, there is an accumulation point $\nu$ of the sequence. By Proposition~\ref{MR 4.7}, there is a maximal end $\mu$ such that $\nu\preceq\mu$, and so therefore there is a subsequence $\{\mu_{n_k}\}_k$ of $\{\mu_n\}_n$ and $\nu_k\sim\mu_{n_k}$ for each $k$ such that $\nu_k\to\mu$ as $k\to\infty$. By the assumption that all the $\mu_n$ are of different types, we may assume that $\mu_{n_k}\not\sim\mu$ for all $k$ by passing to a subsequence if necessary. This means that $\mu$ is not an accumulation point of $E(\mu_{n_k})$ for any $k$. But as $\mu$ is maximal, it is stable too, which means that there must exist a $k$ such that for all $l>k$, we have that $E(\mu_{n_l})$ intersects any stable neighborhood of $\mu$ non-trivially. This would imply that $\mu$ is an accumulation point of $E(\mu_{n_l})$ for $l>k$, which is a contradiction.
\end{proof}

We now prove a technical lemma which will be useful in Proposition~\ref{prop: wedge decomposition}.

\begin{lemma}\label{lem: stable nbhds of wedge comps}
    Let $X$ be a locally finite graph and $\mu\sim\partial Z$ be a stable maximal end of $X$ which is of Cantor type. Then there exists a neighborhood $U$ of $E(\mu)$ such that for any $\lambda\in E(\mu)$, the set $U$ is a stable neighborhood of $\lambda$.
\end{lemma}
\begin{proof}
    For each $\lambda\in E(\mu)$, let $U_\lambda$ be a stable neighborhood of $\lambda$. By Lemma~\ref{lem: orbits of max ends}, the set $E(\mu)$ is closed, and therefore covered by finitely many of these stable neighborhoods $U_{\lambda_1},...,U_{\lambda_n}$. We claim that $U:=U_{\lambda_1}\cup\cdots\cup U_{\lambda_n}$ is the desired set. We may assume, increasing $n$ if necessary, that $\mu=\lambda_1$ without loss of generality. Because $\partial Z$ is of Cantor type, any neighborhood of a point of type $\partial Z$ contains infinitely many other points of type $\partial Z$. Thus by Lemma~\ref{lem: 4.17}, there are non-intersecting homeomorphic copies of $U_{\lambda_2},...,U_{\lambda_n}$ in $U_{\lambda_1}$ which do not contain $\lambda_1$. By Lemma~\ref{lem: 4.17} we get that $U_{\mu}\cong U_{\lambda_1}\cup\cdots\cup U_{\lambda_n}$. To show that $U$ is a stable neighborhood of any other $\lambda\in E(\mu)$, we may follow the same argument as above, increasing $n$ if necessary, to assume that $\lambda=\lambda_1$.
\end{proof}

We are now ready to state the main result of this section.

\begin{proposition}[Wedge decomposition]\label{prop: wedge decomposition}
    Let $X$ be a locally finite graph with stable maximal ends and $g(X)\in\{0,\infty\}$. Then $X\cong X_1\vee\cdots\vee X_n$, where $\partial X_i$ is a local structure for each $i$, and for $i\neq j$, either $\partial X_i\cong\partial X_j$ and both have a unique maximal end, or they are incomparable. This decomposition is unique up to proper homotopy equivalence and ordering.
\end{proposition}
\begin{proof}
    Decompose the set of maximal ends into $M_1\sqcup\cdots\sqcup M_n$, where each $M_i$ consists of either a singleton or the orbit of a maximal end type which is homeomorphic to Cantor space. Then each $M_i$ is closed, and thus we may apply Lemmas~\ref{lem: small clopen neighborhoods} and~\ref{lem: stable nbhds of wedge comps} to produce clopen sets $U_i$ containing $M_i$ for each $1\leq i\leq n$ such that $U_i\cap M_j=\0$ for $i\neq j$ and such that each $U_i$ is a stable neighborhood of each $\mu\in M_i$. Moreover, by Lemma~\ref{lem: 4.17}, we may assume that the $U_i$ are disjoint. Let $U=\partial X\setminus(U_1\cup\cdots\cup U_n)$, and note that $U$ is clopen, and hence compact.

    We now show that $U_1\cup\cdots\cup U_n\cup U\cong U_1\cup\cdots\cup U_n$. For each $\nu\in U$, there exists a maximal end $\mu$ contained in one of the $U_i$, say $U_1$, such that $\nu\prec\mu$. It follows from the order on end spaces that there is a clopen homeomorphic copy of a neighborhood $U_\nu$ of $\nu$ inside of $U_1$. Thus $U_1\cup U_\nu\cong U_1$ by Lemma~\ref{lem: 4.17}. After obtaining a clopen cover $\{U_\nu\}_{\nu\in U}$ of $U$ in a similar form and passing to a finite sub-cover $U_{\nu_1},...,U_{\nu_m}$, and after shrinking some of these sets to ensure disjointness, we may follow a similar argument and obtain a homeomorphic copy of $U_{\nu_1}\sqcup\cdots\sqcup U_{\nu_m}$ in $U_1\cup\cdots\cup U_n$. Thus there is a homeomorphic copy of $U=U_{\nu_1}\cup\cdots\cup U_{\nu_m}$ in $U_1\cup\cdots\cup U_n$, and so $\partial X\cong U_1\sqcup\cdots\sqcup U_n$. Applying Lemma~\ref{lem: wedges and clopens}, we obtain that $X\cong X_1\vee\cdots\vee X_n$, where $\partial X_i\cong U_i$ for each $i$.

    We now argue for uniqueness. The $M_i$ were constructed in such a way that $n$ is the number of maximal end types of Cantor type plus the number of discrete maximal ends. Any decomposition of $X$ into the wedge product of a larger number of maximal local structures would thus have two wedge components equal to each other of Cantor type. These two local structures would then be able to be combined. Conversely, the existence of a decomposition of $X$ into the wedge product of a smaller number of maximal local structures would give a contradiction to the number of maximal end types and number of discrete maximal ends being unique. Thus, $n$ is unique.

    Lastly, suppose that $X_1\vee\cdots\vee X_n$ and $Y_1\vee\cdots\vee Y_n$ are two wedge decompositions of $X$. Up to relabeling, we can assume that $X_i$ and $Y_i$ intersect the same maximal end orbit in each decomposition for each $i$. As $\partial X_i$ and $\partial Y_i$ are both local structures, if $\mu\in \partial X_i$ is maximal, it is maximal in $\partial Y_i$ as well. Thus $\mu\sim \partial X_i$ and $\mu\sim \partial Y_i$. Then $\partial X_i\cong\partial Y_i$ by Lemma~\ref{lem: translate local structures to ends} and anti-symmetry in Proposition~\ref{prop: poset of end types}.
\end{proof}

\begin{definition}
    We refer to the decomposition $X\cong X_1\vee\cdots\vee X_n$ in Proposition~\ref{prop: wedge decomposition} as the \emph{wedge decomposition of $X$}.
\end{definition}

We give two corollaries of Proposition~\ref{prop: wedge decomposition}. The following corollary contextualizes the discussion on wedge decompositions in the framework of Proposition~\ref{prop: classification of local structures}.

\begin{corollary}\label{cor: self-similar wedge decomp}
    Let $X$ be a locally finite graph with stable maximal ends such that $g(X)\in\{0,\infty\}$. Then $X$ is self-similar if and only if the wedge decomposition of $X$ consists of only one wedge component.
\end{corollary}
\begin{proof}
    As the end space of each wedge component of $X$ is a local structure, this follows immediately from Proposition~\ref{prop: classification of local structures}.
\end{proof}

We next give a corollary which allows us to extend the poset of local structures to one of the set of all locally finite graphs with stable maximal ends. It says that, because of the existence of a wedge decomposition, it is possible to determine if one end space $\partial X$ with stable maximal ends embeds as a clopen set into another $\partial Y$ based on whether $\partial Y$ contains the maximal local structures of $\partial X$.

\begin{corollary}\label{cor: stable lfg poset}
    Let $X$ and $Y$ be locally finite graphs, and suppose that $X$ has stable maximal ends. Then $\partial X$ embeds as a clopen subset of $\partial Y$ if and only if there exists $\nu_1,...,\nu_n\in\partial Y$ such that $\nu_i\sim \partial X_i$ for all $1\leq i\leq n$, where $X=X_1\vee\cdots\vee X_n$ is the wedge decomposition of $X$.
\end{corollary}
\begin{proof}
    Suppose that $\partial Y$ contains a clopen homeomorphic copy of $\partial X$. Then $\partial Y$ contains ends of type $\partial X_i$ for $1\leq i\leq n$, and stable neighborhoods of those points will be homeomorphic to $\partial X_i$ for $1\leq i\leq n$. The converse is immediate.
\end{proof}

We next give an example to show that the assumption in Proposition~\ref{prop: wedge decomposition} that maximal ends are stable is necessary. We define the \emph{Keystone locally finite graph} and show that it contains two maximal ends but no wedge decomposition.

\begin{example}[The Keystone locally finite graph]\label{ex: keystone lfg}
    Consider the set $\{X_n\}_n$ of locally finite graphs constructed in Lemma~\ref{lem: class of incomparables} and a bi-infinite line whose vertices are labeled in an order preserving bijection with $\Z$. Call the end in the positive direction $\mu_O$ and the end in the negative direction $\mu_E$. At vertex $1$, attach $X_2\vee X_1$, at vertex $2$ attach $X_4\vee(X_1\vee X_3)$, and at vertex $n>0$ attach $X_{2n}\vee(\bigvee_{i=1}^{n} X_{2i-1})$. At vertex $-1$, attach $X_1\vee X_2$, at vertex $-2$ attach $X_3\vee(X_2\vee X_4)$, and at vertex $-n$ attach $X_{2n-1}\vee(\bigvee_{i=1}^{n} X_{2i})$. The end $\mu_E$ dominates $\partial X_n$ for $n$ even, and $\mu_O$ dominates $\partial X_n$ for $n$ odd. See Figure~\ref{fig: keystone lfg}.

    \begin{figure}[h]
        \centering
            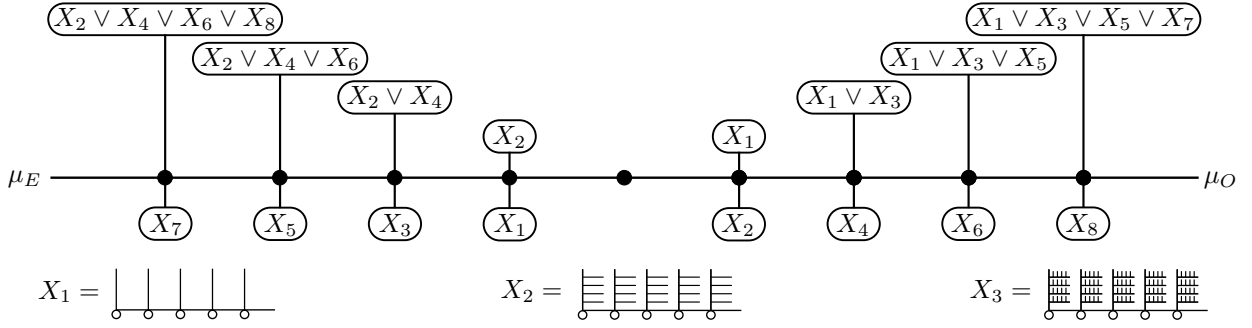
        \centering
        \caption{Keystone locally finite graph.}
        \label{fig: keystone lfg}
    \end{figure}
    
    The ends $\mu_E$ and $\mu_O$ are the only maximal ends. For any $x\in X$ such that $\mu_E$ and $\mu_O$ lie in in different components of $X\setminus\{x\}$, the $X_n$ for $n$ odd are maximal in the component containing $\mu_E$ and vice versa for the component containing $\mu_O$. Thus any decomposition of this locally finite graph as a wedge decomposition which separates the maximal ends results in two locally finite graphs, each with infinitely many maximal ends. Intuitively, to keep the number of maximal ends finite, the two maximal ends need each other similar to how a keystone keeps an arch from crumbling. Note that the Keystone graph also provides an example of a non-stable locally finite graph with finitely many maximal end types, showing that the converse of Lemma~\ref{lem: stable -> finite max end types} is false.
\end{example}

\medskip

We defined wedge decompositions among graphs with stable maximal ends by using local structures, but we can extend this definition to arbitrary graphs by defining a wedge decomposition of a locally finite graph $X$ to be $X_1\vee\cdots\vee X_n$, where each $X_i$ contains either a unique maximal end which is not of Cantor type or a Cantor space of maximal ends of the same type. It would be interesting to determine a necessary condition for the existence of wedge decompositions among arbitrary graphs. For example, any graph with a unique maximal end type of size one or infinity has a wedge decomposition consisting of a single component. Conversely, the Keystone graph does not have a wedge decomposition. The existence of a wedge decomposition can be thought of as a weakening of stability.

\subsection{Ordered signatures}\label{sec: ordered signatures}

Proposition~\ref{prop: building blocks} constructs signatures for arbitrary locally finite graphs, albeit the signatures constructed do not give very much information about the structure of the graphs. For example, Proposition~\ref{prop: building blocks} constructs signatures such as $(\{Y_n\}_n,a)\to\mathcal{C}$, where $\mathcal{C}$ is an infinite binary splitting tree, $Y_0\cong R_0$, $Y_n\cong \partial 1$ for $n>0$, and $a(v)=1$ along a line in $\mathcal{C}$ and equals zero elsewhere. It is not immediately obvious for signatures such as these how many maximal ends types there are and how to formulate the local structures of the stable ends in this graph. See Figure~\ref{fig: bad signature} for the graph and for a ``better'' signature. In this section, we define \emph{ordered signatures}, which are signatures giving as much information about the structure of ends as possible, and show that every stable locally finite graph has an ordered signature. The section culminates with Theorem~\ref{thm: ordered signatures}, which generalizes Theorem~\ref{thm: ordered signatures intro}.

    \begin{figure}[h]
        \centering
            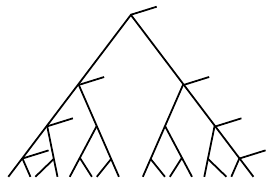
        \centering
        \caption{The graph $(\{Y_n\}_n,a)\to\mathcal{C}$, which can also be written as $((1\vee C)\to1)\vee ((1\vee C)\to1)$. In the latter signature, the self-similar locally finite graphs $1$, $C$, and $(1\vee C)\to 1$, whose end spaces are the three local structures in the depicted graph, all appear as wedge components.}
        \label{fig: bad signature}
    \end{figure}

\begin{definition}\label{def: ordered signature}
    Let $X$ be a stable locally finite graph. A signature for $X$ is \emph{ordered} if the following conditions hold. We give examples in bullet points.
    \begin{enumerate}
        \item (Finite genera are clumped) If $X$ has finite positive genus $n$, then $\bigvee_{i=1}^n R_1$ is included as a wedge in the signature.
        \item (Wedges lack order) The signature for $X$ is given by its wedge decomposition.
        \begin{itemize}
            \item $1\vee(1\rightarrow C)$ is not ordered because this is not the wedge decomposition of $X$, but $1\rightarrow C$ is.
        \end{itemize}
        \item (Arrows climbing towards maximality) There is no spread in the signature, only convergence. Also, there are no graphs to the right of arrows, only other signatures. Whenever $\{Y_n\}_n\rightarrow Y$ appears in the signature, then
        \begin{enumerate}
            \item $\partial Y_1\subseteq\partial Y_2\subseteq\cdots$ as clopen subsets;
            \item for all $\mu\in\partial Y$ and $\nu\in\partial Y_n$, we have that $\nu\prec\mu$ for any $n$; and
            \item if $\nu,\nu'\in\partial Y$, then $\nu\sim\nu'$ in the locally finite graph $Y$.
        \end{enumerate}
        Arrows climbing towards maximality will become important later, so we will give more care to the examples:
        \begin{itemize}
            \item $\{(\omega^n+1)\rightarrow o(1)\}_n\rightarrow 1$ violates $(a)$ by Lemma~\ref{lem: class of incomparables}.
            \item $C\rightarrow C$ and $C\rightarrow 1$ violate $(b)$. Instead, $C$ is ordered. The signature $(1\rightarrow C)\rightarrow 1$ is also not ordered for the same reason: an ordered signature is $1\rightarrow C$.
            \item $1\rightarrow (o(1)\vee 1)$ violates (c), but $(1\rightarrow 1)\vee(1\rightarrow o(1))$ or $(\omega+1)\vee(1\rightarrow o(1))$ are ordered. Also observe that by (c) the only locally finite graphs which may be immediately to the right of an arrow are $1,o(1),C$ and $o(C)$, which are the ones with minimal end spaces in the poset of local structures.
            \item $\{\omega^n+1\}_n\rightarrow(\{o(\omega^m+1)\}_m\rightarrow o(C))$ violates $(c)$. An ordered signature is $\{Y_n\}_n\rightarrow o(C)$, where $Y_0=\{\omega^m+1\}_m\rightarrow o(1)$, and $Y_{n+1}=Y_n\rightarrow o(1)$ for $n>0$. Note that auxiliary symbols such as $Y_n$ are allowed in ordered signatures.
        \end{itemize}
        \item (Simplicity of genus) If the ends in the locally finite graph to the right of an arrow are not an input to $o(-)$, then no ends to the left of that arrow are accumulated by genus.
        \begin{itemize}
            \item $o(1)\rightarrow 1$ is not ordered because the $1$ to the right of the arrow does not appear inside of an instance of $o(-)$, but $o(1\to 1)$ and $o(1)\rightarrow o(1)$ are ordered.
        \end{itemize}
        \item (Recursion) If a signature contains $Y_1\vee\cdots\vee Y_n$, $Y_1\rightarrow Y$, or $\{Y_n\}_n\rightarrow Y$, then all of the $Y_i$ and $Y$ are ordered signatures.
    \end{enumerate}
\end{definition}

To show that every stable locally finite graph has an ordered signature, we define $\MS(X)$, which captures the structure of the maximal ends of a graph $X$.

\begin{definition}\label{def: MaxShell}
    Let $X$ be a locally finite graph, and let $A=\{\nu\in\partial X\,:\,\nu$ is a maximal end in $X\}$. Suppose $X$ has characteristic pair $(\partial X,\partial X_g)$, where $\partial X_g=\0$ if $g(X)<\infty$. We define $\MS(X)$ to be the locally finite graph with characteristic pair $(\overline{A},\overline{A}_g)$; this is well defined by Theorem~\ref{AK-B 2.2}. We may refer to $\MS(X)$ as a subgraph of $X$, by which we mean the smallest subgraph of $X$ whose end space is precisely the set of maximal ends of $X$ and whose ends which are accumulated by genus in $X$ are accumulated by genus in $\MS(X)$.
\end{definition}

We took the closure of $A$ in this definition because the set of maximal ends may not be closed. For example, consider the locally finite graph $\{(\omega^n+1)\rightarrow C\}_n\rightarrow(\omega+1)$. The maximal ends in this locally finite graph are precisely the ends of Cantor type along with the maximal end of $\omega+1$. In particular, the other ends of $\omega+1$ are not maximal, but are limits of ends of Cantor type.

The next lemma shows that any graph whose end space is a local structure has a signature with an arrow climbing towards maximality (see Definition~\ref{def: ordered signature}).

\begin{lemma}[The invocation of convergence]\label{lem: invocation of ordered spread}
    If $\partial Z$ is a local structure, then 
    $$\MS(Z)\in\{1,o(1),C,o(C)\}$$
    and there exist locally finite graphs $\{Y_n\}_n$ with end spaces such that $\partial Y_1\subseteq \partial Y_2\subseteq\cdots$ and such that 
    $$\partial Z\cong\partial (\{Y_n\}_n\rightarrow\MS(Z)).$$
\end{lemma}
\begin{proof}
    Because $\partial Z$ is a local structure, all maximal ends in $\partial Z$ are of the same type. In particular, there is either one or infinitely many maximal ends, all or none of which are accumulated by genus. This implies that $\MS(Z)\in\{1,o(1),C,o(C)\}$. By Lemma~\ref{lem: orbits of max ends}, $\partial \MS(Z)$ is closed in $\partial Z$. Thus, by Lemma~\ref{lem: invocation of spread}, there is a proper homotopy equivalence $Z\cong(\{X_n\}_n,a)\rightarrow\MS(Z)$ for some bijective $a \colon \mathcal{V}(\MS(Z))\rightarrow\Z_{>0}$.

    The next step of the proof is to show that we can write $Z$ without invoking spread. First, we define some terms. Let $v\in \mathcal{V}(\MS(Z))$, let $V_n:=\{w\in \mathcal{V}(\MS(Z))\,:\,d(v,w)=n)\}$, and let $h\colon\Z_{\geq0}\rightarrow\Z_{\geq0}$ be given by $h(n)=|V_n|$. For each $n$, label the elements of $V_n$ as $\{w_{n,1},...,w_{n,h(n)}\}$. Note that $V_0=\{w_{0,1}\}=\{v\}$. 
    
    Let $X$ be a subgraph of $Z$ such that $\partial X$ is clopen in $\partial Z\setminus\partial\MS(Z)$, and let $w_{k,i}\in\mathcal{V}(\MS(Z))$. The assumption that the maximal ends of $Z$ are stable implies that there are clopen homeomorphic copies of $\partial X$ arbitrarily close to any maximal end in $\partial Z$. Hence, there exists a copy of $\partial X$ which is ``further away from $v$ than $w_{k,i}$,'' i.e., such that any geodesic ray starting at $v$ which limits to an end in $\partial X$ passes through $w_{k,i}$. We define a proper homotopy equivalence, which we call \emph{pushing $X$ up to $w_{k,i}$}, that moves this copy of $\partial X$ so that it is attached as a wedge component to $w_{k,i}$; see Figure~\ref{fig: pushing up}. We then redefine the sequence $\{X_n\}_n$ so that $X_{a(w_{k,i})}$ has an additional wedge component proper homotopy equivalent to $X$ and so that all $X_n$ which intersected the copy of $X$ which was pushed up have appropriate ends taken away. This ensures that the redefined $X_n$ reflect the push up. There are infinitely many proper homotopy equivalences that would constitute pushing $X$ up to $w_{k,i}$ because there are copies of $\partial X$ arbitrarily close to any maximal end. For the purpose of this proof, any such proper homotopy equivalence will suffice.

    \begin{figure}[t]
        \centering
        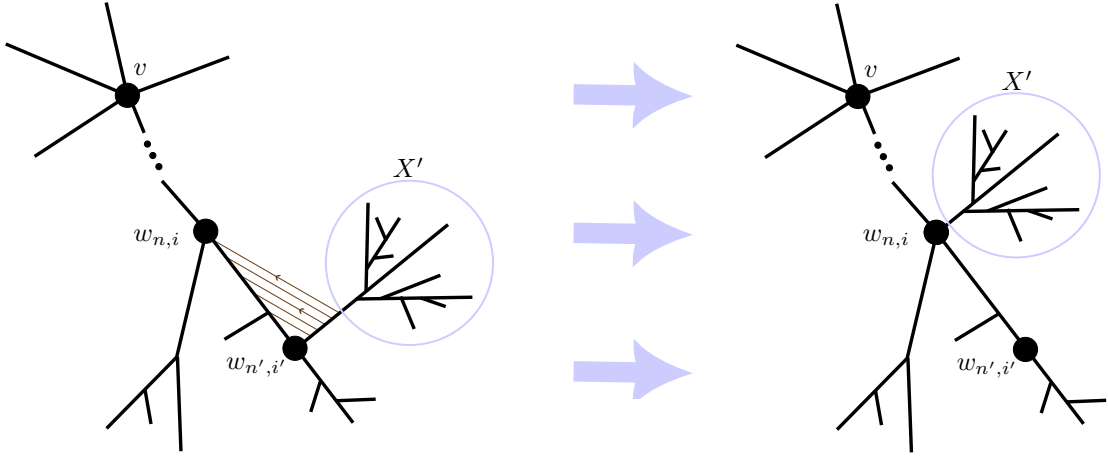
        \centering
        \caption{Example of pushing $X$ up to $w_{n,i}$. The locally finite graph $X'$ in the periwinkle circle, which is proper homotopy equivalent to $X$, goes from being attached to the vertex $w_{n',i'}$, which is further than $w_{n,i}$ from $v$, to being attached to $w_{n,i}$ via a proper homotopy equivalence which collapses a line segment between $w_{n',i'}$ and $X'$ to a line segment from $w_{n',i'}$ to $w_{n,i}$.} 
        \label{fig: pushing up}
    \end{figure}

    Note that pushing $X$ up to $w_{k,i}$ induces the identity map on the end space $\partial Z$, on $\MS(Z)$, and on the component of $Z\setminus V_k$ containing $v$. Hence a sequence of proper homotopy equivalences which push locally finite graphs up to vertices which are increasingly further from $v$ converges to a mapping class in $\Maps(Z)$, by completeness. We will use this fact to complete the proof of the lemma.

    \begin{claim}\label{cl: construction of the phe}
        The locally finite graph $(\{X_n\}_n,a)\rightarrow\MS(Z)$ is proper homotopy equivalent to one with a signature of the form $\{Y_n\}_n\rightarrow(\MS(Z),v)$, where $\partial Y_1\subseteq\partial Y_2\subseteq\cdots$.
    \end{claim}
    \begin{proof}
        We construct a proper homotopy equivalence from 
        $$(\{X_n\}_n,a)\rightarrow\MS(Z)$$ 
        to 
        $$\{Y_n\}_n\rightarrow\MS(Z)$$
        by inductively constructing a sequence of proper homotopy equivalences which will converge to the desired proper homotopy equivalence.
        
        First, let $Y_0=\{*\}$. Similarly to the definition of pushing up, we may apply a proper homotopy equivalence of $Z$ which induces the identity on $\partial Z$ so that $X_{a(v)}$ is wedged to a vertex in $V_1$ and $a(v)=0$. We essentially ``push $X_{a(v)}$ down'' to a vertex in $V_1$. Having done this, we get that $X_{a(v)}=Y_0$.

        For the induction step, we define $Y_k$ satisfying $\partial Y_0\subseteq\partial Y_1\subseteq\cdots \subseteq\partial Y_{k}$, and construct a proper homotopy equivalence from 
        $$(\{Y_0,...,Y_{k-1},X_k,X_{k+1},...\},a)\rightarrow\MS(Z)$$ to 
        $$(\{Y_0,...,Y_{k},X_{k+1},X_{k+2},...\},a')\rightarrow\MS(Z),$$ 
        where for all $i$, $a'(w_{l,i})=l$ when $l\leq k$. Let $Y_k:=(\bigvee_{i=1}^{h(k)} X_{a(w_{k,i})})\vee Y_{k-1}$, and for each $1\leq i \leq h(k)$, push 
        $$X_{a(w_{k,1})}\vee\cdots \vee X_{a(w_{k,i-1})}\vee X_{a(w_{k,i+1})}\vee\cdots \vee X_{a(w_{k,h(k)})}\vee Y_{k-1}$$ up to $w_{k,i}$; note that this graph is $Y_k$ with the $X_{a(w_k,i)}$ component removed. Then we have that $\partial Y_{k-1}\subseteq\partial Y_k$. Let $a'(w_{l,i})=l$ for $l\leq k$ so that $a'$ agrees with distance from $v$ for vertices up to distance $k$.

        The sequence of such proper homotopy equivalences converges to a proper homotopy equivalence, because the elements of the sequence agree on an exhaustion of compact sets. The limit has the property that for all vertices $w_{k,i}\in \mathcal{V}(\MS(Z))$, there is a copy of $Y_k$ attached to $Z$ at $w_{k,i}$, as desired.
    \end{proof}

    Lastly, given $v,v'\in\mathcal{V}(\MS(Z))$, there exists a proper homotopy equivalence from 
    $$\{Y_n\}_n\rightarrow(\MS(Z),v)$$
    to 
    $$\{Y_n\}_n\rightarrow(\MS(Z),v')$$
    by a similar argument to the one in Claim~\ref{cl: construction of the phe}. This concludes the proof of the lemma.
\end{proof}

Lemma~\ref{lem: invocation of ordered spread} shows that every stable locally finite graph has a signature without spread or base-points for convergence arrows. There do exist non-stable locally finite graphs, however, that have signatures not involving spread or base-points for arrows. For example, 
$$\{(\omega^n+1)\rightarrow o(1)\}_n\rightarrow o(1)$$ 
is not stable for the following reason. The set of end types $E(\partial((\omega^n+1)\rightarrow o(1)))$ are pairwise incomparable by Lemma~\ref{lem: class of incomparables}. They are also of cardinality one, and thus, by Lemma~\ref{lem: orbits of max ends}, each end of type $\partial((\omega^n+1)\rightarrow o(1))$ is maximal. As there are infinitely many maximal end types, the graph must be unstable by Lemma~\ref{lem: stable -> finite max end types}.

We are now ready to prove Theorem~\ref{thm: ordered signatures intro}, which we state as a three-way equivalence below.

\begin{theorem}\label{thm: ordered signatures}
    For a locally finite graph $X$, the following are equivalent.
    \begin{enumerate}
        \item $X$ is stable.
        \item $X$ has an ordered signature.
        \item $X$ has a signature without spread and such that whenever $\{Y_n\}_n\rightarrow Y$ appears in this signature, $Y$ is stable and there is a homeomorphic clopen copy of $\partial Y_i$ in $\bigsqcup_{n=i+1}^\infty \partial Y_n$  for all but finitely many $i$.
    \end{enumerate}
\end{theorem}
\begin{proof}
    We first show the implication $1\Rightarrow2$. Let $X$ be a stable locally finite graph. By Proposition~\ref{prop: wedge decomposition}, $X$ has a wedge decomposition $X=W_1\vee\cdots \vee W_n$, where $W_n=\bigvee_{i=1}^{g(X)}R_1$ if $0<g(X)<\infty$. It suffices to show that each $W_i$ has an order-preserving signature. Thus, without loss of generality, we may assume that $X$ has a unique maximal end type of size one or infinity. This will be a recursive proof, and so we let $X_1:=X$. If $X_1\cong C$, $o(C)$, $\omega^\alpha+1$, or $o(\omega^\alpha+1)$ for a countable ordinal $\alpha$, we are done. Otherwise, by Lemma~\ref{lem: invocation of ordered spread}, $X_1\cong\{Y_n\}_n\rightarrow\MS(X_1)$ where $\MS(X_1)$ is a minimal local structure. Note that if all the $Y_n$ are ordered signatures, then the signature of $X_1$ is ordered too. Each $Y_n$ has a wedge decomposition, so we let $\partial X_2$ be a maximal local structure in some $Y_n$. Note that $\partial X_2\prec \partial X_1$. By Proposition~\ref{prop: descending chain condition}, this process will eventually terminate, although not necessarily in a uniform number of steps across all $Y_n$. Therefore, $X$ has an ordered signature.

    The implication $2\Rightarrow3$ follows directly from the definition of ordered signatures.

    The rest of the proof shows the implication $3\Rightarrow1$. We first prove the following claim.

    \begin{claim}\label{cl: pass left of arrows}
        Let $\{Y_n\}_n$ be a sequence of stable locally finite graphs, and let $Y$ be another stable locally finite graph. If there is a homeomorphic clopen copy of $\partial Y_n$ in $\bigsqcup_{k=n+1}^\infty \partial Y_k$ for all but finitely many $n$, then $\{Y_n\}_n\to Y$ is stable.
    \end{claim}
    \begin{proof}
        Since any end in one of the $\partial Y_n$ is stable in $\partial Y_n$, it must be stable in $\{Y_n\}_n\to Y$, as $\partial Y_n$ is clopen in $\{Y_n\}_n\to Y$. Thus, it suffices to check that an arbitrary end $\nu$ in $\partial Y\subset\partial(\{Y_n\}_n\to Y)$ is stable. By assumption, there exists a number $N$ such that for $n>N$, all $\partial Y_n$ embed homeomorphically into $\bigsqcup_{k=n+1}^\infty \partial Y_k$. Observe that $Y\cong \{Z_m\}_m\rightarrow (1,x_0)$ for some sequence $\{Z_m\}_m$ and choice of base point $x_0$, where the $1$ represents $\nu$. By stability of $Y$ and by increasing $N$ if necessary, there is a homeomorphic copy of $\partial Z_m$ in $\bigsqcup_{k=m+1}^\infty \partial Z_k$ for all $m>N$.

        We have that $\{Y_n\}_n\rightarrow Y\cong\{Y_n\}_n\rightarrow((\{Z_m\}_m\rightarrow (1,x_0)),y_0)$, which is proper homotopy equivalent to $\{\{Y_{n+k_l}\}_n\rightarrow (Z_l,a_l)\}_l\rightarrow (1,z_0)$, where the $1$ represents $\nu$ and $\{k_l\}_l$ is the minimal $n$ such that $Y_n$ is attached to each $Z_l$ (see Figure~\ref{fig: signature stability criterion} for an example). Increasing $N$ if necessary, we may assume that for $l>N$, $k_l$ is strictly increasing. Then for $l>N$, $\partial (\{Y_{n+k_l}\}_n\rightarrow Z_l)\subseteq\bigsqcup_{j=l+1}^\infty \partial (\{Y_{n+k_{j}}\}_n\rightarrow Z_{j})$. This implies stability of $\nu$ in $\{Y_n\}_n\rightarrow Y$.
    \end{proof}
    
    By definition of signatures, any end in $\partial X$ in any signature is either represented in a sub-graph which is to the right of an arrow, or one of $\omega^\alpha+1$, $o(\omega^\alpha+1)$, $C$, or $o(C)$. Ends in the first case are stable in $\partial X$ by Claim~\ref{cl: pass left of arrows}, and ends in the latter cases are stable from earlier discussions. Because there are no ends to the right of infinitely many arrows, we may apply the above argument as many times as we need until the sub-graph with $\nu$ on the right is a wedge component of $X$. At this point, because stability is a local condition, this gives that $\nu$ is stable in $X$.
\end{proof}

The following corollary follows immediately from Theorem~\ref{thm: ordered signatures}.

\begin{corollary}\label{cor: stable -> no graphs}
    If a locally finite graph $X$ is a stable, then it has a signature which only involves other signatures, and not other graphs.
\end{corollary}
\begin{proof}
    This follows immediately from Theorem~\ref{thm: ordered signatures}, given the definition of ordered signatures.
\end{proof}

The converse of Corollary~\ref{cor: stable -> no graphs} is not true, namely, there exists signatures that do not contain other graphs which unambiguously represent a unique graph that is not stable. For example, consider $\{Y_n\}_n\to 1$, where $Y_n=(\omega^n+1)\to C$. The end represented by $1$ in this signature is non-stable, because the $Y_n$ are incomparable by Lemma~\ref{lem: class of incomparables}.

It is possible to define ordered signatures in a way which does not allow for the genus operation $o(-)$ and treats $R_1$ with the same rules as a minimal end type. For example, an ordered signature for $o(1)$ is $R_1\to 1$. Under this system, for any local structure $Z$, $\MS(Z)\in\{1,C\}$. While this system eliminates more redundancies in Definition~\ref{def: signature}, it often results in graphs whose ordered signatures are much more difficult to write because of Definition~\ref{def: ordered signature} 3(c). For example, an ordered signature of $o(\omega^\omega+1)$ is $\{X_n\}_n\to 1$, where $X_0\cong R_1\to 1$ and $X_{n+1}\cong X_n\to 1$ for all $n>0$.

\begin{figure}[t]
    \centering
    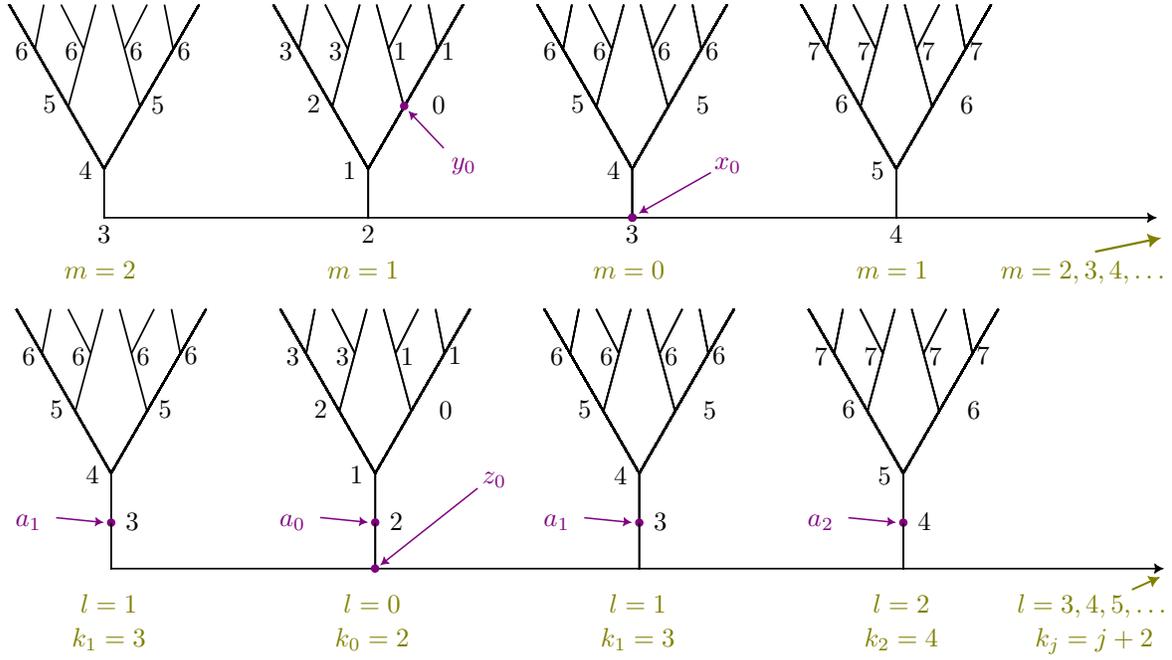
    \centering
    \caption{The graphs $\{Y_n\}_n\rightarrow((\{Z_m\}_m\rightarrow(1,x_0)),y_0)$ (on the top) and $\{\{Y_{n+k_l}\}_n\rightarrow (Z_l,a_l)\}_l\rightarrow (1,z_0)$ (on the bottom) in Theorem~\ref{thm: ordered signatures} when $Y\cong C$, and hence $Z_m\cong C$ for all $m$. In the top graph, the index of the $Y_n$ attached to a given vertex, given by a number in black, is equal to the distance from $y_0$ to that vertex. The value of $m$ at a given Cantor tree is given by the distance from that tree to $x_0$. In the bottom graph, at a given Cantor tree, the value of $l$ is equal to the distance from that tree to $z_0$, and the value of $k_l$ is the index of the $Y_n$ attached to $a_l$. These two graphs are proper homotopy equivalent via an expansion of the wedge points in the top graph. Also, the sequence $\{k_l\}_l$ is eventually strictly increasing.}
    \label{fig: signature stability criterion}
\end{figure}

\section{Dense conjugacy classes in $\Maps(X)$}\label{sec: Rokhlin}

We now turn our attention to the question of which locally finite graphs $X$ are such that $\Maps(X)$ has a dense conjugacy class. As mentioned in the introduction, there exists a complete classification in the surface setting due to \cite{LV} and \cite{HHMRSV}.

\begin{theorem}[Rokhlin property for surfaces]\label{thm: Rokhlin property for surfaces}
    The mapping class group of a connected orientable $2$-manifold has the Rokhlin property if and only if the manifold is either the $2$-sphere or a non-compact manifold whose genus is either zero or infinite and whose end space is self-similar with a unique maximal end.
\end{theorem}

It is not common for $\Maps(X)$ to have a dense conjugacy class, and thus, much of Section~\ref{sec: Rokhlin} will be focused on showing obstructions to having this property. Given Proposition~\ref{AK-B 4.5}, we will often use Proposition~\ref{AK-B 2.6} by first assuming that $X$ is in standard form. Then we will employ two main tools to show that $\Maps(X)$ does not have a dense conjugacy class. The first is the \emph{joint embedding property}.

\begin{definition}
    A topological group $G$ has the \emph{joint embedding property} (\emph{JEP}) if for all non-empty open sets $U$ and $V$ in $G$, there exists $g\in G$ such that $U\cap V^g\neq\0$, where $V^g=\{gfg^{-1}\,:\,f\in V\}$. In other words, $G$ has the JEP if the action of conjugation on itself is topologically transitive.
\end{definition}

The JEP is useful because if $G$ is a Polish group, as in the case of $\Maps(X)$ by Proposition~\ref{AK-B Polish}, then $G$ has a dense conjugacy class if and only if $G$ has the JEP \cite[Theorem~2.1]{KR}. The topology of $\Maps(X)$ is much more abstract than the topology of the graph $X$ itself, and when we will use the JEP to show that $\Maps(X)$ does not have a dense conjugacy class, the sets $U$ and $V$ are going to be constructed directly using the topology on $X$. Thus, the JEP serves as a dictionary between the topology of $\Maps(X)$ and the topology of $X$.

The second main tool that we will use to show that $\Maps(X)$ does not have a dense conjugacy class is the following lemma.

\begin{lemma}\label{lem: proper open normal subgrp}
    If a topological group contains a proper open normal subgroup, then it does not contain a dense conjugacy class.
\end{lemma}
\begin{proof}
    Let $G$ be a topological group, and $H$ be a proper open normal subgroup. The normality of $H$ implies that conjugacy classes in $G$ are either completely contained in $H$ or completely contained in $G\setminus H$. As the subgroup $H$ is open, so are all of its cosets, implying that $G\setminus H$ is open. Thus, there cannot exist a dense conjugacy class, as dense subsets must non-trivially intersect every open subset.
\end{proof}

In practice, the way we will find a proper open normal subgroup of $\Maps(X)$ is to find a group homomorphism $\Phi$ with an open kernel from $\Maps(X)$ to a discrete group $K$, and apply Lemma~\ref{lem: proper open normal subgrp} to the kernel of $\Phi$. See Section~\ref{sec: finite genus and maximal end type} for examples when $K$ is a symmetric group or Out$(\pi_1(X))$, and Section~\ref{sec: flux maps} for an example when $K=\Z$.

Theorem~\ref{thm: Rokhlin property for surfaces} implies that the mapping class group of a surface with a unique maximal end, with genus zero or infinity, and which has self-similar end space, has a dense conjugacy class. The proof in \cite{LV} of this statement is in fact the same in our setting. We give a sketch of their argument below.

\begin{proposition}[Unique maximal end]\label{prop: unique maximal end}
    Let $X$ be a locally finite graph with $g(X)\in\{0,\infty\}$ such that $X$ has a unique maximal end which is stable. Then $\Maps(X)$ has a dense conjugacy class.
\end{proposition}
\begin{proof}
    \cite{LV} technically shows this result in the setting where the end space is self-similar, howeaver, this is true in our setting by Proposition~\ref{prop: classification of local structures}. We may assume by Proposition~\ref{AK-B 4.5} that $X$ is in standard form. Let $\mu$ be the maximal end of $X$. Given any edge $e$ in the underlying tree of $X$, the complement in $X$ of the interior of $e$ consists of two connected components; let $\Omega_e$ be the component such that $\mu\in\partial\Omega_e$, and let $\Sigma_e$ be the other one. Let $G$ be the subgroup of $\Maps(X)$ consisting of elements with representatives restricting to the identity on some $\Omega_e$ for some edge $e$ in the underlying tree of $X$. The proof begins by proving the following claim.

    \begin{claim}
        The group $G$ is dense in $\Maps(X)$, and given any separating edge $e$ in $X$, there exists $h\in\PHE(X)$ such that $h(\Sigma_e)\subset\Omega_e$.
    \end{claim}
    \begin{proof}
        The existence of $h$ is deduced using self-similarity, and the density of $G$ is proven by showing that for any $[f]\in\Maps(X)$ and $K\subset X$ compact, the open set $[f]U_K$ intersects $G$ non-trivially, where $U_K$ is as constructed in the discussion preceding Proposition~\ref{AK-B 4.7}.
    \end{proof}

    After this claim, it can be shown that $\Maps(X)$ has the JEP, i.e., that for any open sets $U$ and $V$ in $\Maps(X)$, there exist $[h]\in\Maps(X)$ such that $[h]U[h]^{-1}\cap V\neq\0$. It suffices to assume without loss of generality that $U$ and $V$ are of the form $[f_1]U_{K_1}$ and $[f_2]U_{K_2}$ (See Proposition~\ref{AK-B 4.7}) for $[f_1],[f_2]\in\Maps(X)$ and $K_1$ and $K_2$ compact sets of $X$. Moreover, by density of $G$ and because each $[f_i]U_{K_i}$ is a coset, it suffices to assume that $[f_1]$ and $[f_2]$ are in $G$. Let $e$ be a separating edge such that both $[f_1]$ and $[f_2]$ have representatives which restrict to the identity on $\Omega_e$ and such that both $K_1$ and $K_2$ are contained in $\Sigma_e$. Then any mapping class with a representative $h$ such that $h(\Sigma_e)\subset\Omega_e$ witnesses the JEP for $U$ and $V$.
\end{proof}

Proposition~\ref{prop: unique maximal end} shows that it is at least as common for mapping class groups of locally finite graphs to have a dense conjugacy class as it is for surfaces.

\begin{corollary}
    Let Mod$(S)$ be the mapping class group of a connected orientable two manifold $S$, and let $X$ be a graph which is proper homotopy equivalent to a filled-in version of $S$ (i.e. $X$ is a graph corresponding to $S$). If Mod$(S)$ has a dense conjugacy class, then $\Maps(X)$ also does.
\end{corollary}
\begin{proof}
    By \cite{LV}, $S$ must either be a two sphere or have genus either zero or infinite, a unique maximal end, and a self-similar end space. If $S$ is a 2-sphere, then $X$ is proper homotopy equivalent to a point, and hence has trivial mapping class group. Otherwise, $X$ must satisfy the conditions of Proposition~\ref{prop: unique maximal end}, in which case $\Maps(X)$ also has a dense conjugacy class.
\end{proof}

The following proposition provides an example of a graph which does not mirror Theorem~\ref{thm: Rokhlin property for surfaces}.

\begin{proposition}[Cantor tree]\label{prop: Cantor tree}
    $\Maps(C)$ has a dense conjugacy class.
\end{proposition}
\begin{proof}
    Because $C$ is a tree, by Proposition~\ref{AK-B 2.3}, it suffices to show that the homeomorphisms of Cantor space has a dense conjugacy class. This was proven in \cite[Theorem~2.6]{GW}.
\end{proof}

Many of the results in this section will come in two versions: one for ends and one for loops. This is because there are many ways in which immersed loops behave like ends. For example, every mapping class has two induced morphisms which parallel one another: an induced topological homeomorphism on the space of ends, and an induced group homomorphism on the fundamental group of $X$ after choosing a base point. In Section~\ref{sec: flux maps}, we will describe a precise way in which immersed loops behave like minimal ends in the poset of ends. From this perspective, Lemma~\ref{lem: minimal local structures} says that the set of minimal elements of the poset of local structures is precisely $\{\partial 1, R_1,C\}$, as $o(1)$ can be thought of as a ray dominating loops. This is motivation for requiring local structures to have genus zero or infinity: when genus is positive and finite, by Lemma~\ref{lem: orbits of max ends}, the loops function similar to maximal ends. 

Another motivation to think of immersed loops as minimal ends comes from the following lemma.

\begin{lemma}\label{lem: end classification}
    Let $\nu\in\partial X$ be an end in some locally finite graph. Then at least one of the following are true.
    \begin{enumerate}
        \item $\partial 1\preceq\nu$
        \item $\nu\in\partial X_g$
        \item $\nu\sim\partial C$
    \end{enumerate}
\end{lemma}
\begin{proof}
    Suppose neither 1 nor 2 hold. Because $\nu\not\in\partial X_g$ and because $\partial X_g$ is a closed subset of $\partial X$, Lemma~\ref{lem: small clopen neighborhoods} implies that there exists a clopen neighborhood $U$ of $\nu$ such that $U\cap\partial X_g=\0$. As $\partial 1\not\preceq\nu$, we may assume that $U$ contains no isolated points by shrinking $U$ if necessary. Then by Brouwer's characterization of Cantor space, the set $U$ must be homeomorphic to Cantor space, giving that $\nu\sim\partial C$.
\end{proof}

\subsection{Finiteness}\label{sec: finite genus and maximal end type}

In this section, we prove Theorem~\ref{thm: obstructions to dense conjugacy classes} (2) and (3). First, we prove a lemma which will be used throughout the rest of the paper.

\begin{lemma}\label{lem: caste system} 
    Let $X$ be a locally finite graph and $E(\nu)$ an equivalence class of ends. Any $[f]\in \Maps(X)$ induces a homeomorphism on $E(\nu)$.
\end{lemma}
\begin{proof} 
    Let $\nu\in E(\nu)$. First, we show that $\nu\preceq f(\nu)$. Let $U$ be an open neighborhood of $f(\nu)$. Then $f^{-1}(U)$ is an open neighborhood of $\nu$. Because $f\circ f^{-1}(U)\subseteq U$, it follows from the definition of the order on ends that $\nu\preceq f(\nu)$. By a symmetric argument, letting $[g]=[f]^{-1}$ and $V$ be an open neighborhood of $\nu$, we get that $\nu\preceq f^{-1}(\nu)$, and so $f(\nu)\preceq\nu$. Therefore $\nu\sim f(\nu)$, and $[f]$ acts on $E(\nu)$. $[f]|_{E(\nu)}$ is continuous by Proposition~\ref{AK-B 2.3}, and it is a homeomorphism because $[f]^{-1}|_{E(\nu)}$ is an inverse map.
\end{proof}

Given the duality between ends and loops discussed above, Lemma~\ref{lem: caste system} can be thought of as the corresponding statement about ends of the fact that any mapping class in $\Maps(X)$ induces an outer automorphism on $\pi_1(X)$. We now prove the following theorem, which is Theorem~\ref{thm: obstructions to dense conjugacy classes} (2) from the introduction.

\begin{theorem}\label{thm: finite end type} 
    Let $X$ be a locally finite graph with an end type $E(\nu)$ such that $1<|E(\nu)|<\infty$. Then $\Maps(X)$ does not have a dense conjugacy class.
\end{theorem}
\begin{proof} 
    Fix an enumeration $E(\nu)=\{\nu_1,...,\nu_n\}$. By Lemma~\ref{lem: caste system}, any map $[f]\in\Maps(X)$ sends $\nu_i\in E(\nu)$ to some $\nu_j\in E(\nu)$. Define $\Phi\colon\Maps(X)\to S_n$, where $S_n$ is the symmetric group on $n$ letters, to be the map corresponding to this action. Namely, let $\Phi$ send $[f]\in\Maps(X)$ to the permutation sending $i\in\{1,...,n\}$ to $j$, where $[f](\nu_i)=\nu_j$. Proper homotopy equivalences inducing the same mapping class in $\Maps(X)$ induce the same homeomorphism on $\partial X$, and so $\Phi$ is well defined. That $\Phi$ is a group homomorphism follows from properties of group actions, and that $\Phi$ is non-trivial follows from Theorem~\ref{MR 1.2}. To see why $\ker(\Phi)$ is open, take a compact subset $K$ of $X$ large enough to separate each end of $E(\nu)$ into different complementary components, and consider the open set $U_K$ as defined in the discussion before Proposition~\ref{AK-B 4.7}. Then $[f]\in[f]U_K\subseteq\ker(\Phi)$ for any $[f]\in\ker(\Phi)$. Thus, the kernel of $\Phi$ is a proper open normal subgroup. By Lemma~\ref{lem: proper open normal subgrp}, $\Maps(X)$ does not contain a dense conjugacy class.
\end{proof}

Continuing the analogy between ends and loops, we next prove Theorem~\ref{thm: obstructions to dense conjugacy classes} (3). This proof follows a similar strategy to that of Theorem~\ref{thm: finite end type}.

\begin{theorem}\label{thm: finite genus}
    Let $X$ be a locally finite graph such that $0<g(X)<\infty$. Then $\Maps(X)$ does not have a dense conjugacy class.
\end{theorem}
\begin{proof}
    Consider the map $\Phi\colon\Maps(X)\to\Out(\pi_1(X))$ by $[f]\mapsto[f]_*$ the induced map on $\pi_1(X)$. If $[f]\in\ker(\Phi)$, then $[f]\in[f]U_K$, where $K$ is a compact subset of $X$ containing the core graph and $U_K$ is as in Proposition~\ref{AK-B 4.7}. The set $[f]U_K$ is open and contained in $\ker(\Phi)$, and thus $\ker(\Phi)$ must be open. Applying Lemma~\ref{lem: proper open normal subgrp}, $\Maps(X)$ does not contain a dense conjugacy class.
\end{proof}

Any finite graph is either a tree or has positive but finite genus. If $X$ is a tree, then $\Maps(X)$ is the trivial group, hence trivially has a dense conjugacy class. Thus, we have the following corollary.

\begin{corollary}
    Let $X$ be a finite graph. $\Maps(X)$ contains a dense conjugacy class if and only if $X$ is a tree.
\end{corollary}

Theorems~\ref{thm: finite end type} and~\ref{thm: finite genus} can be interpreted as showing that finite non-trivial orbits are obstructions to the existence of a dense conjugacy class. Given the results in this section, the only locally finite graphs whose mapping class group possibly contains a dense conjugacy class are those where each maximal end type is homeomorphic to Cantor space or is a singleton, and where the locally finite graph has genus zero or infinity.

\subsection{Maximal ends of Cantor type}\label{sec: Cantor type}

This section discusses when $\Maps(X)$ has a dense conjugacy class in the case where every maximal end type is of Cantor type. Crucial to this section will be the following lemma.

\begin{lemma}\label{lem: swaps}
    Let $X$ be a locally finite graph such that every maximal end is stable and of Cantor type. Then $X\cong X\vee X$.
\end{lemma}
\begin{proof}
    We first show that $X\cong X\vee X$. Suppose that $X$ is self-similar. Decompose $\partial X$ into two open sets $E_1\sqcup E_2$ such that each $E_i$ contains a maximal end. Proposition~\ref{prop: classification of local structures} in combination with Lemma~\ref{lem: 4.17} implies that each $E_i\cong\partial X$, and so $E_1\cong E_2$. Thus $X\cong X\vee X$ by Theorem~\ref{AK-B 2.2}, since they have isomorphic end spaces. If $X$ is not self-similar and $X_1\vee\cdots \vee X_n$ is the wedge decomposition of $X$, each $X_i$ is self-similar by Proposition~\ref{prop: classification of local structures}, and so $X_i\cong X_i\vee X_i$ for each $i$. Thus $X\cong X\vee X$ as desired.
\end{proof}

When all maximal ends are stable and of Cantor type, Lemma~\ref{lem: swaps} guarantees that there always exists a mapping class such that the induced maps on the end space and the fundamental groups have no fixed points. In the proposition to follow, we will use this lemma to show that $\Maps(X)$ does not have the JEP. We will construct the open sets $U$ and $V$ witnessing the failure of the JEP so that all elements in any conjugate of $U$ have a fixed point, but no element in $V$ has a fixed point. 

\begin{proposition}\label{prop: babel}
    Let $X$ be a locally finite graph with stable maximal ends. Then $\Maps(X)$ does not have a dense conjugacy class if either:
    \begin{enumerate}
        \item every maximal end dominating $\partial 1$ is of Cantor type; or
        \item every maximal end accumulated by genus is of Cantor type.
    \end{enumerate}
\end{proposition}
\begin{proof}
    Let $X=X_1\vee\cdots \vee X_n$ be the wedge decomposition of $X$. Assume either that $\partial 1\prec\partial X_1,...,\partial X_k$ and that $\partial 1\not\prec\partial X_{k+1},...,\partial X_n$, or that $\partial X_1,...,\partial X_k$ are all accumulated by genus and that $\partial X_{k+1},...,\partial X_n$ are all trees. Define $Y=X_1\vee\cdots \vee X_k$ and $Z=X_{k+1}\vee\cdots\vee X_n$. By Lemma~\ref{lem: swaps}, the locally finite graph $Y$ is proper homotopy equivalent to $Y\vee Y$.

    Suppose first that every maximal end of $Y$ dominates $\partial 1$. Then $Y\cong Y\vee 1$, and so $X\cong (Y\vee 1)\vee (Y\vee 1)\vee Z$, and we may assume that there exists a compact subset $K$ of $X$ such that the connected components of $X\setminus K$ are two copies of $Y$, two copies of $1$, and $Z$. Let $[g]$ be a mapping class swapping the two copies of $Y\vee 1$. We show that $\Maps(X)$ does not have the JEP. Let $U=U_K$ as in Proposition~\ref{AK-B 4.7}, and let $V=[g]U$; both are open by construction. Let $\rho$ and $\lambda$ be the two distinguished ends of type $\partial 1$. See Figure~\ref{fig: spiky Cantor tree}. By construction of $[g]$, no element of $V$ induces a map containing a fixed point on $\partial X$. On the other hand, if $[h]\in\Maps(X)$ is arbitrary and $[f]\in U$, then $[h][f][h]^{-1}([h](\rho))=[h][f](\rho)=[h](\rho)$, and so all elements of $[h]U[h]^{-1}$ fix $[h](\rho)$ (and $[h](\lambda)$, too), as $[f]$ was arbitrary. This shows that $[h]U[h]^{-1}\cap V=\0$, and thus $\Maps(X)$ does not have the JEP nor a dense conjugacy class.

    \begin{figure}[h]
        \centering
\begingroup%
  \makeatletter%
  \providecommand\color[2][]{%
    \errmessage{(Inkscape) Color is used for the text in Inkscape, but the package 'color.sty' is not loaded}%
    \renewcommand\color[2][]{}%
  }%
  \providecommand\transparent[1]{%
    \errmessage{(Inkscape) Transparency is used (non-zero) for the text in Inkscape, but the package 'transparent.sty' is not loaded}%
    \renewcommand\transparent[1]{}%
  }%
  \providecommand\rotatebox[2]{#2}%
  \newcommand*\fsize{\dimexpr\f@size pt\relax}%
  \newcommand*\lineheight[1]{\fontsize{\fsize}{#1\fsize}\selectfont}%
  \ifx\svgwidth\undefined%
    \setlength{\unitlength}{143.82011834bp}%
    \ifx\svgscale\undefined%
      \relax%
    \else%
      \setlength{\unitlength}{\unitlength * \real{\svgscale}}%
    \fi%
  \else%
    \setlength{\unitlength}{\svgwidth}%
  \fi%
  \global\let\svgwidth\undefined%
  \global\let\svgscale\undefined%
  \makeatother%
  \begin{picture}(1,0.85444562)%
    \lineheight{1}%
    \setlength\tabcolsep{0pt}%
    \put(0,0){\includegraphics[width=\unitlength,page=1]{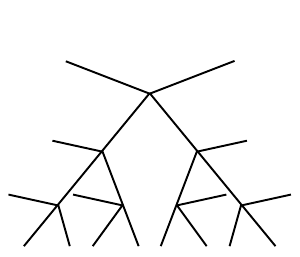}}%
    \put(0.14979113,0.66123472){\color[rgb]{0,0,0}\makebox(0,0)[t]{\smash{\begin{tabular}[t]{c}$\lambda$\end{tabular}}}}%
    \put(0.85073477,0.66384215){\color[rgb]{0,0,0}\makebox(0,0)[t]{\smash{\begin{tabular}[t]{c}$\rho$\end{tabular}}}}%
    \put(0,0){\includegraphics[width=\unitlength,page=2]{spiky_Cantor_tree.pdf}}%
    \put(0.49926729,0.78744623){\color[rgb]{0.8,0.8,1}\makebox(0,0)[t]{\smash{\begin{tabular}[t]{c}$g$\end{tabular}}}}%
    \put(0,0){\includegraphics[width=\unitlength,page=3]{spiky_Cantor_tree.pdf}}%
    \put(0.57846936,0.49906724){\color[rgb]{0.64705882,0.41176471,0.30980392}\makebox(0,0)[t]{\smash{\begin{tabular}[t]{c}$K$\end{tabular}}}}%
  \end{picture}%
\endgroup%

        \centering
        \caption{The ends $\lambda$ and $\rho$ are labeled. The proper homotopy equivalence $g$ reflects the graph across the periwinkle line, and $K$ is the portion of the graph in the brown circle.}
        \label{fig: spiky Cantor tree}
    \end{figure}

    Now suppose that every maximal end of $Y$ is accumulated by genus. A signature for $X$ is $(Y\vee R_1)\vee (Y\vee R_1)\vee Z$. Let $K$ be a compact subset of $X$ such that $K$ contains both copies of $R_1$ and the connected components of $X\setminus K$ are two copies of $Y$, and $Z$. Letting $[g]$ be a mapping class which swaps the two copies of $Y\vee R_1$, the argument above works here, replacing $\rho$ and $\lambda$ with two elements $r$ and $l$ of $\pi_1(X,x)$ which wrap around each distinguished copy of $R_1$ once so that $[g](r)=l$ and vice versa.
\end{proof}

\begin{remark}\label{rmk: generalized babel}
    The proof of Proposition~\ref{prop: babel} did not require the distinguished ends $\rho$ and $\lambda$ to be of type $\partial 1$. For example, in the graph $(\omega+1)\to C$, the proof would work if we distinguished two ends of type $\partial(\omega+1)$ instead of two ends of type $\partial 1$. We chose $\partial 1$ because it covers strictly more graphs, but we will use this observation later.
\end{remark}

\begin{example}
    Proposition~\ref{prop: babel} shows that the mapping class groups of the following locally finite graphs do not have dense conjugacy classes.
    \begin{itemize}
        \item Any locally finite graph with stable maximal ends such that every maximal end is of Cantor type (excluding $C$). This includes $o(C)$, $C\vee o(C)$, $(C\vee o(C))\to o(C)$, $1\to C$, and $((\omega^\omega+1)\to C)\vee o(C)\to o(C)$, among many others.
        \item Any self-similar locally finite graph with more than one maximal end which is not $C$.
        \item $(1\to C)\vee o(1)$. Note that Proposition~\ref{prop: babel} applies even if not all maximal ends are of Cantor type.
    \end{itemize}
\end{example}

We are now ready to prove Theorems~\ref{thm: main Rokhlin result} and~\ref{thm: Rokhlin nesting}, which we restate for the convenience of the reader.

\mRr*
\begin{proof}
    Let $X$ be a locally finite graph with self-similar end space. If $0<g(X)<\infty$, then $\Maps(X)$ does not contain a dense conjugacy class by Theorem~\ref{thm: finite genus}, and so we may assume that $g(X)\in\{0,\infty\}$. By Proposition~\ref{prop: classification of local structures}, $X$ has a unique maximal end type of size either one or infinity. If $X$ has a unique maximal end, then $\Maps(X)$ has a dense conjugacy class by Proposition~\ref{prop: unique maximal end}. On the other hand, suppose that $X$ has infinitely many maximal ends. If $X=C$, then $\Maps(X)$ has a dense conjugacy class by Proposition~\ref{prop: Cantor tree}. Otherwise, Lemma~\ref{lem: end classification} implies that the unique maximal end type of $X$ is either accumulated by genus or by $\partial 1$, in which case, by Proposition~\ref{prop: babel}, $\Maps(X)$ has no dense conjugacy class.
\end{proof}

\Rn*
\begin{proof}
    Let $X$ be a locally finite graph. Consider $Z^1$ and $Z^C$ as constructed in Lemma~\ref{lem: immediate successors}. As their end spaces are local structures, these two graphs are both self-similar by Proposition~\ref{prop: classification of local structures}, and, by construction, they both have genus either zero or infinity. By Theorem~\ref{thm: main Rokhlin result}, we have that $\Maps(Z^1)$ has a dense conjugacy class. On the other hand, the construction of $Z^C$ guarantees that $\partial 1\prec\partial Z^C$, and thus $\Maps(Z^C)$ does not have a dense conjugacy class by Proposition~\ref{prop: babel}.
    
    By construction of $Z^1$, $X\vee Z^1\cong Z^1$. Then the set of maps in $\Maps(Z^1)$ which induce the identity everywhere in $X\vee Z^1$ except for the distinguished copy of $X$ is equal to an intersection of clopen sets of the form $U_K$ (see Proposition~\ref{AK-B 4.7}) which fix compact sets $K\subset Z^1\setminus X$. This intersection is homeomorphic to $\Maps(X)$. The same reasoning shows that $\Maps(X)$ embeds as a closed subset into $\Maps(Z^C)$, and the moreover statement is clear from the constructions of $Z^1$ and $Z^C$.
\end{proof}

\subsection{Flux maps}\label{sec: flux maps}

In this section, we prove that a large class of mapping class groups do not contain a dense conjugacy class, proving Theorem~\ref{thm: obstructions to dense conjugacy classes} (1). Our strategy is to find a non-trivial continuous group homomorphism from $\Maps(X)$ to $\Z$, which we call a \emph{flux map}, similar to the flux maps defined in \cite[Section~7]{DHK}. The kernel of the flux map is a proper open normal subgroup of $\Maps(X)$, and, by Lemma~\ref{lem: proper open normal subgrp}, this implies that $\Maps(X)$ has no dense conjugacy class.

Intuitively, the maps we construct from $\Maps(X)$ to $\Z$ will ``count'' how many of an object, whether it is immersed loops or ends of a given type, pass from one location in $X$ to another under a given mapping class. We first construct flux maps which count ends and then those which count loops.

As a first example, consider the graph $o(1)\vee o(1)$, which is a line with a copy of $S^1$ attached to each vertex. A loop shift, formally defined in \cite[Section~3.4]{DHK}, is a proper homotopy equivalence of $o(1)\vee o(1)$ which sends each copy of $S^1$ to the adjacent one as in Figure~\ref{fig: loop shift}, and it's image under the flux map that counts loops is one. On the other hand, a similar map defined on the graph $o(\omega+1)\vee o(\omega+1)$ sending non-maximal ends to adjacent ones as in Figure~\ref{fig: end shift} has infinitely many loops passing between the two maximal ends, and thus, counting loops would not yield a map to $\Z$. Instead, we count ends in $E(\partial o(1))$. Near the beginning of Section~\ref{sec: Rokhlin}, we mentioned that the graph $o(1)$ can be thought of as a ray dominating copies of $S^1$ as motivation for considering immersed loops to be minimal ends, and this setting provides more motivation. Note that in both examples, the ends we count should satisfy a maximality condition to avoid the possibility of infinite ends or loops shifting towards an end. This becomes more apparent if one replaces the copies of $S^1$ in Figure~\ref{fig: end shift} by rays. In light of this, we define a \emph{gcd}, or \emph{greatest common divisor} of a pair of ends.

\begin{figure}[h]
    \centering
    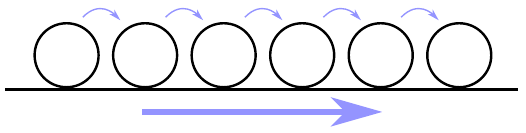
    \caption{A loop shift.}
    \label{fig: loop shift}
\end{figure}

\begin{figure}[h]
    \centering
    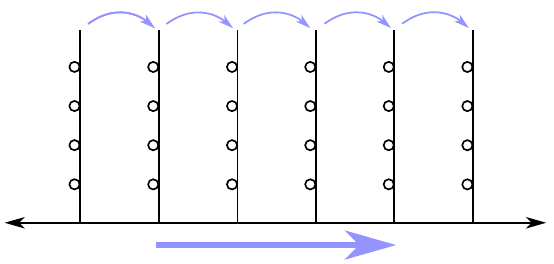
    \caption{There are infinitely many loops passing between the two maximal ends, but only finitely many ends in $E(\partial o(1))$.}
    \label{fig: end shift}
\end{figure}

\begin{definition}\label{def: gcd}
    Let $\mu_1,\mu_2\in\partial X$ be two maximal ends. We say that a stable end $\lambda\in\partial X$ \emph{is a gcd of $\mu_1$ and $\mu_2$} and write $\lambda\in$ gcd$(\mu_1,\mu_2)$, if $\lambda\prec\mu_1,\mu_2$, and for any non-maximal $\lambda'$ such that $\lambda\prec\lambda'$, the maximal end types which dominate $\lambda'$ are either $E(\mu_1)$ or $E(\mu_2)$, but not both. Similarly, we say that $R_1$ \emph{is a gcd of $\mu_1$ and $\mu_2$} and write $R_1\in\gcd(\mu_1,\mu_2)$, if $\mu_1,\mu_2\in\partial X_g$, and for any $\lambda'\in\partial X_g$, the maximal end types which dominate $\lambda'$ are either $E(\mu_1)$ or $E(\mu_2)$, but not both. We define gcds of end types with the analogous partial order.
\end{definition}

For example, we have that $R_1\in\gcd(o(1),o(1))$ in $o(1)\vee o(1)$ (see Figure~\ref{fig: loop shift}), and while $R_1$ is not a gcd of the two maximal ends in Figure~\ref{fig: end shift}, $\partial o(1)$ is. Additionally, $\partial 1$ is a gcd of the two maximal end types in both graphs in Figure~\ref{fig: flux map T}. Note that gcds need not be unique. In fact, the two maximal ends of the graph $(\{\bigvee_{i=0}^n X_i\}_n\to 1)\vee(\{\bigvee_{i=0}^n X_i\}_n\to 1)$, where the $X_i$ are as constructed in Lemma~\ref{lem: class of incomparables}, have countably infinitely many gcds.

\begin{proposition}\label{prop: flux map for ends}
    Let $X$ be a locally finite graph with stable maximal ends, let $\mu_1\not\sim\mu_2$ be two maximal ends, and let $\lambda\in\gcd(\mu_1,\mu_2)$. If $\lambda$ is not of Cantor type, then $\Maps(X)$ does not have a dense conjugacy class.
\end{proposition}

The proof of \cite[Theorem~7.5]{DHK}, whose outline we will follow to prove Proposition~\ref{prop: flux map for ends}, shows the existence of continuous group homomorphisms from P$\Maps(X)$ to $\Z$, where P$\Maps(X)$ is the subgroup of $\Maps(X)$ consisting of elements which induce the trivial homeomorphism on $\partial X$. Roughly speaking, the flux maps of \cite[Theorem~7.5]{DHK} count how many immersed loops pass between two given ends of $X$. Proposition~\ref{prop: flux map for ends} adapts this proof to define an analogous homeomorphism that counts ends and is defined on all of $\Maps(X)$. We then prove an analogous result concerning loops, again building a homeomorphism defined on all of $\Maps(X)$. 

Before we prove Proposition~\ref{prop: flux map for ends}, we first construct the flux homomorphism. There are three main steps: defining subgraphs $T$ and $X_n$ for any $n\in\Z$, defining corank and admissible pairs, and finally constructing the flux map. 

\subsubsection*{Step 1: constructing subgraphs}

\textbf{Constructing $Y_1$ and $Y_2$:} By Proposition~\ref{AK-B 4.5}, we may assume that $X$ is broken up into its wedge decomposition $W_1\vee\cdots \vee W_k$, i.e., that there is a vertex whose complementary components are each of the $W_i$. We may also assume that $\mu_i\in\partial W_i$ for $i=1,2$. Expand the wedge point of $W_1\vee\cdots \vee W_k$ to an edge via a proper homotopy equivalence so that $Y_1=W_1$ is connected to one end of that edge and $Y_2=W_2\vee\cdots \vee W_k$ is connected to the other, and let $x_0$ be the midpoint of this edge. Then $X$ is equal to the union of $Y_1$, $Y_2$, and the edge between them.
    
\textbf{Constructing $T$ and $Z_\nu$:} We define a subgraph $T$ of $X$ which contains no ends in $E(\lambda)$, similar to how an underlying tree consists of $X$ without loops. Because $\lambda$ is not of Cantor type, it follows from Lemma~\ref{lem: small clopen neighborhoods} that any stable neighborhood $U_\nu$ of each end $\nu\in E(\lambda)$ is such that $U_{\nu}\cap E(\lambda)=\{\nu\}$, i.e., such that $U_\nu$ is disjoint from all other ends in $E(\lambda)$. The neighborhoods $U_\nu$ can be made arbitrarily small, and thus we may assume that the $U_\nu$ are pairwise disjoint. Consider the clopen cover $\{U_\nu\}_{\nu\in E(\lambda)}$ of $E(\lambda)$. The complement of this cover in the end space, given by $\partial X\setminus(\bigcup_{\nu\in E(\lambda)}U_\nu)$, is closed in $\partial X$. Therefore, there exists a subgraph $T$ of $X$ whose end space is $\partial X\setminus(\bigcup_{\nu\in E(\lambda)}U_\nu)$. Intuitively, the graph $T$ is obtained from $X$ by removing all ends in $E(\lambda)$. See Figure~\ref{fig: flux map T} for examples when $X\cong1\to(1\vee C)$ and $(\omega^2+1)\vee((1\to C)\to o(1))$. Define $Z_\nu$ to be a graph with end space equal to $U_\nu$ for each $\nu\in E(\lambda)$. Note that for each $\nu\in E(\lambda)$, $\nu\sim\partial Z_{\nu}$. Assume, by potentially altering $X$ up to proper homotopy equivalence, that $X$ is a wedge product of $T$ with each $Z_\nu$, and for each $\nu\in E(\lambda)$, let $x_\nu$ be vertex in $T$ which is the wedge point connecting $Z_\nu$ to the rest of $X$.

\begin{figure}
    \centering
\begingroup%
  \makeatletter%
  \providecommand\color[2][]{%
    \errmessage{(Inkscape) Color is used for the text in Inkscape, but the package 'color.sty' is not loaded}%
    \renewcommand\color[2][]{}%
  }%
  \providecommand\transparent[1]{%
    \errmessage{(Inkscape) Transparency is used (non-zero) for the text in Inkscape, but the package 'transparent.sty' is not loaded}%
    \renewcommand\transparent[1]{}%
  }%
  \providecommand\rotatebox[2]{#2}%
  \newcommand*\fsize{\dimexpr\f@size pt\relax}%
  \newcommand*\lineheight[1]{\fontsize{\fsize}{#1\fsize}\selectfont}%
  \ifx\svgwidth\undefined%
    \setlength{\unitlength}{480.90674693bp}%
    \ifx\svgscale\undefined%
      \relax%
    \else%
      \setlength{\unitlength}{\unitlength * \real{\svgscale}}%
    \fi%
  \else%
    \setlength{\unitlength}{\svgwidth}%
  \fi%
  \global\let\svgwidth\undefined%
  \global\let\svgscale\undefined%
  \makeatother%
  \begin{picture}(1,0.3144429)%
    \lineheight{1}%
    \setlength\tabcolsep{0pt}%
    \put(0,0){\includegraphics[width=\unitlength,page=1]{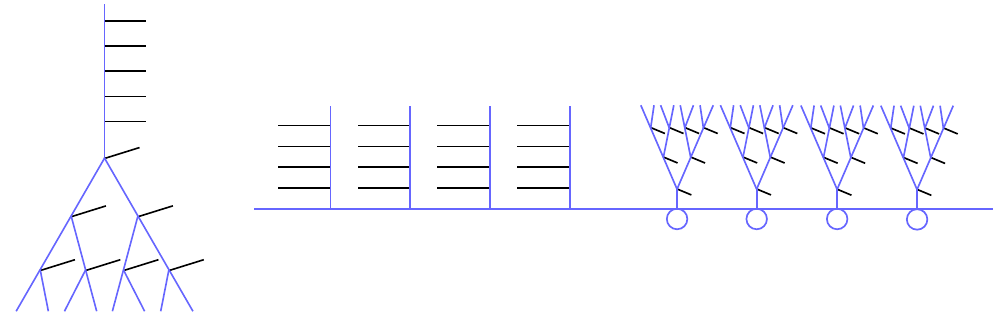}}%
    \put(0.15787876,0.23631745){\color[rgb]{0,0,0}\makebox(0,0)[t]{\smash{\begin{tabular}[t]{c}$\nu$\end{tabular}}}}%
    \put(0.69752135,0.11172392){\color[rgb]{0,0,0}\makebox(0,0)[t]{\smash{\begin{tabular}[t]{c}$\nu$\end{tabular}}}}%
    \put(0,0){\includegraphics[width=\unitlength,page=2]{flux_map_T.pdf}}%
    \put(0.05399955,0.24289837){\color[rgb]{0.4,0.4,1}\makebox(0,0)[t]{\smash{\begin{tabular}[t]{c}$x_\nu$\end{tabular}}}}%
    \put(0.62198656,0.12163747){\color[rgb]{0.4,0.4,1}\makebox(0,0)[t]{\smash{\begin{tabular}[t]{c}$x_\nu$\end{tabular}}}}%
    \put(0,0){\includegraphics[width=\unitlength,page=3]{flux_map_T.pdf}}%
  \end{picture}%
\endgroup%

    \caption{Above are the graphs $1\to(1\vee C)$ (left) and $(\omega^2+1)\vee((1\to C)\to o(1))$ (right). The unique gcd of the two maximal end types in both graphs is $\partial 1$. The corresponding trees $T$ are in periwinkle. An example of a subgraph $Z_\nu$ corresponding to an end $\nu$ is given in the brown boxes on each graph, with the corresponding $x_\nu$ labeled. Note that as with the graph on the right, the gcd of two maximal ends need not be an immediate predecessor of either of them.}
    \label{fig: flux map T}
\end{figure}

\textbf{Constructing subgraphs $X_n$:} If $n\geq0$, let $X_n$ be equal to $\overline{Y_1\cup B_n(x_0)}$ together with the graphs $Z_{\nu}$ wedged to each $x_\nu$ appearing in $B_n(x_0)$. If $n<0$, then form $X_n$ from $Y_1$ by removing each $Z_\nu$ wedged to any $x_\nu$ appearing in $B_{-n}(x_0)$. See Figure~\ref{fig: X_n example} for examples.

\begin{figure}[h]
    \centering
    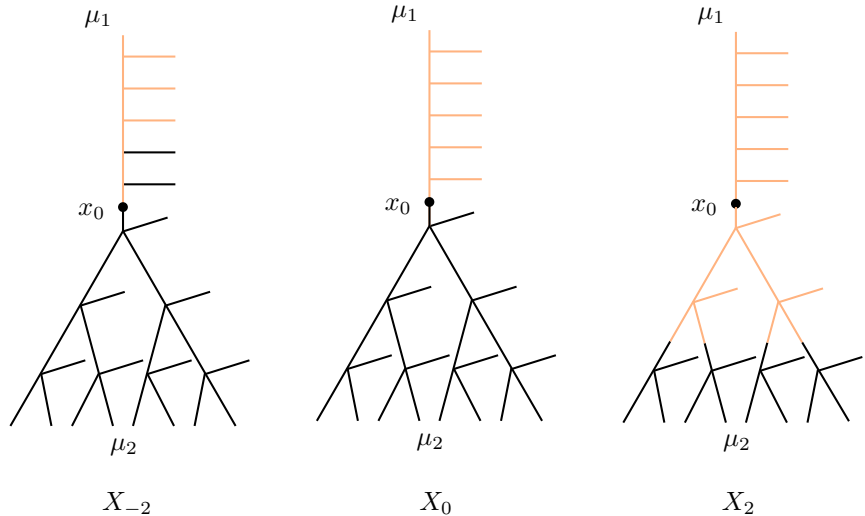
    \caption{The subgraphs $X_n$ for $n=-2,0,2$ are shown in orange.}
    \label{fig: X_n example}
\end{figure}

\begin{remark}\label{rmk: end space of X_n and f(X_n)}
    The following facts about the end spaces of the graphs defined above will be useful later in the proof. Let $n\in\Z$, and let $f\in\PHE(X)$.
    \begin{enumerate}
        \item $\partial Y_1$ and $\partial Y_2$ are clopen in $X$.
        \item $E(\mu_i)\subset\partial Y_i$ for $i=1,2$.
        \item Any end $\lambda'$ such that $\lambda\prec\lambda'$ is in either $\partial Y_1$ or $\partial Y_2$, but not both, because otherwise, $\lambda'$ would be dominated by multiple end types, contradicting that $\lambda$ is a gcd.
        \item $E(\mu_1)\subset\partial X_n$, and, by Lemma~\ref{lem: caste system}, $E(\mu_1)\subseteq \partial f(X_n)$ as well.
        \item Each end in $X_n$ is either in $\partial Y_1\cap\partial T$ or in $\partial Z_\nu$ for some $\nu\in E(\lambda)$.
        \item Both $\partial X_n$ and $\partial f(X_n)$ are clopen in $\partial X$.
        \item For all ends $\lambda'\in\partial Z_\nu$, $\lambda'\not\succ\lambda$ by Proposition~\ref{prop: classification of local structures}.
    \end{enumerate}
\end{remark}

\subsubsection*{Step 2: corank and admissible pairs}

The flux maps in \cite[Section~7]{DHK} are defined using a tool called corank, which calculates how many loops are in one subgraph containing $\mu_1$ but not another. We define corank in our setting in a similar vein.

\begin{definition}[corank]
    Let $W$ and $W'$ be two subgraphs of $X$ such that $E(\mu_1)\subseteq \partial W\cap\partial W'$. Define $\cork(W,W'):=|\{\nu\in\partial W\setminus\partial W'\,:\,\nu\in E(\lambda)\}|$.
\end{definition}

\begin{example}\label{ex: cork examples} By construction, the end space of each $X_n$ contains $E(\mu_1)$, and thus for $m,n\in\Z$, the quantity $\cork(X_m,X_n)$ is well defined. We provide the following examples and properties.
    \begin{itemize}
        \item In Figure~\ref{fig: X_n example}, $\cork(X_0,X_{-2})=2$, $\cork(X_2,X_0)=3$, and $\cork(X_2,X_{-2})=5$.
        \item $\cork(X_n,X_m)=0$ when $n<m$.
        \item The cork operation satisfies \emph{additivity}, that is, if $E(\mu_1)\subset\partial W\subseteq\partial W'\subseteq\partial W''$, then $\cork(W'',W)=\cork(W'',W')+\cork(W',W)$.
        \item The cork operation is invariant after applying proper homotopy equivalences. In other words, if $f\in\PHE(X)$ and $E(\mu_1)\subset\partial W\cap\partial W'$, then $\cork(W',W)=\cork(f(W),f(W'))$.
        \item Given $f\in\PHE(X)$, it follows from Remark~\ref{rmk: end space of X_n and f(X_n)} that $\cork(X_m,f(X_n))$ is well defined.
    \end{itemize}
\end{example}

The quantity $\cork(X_m,X_n)-\cork(X_m,f(X_n))$ counts how many ends in $\partial X_n$ which are equivalent to $\lambda$ move towards $\mu_2$ under a mapping class $f$ for some large enough $m$, hence reducing dynamics on the whole graph to a finite quantity. This quantity will yield the desired map to $\Z$, and we will show that it is independent of the choice $m$ and $n$, as long as $m$ is big enough. First we make precise the notion of ``large enough $m$'' through admissible pairs.

\begin{definition}
    We say that a pair $(m,n)\in(\Z_{\geq 0})^2$ forms an \emph{admissible pair} relative to $f\in\PHE(X)$ if $(\partial X_n\cup\partial f(X_n))\cap E(\lambda)\subseteq\partial X_m\cap E(\lambda)$.
\end{definition}

For example, if $X_n$ and $X_m$ are such that $\partial X_n\cup\partial f(X_n)\subseteq\partial X_m$, then $(m,n)$ is an admissible pair. Informally, the pair $(m,n)$ is admissible relative to $f$ if $X_m$ is a sufficiently large ambient space in which to understand the dynamics of the induced map on $\partial X_n$. We next show that for any $n$, there always exists such an $m$.

\begin{claim}\label{cl: existance of admissible pairs}
    Let $n\in\Z_{\geq 0}$ and let $f\in\PHE(X)$. There exists $m\in\Z_{\geq0}$ such that $(m,n)$ is an admissible pair relative to $f$.
\end{claim}
\begin{proof}
    It suffices to show that 
    $$p=\sup\{d(x_0,x_{\nu})\,:\,\nu\in\partial f(X_n)\cap\partial Y_2\cap E(\lambda)\}$$ 
    is finite because then the pair $(m,n)$ would be admissible relative to $f$, where $m=\max\{\lceil p\rceil,n\}$. This is because $\partial X_m\cap E(\lambda)$ must contain both $\partial f(X_n)\cap E(\lambda)$ and $\partial X_n\cap E(\lambda)$ by construction. Suppose for contradiction that $p$ is not finite. Then the full subgraph generated by all of the $x_\nu\in\mathcal{V}(T\cap f(X_n)\cap Y_2)$ (where $\nu\in E(\lambda)$) is infinite. Thus, by K\H{o}nig's lemma \cite{K}, that subgraph contains an end $\lambda'$. By Remark~\ref{rmk: end space of X_n and f(X_n)}, both $\partial Y_2$ and $\partial f(X_n)$ are clopen in $\partial X$, and so $\lambda'\in\partial Y_2$ and $\lambda'\in\partial f(X_n)$. Also, $\lambda\prec\lambda'$, since a sequence of vertices $x_\nu$ converges to $\lambda'$. By Lemma~\ref{lem: caste system}, the intersection $E(\lambda')\cap\partial X_n$ is nonempty. But then by Remark~\ref{rmk: end space of X_n and f(X_n)}, $\partial Y_1\cap E(\lambda')\neq\0$, contradicting the assumption that $\lambda$ is a gcd. Thus the pair $(\lceil p\rceil,n)$ is admissible.
\end{proof}

\begin{claim}\label{cl: finite coker}
    For every admissible pair $(m,n)$ relative to $f$, the quantities $\cork(X_m,X_n)$ and $\cork(X_m,f(X_n))$ are finite.
\end{claim}
\begin{proof}
    From the definitions of $X_n$, $X_m$, and local finiteness of $X$, it follows that $\cork(X_m,X_n)$ is finite. Now suppose for contradiction that $\cork(X_m,f(X_n))$ is infinite. Then, as in the previous claim, K\H{o}nig's lemma implies that there is an end $\lambda'$ strictly dominating $\lambda$ in $\partial X_m\setminus \partial f(X_n)$. Thus, by Remark~\ref{rmk: end space of X_n and f(X_n)}, $\lambda'\not\in\partial Z_\nu$ for any $\nu\in E(\lambda)$, and so the end $\lambda'$ must be in $\partial Y_1$. If $E(\lambda')\cap Y_2\neq \0$, we would reach a contradiction on the assumption that $\lambda$ is a gcd, and so $E(\lambda')\subset \partial Y_1$. Since $Y_1\subseteq X_n$, Lemma~\ref{lem: caste system} implies that $E(\lambda')\subset\partial f(X_n)$, and so $\lambda'\in\partial f(X_n)$, which is a contradiction.
\end{proof}

\subsubsection*{Step 3: constructing the flux map}

Given an admissible pair $(m,n)$, define the map $\phi_{m,n}\colon\PHE(X)\to\Z$ to be given by 
$$f\mapsto\cork(X_m,X_n)-\cork(X_m,f(X_n)),$$
where $(n,m)$ is an admissible pair relative to $f\in\PHE(X)$.

\begin{claim}\label{cl: phi is continuous group homomorphism}
    The map $\phi_{m,n}$ constructed above is a well-defined continuous group homomorphism.
\end{claim}
\begin{proof}
    The proof of \cite[Lemma~7.11]{DHK} is sufficient for this claim, and we summarize it here. First, the quantity $\cork(X_m,X_n)-\cork(X_m,f(X_n))$ is indeed an element of $\Z$ by Claim~\ref{cl: finite coker}, and it is immediate that $\phi_{m,n}(id)=0$. For $f,g\in\PHE(X)$ and $n\in\Z$, there exists $m\in\Z$ such that $(m,n)$ is admissible relative to all three maps $f$, $g$, and $fg$ by Claim~\ref{cl: existance of admissible pairs}. The rest of the proof consists of algebraic manipulations and making use of the additivity of corank mentioned in Example~\ref{ex: cork examples}. Lastly, the continuity of $\phi_{m,n}$ follows as any homeomorphism to $\Z$ is continuous \cite[Theorem~1]{D}.
\end{proof}

\begin{claim}\label{cl: phi desends to Phi}
    If $f$ and $g$ are properly homotopic, then $\phi_{m,n}(f)=\phi_{m,n}(g)$. Moreover, if $(m,n)$ and $(m',n')$ are two admissible pairs for $f$, then $\phi_{m,n}(f)=\phi_{m',n'}(f)$.
\end{claim}
\begin{proof}
    The first statement follows from the fact that if $f$ and $g$ are properly homotopic, then they induce the same homeomorphism on $\partial X$. The moreover statement follows from the same proof as in \cite[Lemma~7.10]{DHK} which we briefly summarize. Additivity of corank (see Example~\ref{ex: cork examples}) gives that $\phi_{m',n}(f)=\phi_{m,n}(f)$ for any $f$ through algebraic manipulations similar to those in triangle inequality proofs, and thus it suffices to show that $\phi_{m,n}=\phi_{m,n'}$. This follows from additivity and invariance after applying proper homotopy equivalences (see Example~\ref{ex: cork examples}) through similar manipulations.
\end{proof}

\begin{definition}\label{def: flux map for ends}
    Let the \emph{flux map} $\Phi\colon\Maps(X)\to\Z$ be given by sending $[f]$ to $\phi_{m,n}(f)$ for some choice of representative $f\in[f]$ and choice of admissible pair $(m,n)$ relative to $f$. By Claims\ref{cl: phi is continuous group homomorphism} and~\ref{cl: phi desends to Phi}, this is a well-defined continuous group homomorphism.
\end{definition}

\begin{claim}\label{cl: nontrivial}
    The flux map $\Phi\colon\Maps(X)\to\Z$ is non-trivial.
\end{claim}
\begin{proof}
    An element of $\Maps(X)$ that is mapped non-trivially to $\Z$ under $\Phi$ can be constructed using \cite[Observation~4.9]{MRlong} on the collection $\{U_\nu\}_{\nu\in E(\lambda)}$.
\end{proof}

We are now ready to prove Proposition~\ref{prop: flux map for ends}.

\begin{proof}[Proof of Proposition~\ref{prop: flux map for ends}]
    Because $\Phi$ is a non-trivial continuous group homomorphism, the kernel of $\Phi$ is a proper open normal subgroup of $\Maps(X)$. By Lemma~\ref{lem: proper open normal subgrp}, $\Maps(X)$ has no dense conjugacy class.
\end{proof}

\begin{proposition}\label{prop: flux map for loops}
    Let $X$ be a locally finite graph with stable maximal ends. Let $\mu_1\not\sim\mu_2$ be two maximal ends. If $R_1\in\gcd(\mu_1,\mu_2)$, then $\Maps(X)$ does not have a dense conjugacy class.
\end{proposition}
\begin{proof}
    In tandem with the analogy between ends and loops, this proof is almost identical to the proof of Proposition~\ref{prop: flux map for ends}. Assuming $X$ is in standard form, the graphs $Y_1$ and $Y_2$ can be defined in the same way, and the analogous $T$ is the underlying tree. Each $Z_\nu$ here is a copy of $R_1$, and the $x_\nu$ can be defined in the same way.

    The subgraphs $X_n$  can be defined analogously, but we describe them explicitly to avoid confusion. If $n\geq0$, let $X_n$ be equal to $\overline{Y_1\cup B_n(x_0)}$. Otherwise, let $X_n=(Y_1\setminus B_{n}(x_0))\cup (T\cap Y_1)$. See Figure~\ref{fig: X_n example for loops} for an example. Remark~\ref{rmk: end space of X_n and f(X_n)} holds in this setting, except that the end space of each $X_n$ is exactly equal to $\partial Y_1\cap\partial T$, because there are no ends in any $Z_\nu$ when $Z_\nu\cong R_1$.

    \begin{figure}[h]
        \centering
        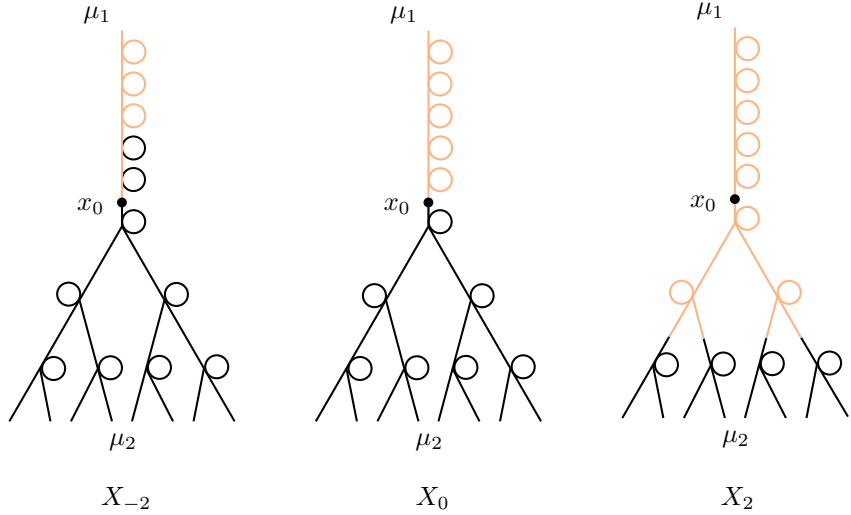
        \caption{The subgraphs $X_n$ for $n=-2,0,2$ are shown in orange.}
        \label{fig: X_n example for loops}
    \end{figure}

    Each $X_n$ corresponds to a subgroup $A_n:=\pi_1(X_n,x_0)$ of $\pi_1(X,x_0)$. Moreover, this subgroup is a \emph{free factor} of $\pi_1(X,x_0)$, namely, there exist subgroups $H_n<\pi_1(X,x_0)$ such that $\pi_1(X,x_0)\cong A_n*H_n$ for each $n$. Furthermore, for $n,m\in\Z$ with $n<m$, $A_n\leq A_m$, and $A_n$ is a free factor of $A_m$ \cite[Lemma~7.2]{DHK}. Let $$\cork(A_m,A_n):=\rk(A_m/\langle\langle A_n\rangle\rangle).$$ 
    Corank is additive and invariant after applying proper homotopy equivalences. A pair of integers $(m,n)$ with $n<m$ is \emph{admissible} relative to a proper homotopy equivalence $f$ if $A_n$ and $f_*(A_n)$ are free factors of $A_m$. The proofs of Claims~\ref{cl: existance of admissible pairs}, \ref{cl: finite coker}, \ref{cl: phi is continuous group homomorphism}, and~\ref{cl: phi desends to Phi} are the same after making the appropriate changes from ends to loops, and Claim~\ref{cl: nontrivial} follows from the existence of loop shifts \cite[Section~3.4]{DHK}.
\end{proof}

We now prove the following, which is Theorem~\ref{thm: obstructions to dense conjugacy classes} (1) in the introduction.

\begin{theorem}\label{thm: item 1}
    Let $X$ be a locally finite graph with stable maximal ends. If there exists distinct maximal ends $\mu_1$ and $\mu_2$ in $\partial X$ with a gcd which is not of Cantor type, then $\Maps(X)$ does not contain a dense conjugacy class.
\end{theorem}
\begin{proof}
    If $\mu_1\not\sim\mu_2$, then the result follows from Propositions~\ref{prop: flux map for ends} and~\ref{prop: flux map for loops}. Otherwise, suppose that $\mu_1\sim\mu_2$. If $E(\mu_1)$ is finite, then $|E(\mu_1)|>1$, so Theorem~\ref{thm: finite end type} implies the result.
    
    By Proposition~\ref{MR 4.7}, we assume that $E(\mu_1)$ is homeomorphic to Cantor space. If there exists an end $\lambda\in\gcd(\mu_1,\mu_2)$, then either there is another maximal end $\mu_3\not\in E(\mu_1)$ such that $\lambda\prec\mu_3$ or not. If not, then the only maximal ends dominating $\lambda$ are of type $E(\mu_1)$, and $\mu_1$ is of Cantor type, so Remark~\ref{rmk: generalized babel} shows the result, after distinguishing ends of type $\lambda$. If there is such a $\mu_3$, then as $\lambda\in\gcd(\mu_1,\mu_2)$, there is no $\nu$ such that $\lambda\prec\nu\prec\mu_1$, and so $\lambda\in\gcd(\mu_1,\mu_3)$. Thus, we may apply Proposition~\ref{prop: flux map for ends} to obtain the result. If $R_1\in\gcd(\mu_1,\mu_2)$, a similar logic follows.
\end{proof}

\begin{example}
    Theorem~\ref{thm: item 1} can be used to show that the mapping class groups of the following locally finite graphs do not contain a dense conjugacy class.
    \begin{enumerate}
        \item In the graph $(\omega^\alpha+1)\to (1\vee C)$, where $\alpha$ is a countable ordinal, the gcd of the two maximal end types is $\partial(\omega^\alpha+1)$.
        \item The graph $(\{\omega^n+1\}_n\to o(1))\vee((\omega^\alpha+1)\to C)$ where $\alpha$ is a countable ordinal, has two distinct maximal end types. If $\alpha<\omega$, then $\partial(\omega^\alpha+1)$ is a gcd of the two maximal end types. If $\alpha\geq\omega$, then $\partial(\omega^\alpha+1)$ is a gcd of any two ends in the maximal end type $\partial ((\omega^\alpha+1)\to C)$. The theorem applies even if the two maximal ends are equivalent.
        \item In the graph $(\omega^2+1)\vee((1\to C)\to o(1))$ from Figure~\ref{fig: flux map T}, the two maximal ends have a gcd $\partial 1$ even though $\partial 1$ is not an immediate successor of either maximal end.
        \item Consider the graph with wedge decomposition $(X\to1)\vee(X\to o(1))$, where $X$ is a locally finite tree with a unique maximal end which is not self-similar (and not stable) such as the one constructed in \cite{MRshort}. Because the maximal end in any copy of $X$ is not stable, ends of this type are not technically gcds. The stability assumption is only used in Claim~\ref{cl: nontrivial}, and by shifting copies of $X$ in this graph, it is clear that there still exists non-trivial flux maps. See Remark~\ref{rmk: nontrivial} for the general principle.
    \end{enumerate}
\end{example}

\begin{example}\label{ex: subtle non-example}
    Consider $(\omega^2+1)\vee(1\to C)\vee(1\to o(1))$. The end type $\partial 1$ is not a gcd of any pair of maximal ends  because $\partial 1$ is dominated by more than two maximal end types. As a result of the subtle requirement that an end is a gcd only if it is dominated by at most two maximal end types, this graph does not contain a gcd.
\end{example}

It is possible to define a flux map on the graph in Example~\ref{ex: subtle non-example} by generalizing flux maps as follows. A flux map on $X$ is specified by the following three pieces of information.

\begin{enumerate}
    \item The two ends of the flow: Two closed nonempty disjoint subsets $A_1$ and $A_2$ of $\partial X$ such that $A_1\subseteq E(\nu_1)$ and $A_2\subseteq E(\nu_2)$ for some $\nu_1,\nu_2\in \partial X$.
    \item The end type we will count: An end $\lambda\prec \nu_1, \nu_2$ that is not of Cantor type such that there exists a \emph{flux splitting}, which is a decomposition of $X$ into a wedge product $Y_1\vee Y_2$ such that the following two conditions are satisfied.
    \begin{enumerate}
        \item $A_i\subset \partial Y_i$ for $i=1,2$.
        \item For any point $\lambda'\in X$ such that $\lambda\prec \lambda'$, $E(\lambda')$ is in either $\partial Y_1$ or $\partial Y_2$, but not both.
    \end{enumerate}
    \item A subgroup $H$ of $\Maps(X)$ that set-wise stabilizes both $A_1$ and $A_2$ and does not set-wise fix $E(\lambda)\cap U$ for any neighborhood $U$ of $A_1$ disjoint from $A_2$ (respectively $A_2$ disjoint from $A_1$).
\end{enumerate}

If the data $(A_1,A_2,y,H)$ satisfies the three conditions above, then we obtain a flux map
$$\Phi_{(A_1,A_2,y,H)}\colon H\to\Z.$$

\begin{remark}\label{rmk: nontrivial}
    The flux map $\Phi$ is non-trivial if there exists a clopen, pairwise disjoint, and pairwise homeomorphic cover $\{U_\nu\}_{\nu\in E(\lambda)}$ of $E(\lambda)$ in $\partial X$ such that $U_{\nu}\cap E(\lambda)$ is a singleton for all $\nu$. Such a cover always exists if $\lambda$ is stable.
\end{remark}

The following is a generalization of Theorem~\ref{thm: obstructions to dense conjugacy classes} (1).

\begin{theorem}\label{thm: flux map general}
    Let $X$ be a locally finite graph, and suppose that $A_1,A_2\subset\partial X$, $\lambda\in\partial X$, and $H\leq\Maps(X)$ satisfy the three conditions above and the conditions of Remark~\ref{rmk: nontrivial}. Then $H$ does not have a dense conjugacy class.
\end{theorem}
\begin{proof}
    By the same methods as in Propositions~\ref{prop: flux map for ends} and~\ref{prop: flux map for loops}, the map $\Phi$ is a continuous group homomorphism, and so the kernel of $\Phi$ is a proper open normal subgroup of $H$. By Lemma~\ref{lem: proper open normal subgrp}, $H$ has no dense conjugacy class.
\end{proof}

Theorem~\ref{thm: flux map general} implies Propositions~\ref{prop: flux map for ends} and~\ref{prop: flux map for loops} by letting $\nu_1$ and $\nu_2$ be maximal ends of different types and setting $A_i=E(\nu_i)$ for $i=1,2$. Moreover, we have the following corollary.

\PMaps*
\begin{proof}
    This follows from Theorem~\ref{thm: flux map general} when $\nu_1$ and $\nu_2$ are distinct ends which are both accumulated by genus, $A_i=\{\nu_i\}$ for $i=1,2$, and when $H=$P$\Maps(X)$. In this case, the gcd $R_1$ will yield the result.
\end{proof}

\subsection{Trees and direct products}\label{sec: trees and direct products}

In this section, we show that an uncountable family of mapping class groups split into direct products of other mapping class groups and use this to obtain more results about dense conjugacy classes.

\begin{lemma}\label{lem: direct product of groups}
    Let $G=H\times K$ be a topological group. Then $G$ has a dense conjugacy class if and only if both $H$ and $K$ have dense conjugacy classes.
\end{lemma}
\begin{proof}
    The conjugacy class of an element $g=(h,k)\in G$ is dense if and only if for any arbitrary element $(h',k')\in G$, there exists a sequence of elements $\{(h_n,k_n)\}_n$ such that $(h_n,k_n)(h,k)(h_n,k_n)^{-1}=(h_nhh_n^{-1},k_nkk_n^{-1})$ converges to $(h',k')$ as $n$ approaches infinity. This is true if and only if the conjugacy class of $h$ is dense in $H$ and the conjugacy class of $k$ is dense in $K$.
\end{proof}

The following proposition provides many examples of graphs whose mapping class group has a dense conjugacy class, in contrast to much of the rest of Section~\ref{sec: Rokhlin}. We first prove the following more general proposition and then prove Theorem~\ref{thm: trees and direct products}, which follows immediately.

\begin{proposition}\label{prop: trees and direct products}
    Let $X$ and $Y$ be locally finite trees such that $X\cong Y\vee C$. Then $\Maps(X)$ has a dense conjugacy class if and only if $\Maps(Y)$ has a dense conjugacy class.
\end{proposition}
\begin{proof}
    Suppose that there is an end $\nu\in\partial Y$ such that $\nu\sim\partial C$. Then $Y\cong Y\vee C$, which implies that $X\cong Y$. On the other hand, suppose that there is no end of type $\partial C$ in $Y$. Then, by Lemma~\ref{lem: caste system}, any homeomorphism of $\partial X$ splits into a homeomorphism $g$ on $\partial Y$ which fixes $\partial C$, and a homeomorphism $h$ on $\partial C$ which fixes $\partial Y$. The supports of these maps are disjoint, and thus the maps commute. By Proposition~\ref{AK-B 2.3}, this implies that $\Maps(X)\cong\Maps(Y)\times\Maps(C)$. By Lemma~\ref{lem: direct product of groups}, $\Maps(X)$ contains a dense conjugacy class if and only if both $\Maps(Y)$ and $\Maps(C)$ do. By Proposition~\ref{prop: Cantor tree}, $\Maps(C)$ does, and thus $\Maps(X)$ does if and only if $\Maps(Y)$ does.
\end{proof}

There are many graphs whose mapping class group contains a dense conjugacy class but such that the corresponding infinite-type surface mapping class group does not, as we next prove.

\Tdp*
\begin{proof}
    The graph $X$ must be homeomorphic to $\omega^\alpha+1$ for some countable ordinal $\alpha$ by \cite{MS}. $\Maps(\omega^\alpha+1)$ contains a dense conjugacy class by Proposition~\ref{prop: unique maximal end}, and so the statement is a direct application of Proposition~\ref{prop: trees and direct products}.
\end{proof}

For all graphs of the form $(\omega^\alpha+1)\vee C$, where $\alpha$ is a countable ordinal, the corresponding infinite-type surface, which can be formed as the boundary of a tubular neighborhood of the graph, contains multiple maximal ends. Thus, the mapping class of these surfaces do not contain dense conjugacy classes by Theorem~\ref{thm: Rokhlin property for surfaces}, despite their graph counterparts containing them.

\section{A word on $\Homeo(\partial X,\partial X_g)$}\label{sec: homeo end space}

In this final section, we discuss dense conjugacy classes in $\Homeo(\partial X,\partial X_g)$ and obtain results paralleling those for $\Maps(X)$. This section is motivated by the observation that $\Maps(o(1\vee C))$ does not have a dense conjugacy class whereas $\Homeo(\partial o(1\vee C),\partial o(1\vee C)_g)\cong\Maps(1\vee C)\cong\Maps(C)$ does. We will make use of the surjective continuous and open group homomorphism $\sigma\colon\Maps(X)\to\Homeo(\partial X,\partial X_g)$ from Proposition~\ref{AK-B 2.3}.

\begin{proposition}\label{prop: Rokhlin pushforward}
    Let $X$ be a locally finite graph $X$.
    \begin{enumerate}[(1)]
        \item If $g(X)$ is infinite, then $\Homeo(\partial X,\partial X_g)$ has a dense conjugacy class if $\Maps(X)$ does. 
        \item If $g(X)$ is finite, then $\Homeo(\partial X,\partial X_g)$ has a dense conjugacy class if and only if the mapping class group of a spanning tree of $X$ does.
    \end{enumerate}
\end{proposition}
\begin{proof}
    Suppose that $g(X)$ is infinite, and suppose that the conjugacy class of $[f]\in\Maps(X)$, which we call Conj$([f])$, is dense. The set $\sigma($Conj$([f]))$ is dense in $\Homeo(\partial X,\partial X_g)$, and for all $\sigma([g])\in$ Conj$(\sigma([f]))$, there exists $[h]\in\Maps(X)$ such that $[h][g][h]^{-1}=[f]$. Then $\sigma([h])\sigma([g])(\sigma([h]))^{-1}=\sigma([f])$, and so $\sigma($Conj$([f]))=$Conj$(\sigma([f]))$. Thus, the conjugacy class of $\sigma([f])$ is dense in $\Homeo(\partial X,\partial X_g)$. The statement about when $g(X)$ is finite follows immediately from Proposition~\ref{AK-B 2.3} applied to trees.
\end{proof}

Proposition~\ref{prop: Rokhlin pushforward} implies that for any locally finite graph $X$, if $\Maps(X)$ has a dense conjugacy class, so does $\Homeo(\partial X,\partial X_g)$. For example, if $X$ is a locally finite graph with a unique maximal end, genus zero or infinite, and self-similar end space, then $\Homeo(\partial X,\partial X_g)$ has a dense conjugacy class by Proposition~\ref{prop: unique maximal end}. This is also true if $X$ has finite genus, for then, $\Homeo(\partial X,\partial X_g)$ is isomorphic as topological groups to the mapping class of a spanning tree of $X$ by Proposition~\ref{AK-B 2.3}. The converse of Proposition~\ref{prop: Rokhlin pushforward} (1) is false by the motivating observation above, but the following partial converse holds.

\begin{proposition}\label{prop: open normal subgroup pushforward}
    If $X$ is a locally finite graph with infinite genus such that $\Maps(X)$ has a proper open normal subgroup $H$. If $\sigma(H)$ is proper, then $\Homeo(\partial X,\partial X_g)$ does not have a dense conjugacy class.
\end{proposition}
\begin{proof}
    By Proposition~\ref{AK-B 2.3}, $\sigma$ is open, and so $\sigma(H)$ must be open as well. To prove that $\sigma(H)$ is normal, let $\sigma([f])\in\Homeo(\partial X,\partial X_g)$ where $[f]\in\Maps(X)$. If $\sigma([h])\in\sigma(H)$, where $[h]\in H$, then $\sigma([f])\sigma([h])(\sigma([f]))^{-1}=\sigma([f][h][f]^{-1})\in\sigma(H)$, as $H$ is normal in $\Maps(X)$. By Lemma~\ref{lem: proper open normal subgrp}, $\Homeo(\partial X,\partial X_g)$ does not have a dense conjugacy class.
\end{proof}

The above two propositions imply the following corollary, which is an analogue of Theorem~\ref{thm: obstructions to dense conjugacy classes} for $\Homeo(\partial X,\partial X_g)$.

\begin{corollary}
    Let $X$ be a locally finite graph such that at least one of the following holds.
    \begin{enumerate}
        \item There exist two maximal ends with a gcd which is an end that is not of Cantor type; or
        \item There exists an end type $E(\mu)$ with $0<|E(\mu)|<\infty$.
    \end{enumerate}
    Then $\Homeo(\partial X,\partial X_g)$ does not have a dense conjugacy class.
\end{corollary}
\begin{proof}
    These follow immediately by Proposition~\ref{prop: open normal subgroup pushforward} by considering the kernel of maps constructed in Proposition~\ref{prop: flux map for ends} and Theorem~\ref{thm: finite end type}.
\end{proof}

We also obtain the following corollary, which is an analogue of Proposition~\ref{prop: babel}.

\begin{corollary}
    Let $X$ be a locally finite graph with stable maximal ends. Then $\Maps(X)$ does not have a dense conjugacy class if every maximal end dominating $\partial 1$ is of Cantor type.
\end{corollary}
\begin{proof}
    The same proof as in Proposition~\ref{prop: babel} shows that $\Homeo(\partial X,\partial X_g)$ does not have the JEP.
\end{proof}

\bibliography{bibliography}
\bibliographystyle{alpha}

\end{document}